\documentclass[oneside,a4paper,11pt,reqno]{amsart}

\usepackage[T1]{fontenc}

\usepackage[a4paper, scale={0.72,0.74}, marginratio={1:1}, footskip=7mm, headsep=10mm]{geometry}

\usepackage{verbatim}
\usepackage{amsmath} \allowdisplaybreaks
\usepackage{amsfonts}
\usepackage{color}
\usepackage{makecell}
\usepackage{bm}
\usepackage{mathabx}
\usepackage[title]{appendix}

\allowdisplaybreaks
\usepackage{tikz}
\usetikzlibrary{topaths,fit,backgrounds,shapes.geometric}
\usetikzlibrary{graphs}
\usetikzlibrary{calc}
\usepackage{multirow}
\usepackage{enumerate}

\usepackage{amsmath, amssymb, amsthm}
\usepackage{pdflscape}
\usepackage[table]{xcolor}
\usepackage{tikz}
\usetikzlibrary{shapes, backgrounds, positioning}
\usepackage{pgfplots}
\usetikzlibrary{graphs}
 \usepackage{enumerate}
 \usepackage{cases}
 \usepackage{graphicx}
 \usepackage{xcolor}
 \usepackage{geometry}
 \usepackage{indentfirst}
 \usepackage{amsfonts}
 \usepackage{mathrsfs}
 \usepackage{hyperref}
 \usepackage{multicol}
 \usepackage{url}
\usepackage{float}
 \usepackage{color}
 \usepackage{hyperref}
 \usepackage{verbatim}
\usepackage{multirow}
 \usepackage{stackengine}
\usepackage{bbold}
\usepackage{float}

 \usepackage{fancyhdr}
 
 \usepackage{caption}
 \usepackage{mathtools}

\usepackage{cases}
\usepackage{graphicx}
\usepackage{xcolor}
\usepackage{geometry}
\usepackage{indentfirst}
\usepackage{amsfonts}
\usepackage{mathrsfs}
\usepackage{multicol}
\usepackage{url}
\usepackage{float}
\usepackage{color}
\usepackage{verbatim}
\usepackage{stackengine}
\stackMath
\usepackage[normalem]{ulem}
\usepackage{cancel}

\usepackage{fancyhdr}
\usepackage{caption}
\usepackage{mathtools}
\usepackage{dsfont}
\usepackage{graphicx}
\usepackage{esdiff}
\usepackage{physics}
\usepackage{subcaption}

\usepackage{nicematrix}

\definecolor{citeorange}{RGB}{200,120,0}
\definecolor{citegreen}{RGB}{0,128,128}

\usepackage{hyperref}
\hypersetup{
  colorlinks=true,
  linkcolor=citegreen,
  citecolor=citeorange,
  urlcolor=
}

\usepackage{amssymb}

\newtheorem{theorem}{Theorem}
 \newtheorem{lemma} {Lemma}[section]
 \newtheorem{prop}[lemma]{Proposition}
\newtheorem{corollary}[lemma]{Corollary}
 \newtheorem{definition}[lemma]{Definition}
 \newtheorem{oss}[lemma]{Remark}
\numberwithin{equation}{section}
\definecolor{color1}{rgb}{0.7, 0.85, 1.0} 
\definecolor{color2}{rgb}{0.4, 0.6, 0.9}  
\definecolor{color3}{rgb}{0.2, 0.4, 0.8}  
\definecolor{giallograno}{rgb}{1.0, 0.75,0.25} 
\definecolor{tealmid}{RGB}{0, 128, 128}
\definecolor{teallight}{RGB}{170, 210, 210}
\definecolor{tealdeep}{RGB}{0, 80, 80}

\newenvironment{myenumerate}{
\renewcommand{\theenumi}{(\roman{enumi})}
\renewcommand{\labelenumi}{\theenumi}
\begin{list}{\labelenumi}
	{
	\setlength{\itemsep}{0.4em}
	\setlength{\topsep}{0.5em}
	\setlength\leftmargin{2.45em}
	\setlength\labelwidth{2.05em}
	\setlength{\labelsep}{0.4em}
	\usecounter{enumi}
	}
	}
{\end{list}
}
\renewenvironment{enumerate}{
\begin{myenumerate}}
{\end{myenumerate}}

\newenvironment{myitemize}{
\begin{list}{$\bullet$}
 	{
	\setlength{\itemsep}{0.4em}
	\setlength{\topsep}{0.5em}
	\setlength\leftmargin{2.45em}
	\setlength\labelwidth{2.05em}
	\setlength{\labelsep}{0.4em}
	}
	}
{\end{list}}
\renewenvironment{itemize}{
\begin{myitemize}}
{\end{myitemize}}

\def \E {\mathbb{E}}
\def \N  {\mathbb{N}} 
\def \P {\mathbb{P}}
\def \R  {\mathbb{R}} 
 
\def \ind {\ind}

\def \WW {{\mathcal W}}
\def \HH {{\mathcal H}} 
\def \CC {{\mathcal C}}

\def \ind {\hbox{\tt 1\hskip-.33em l}}

\newcommand{\Cmax}{\mathcal{C}_{max}}

\renewcommand{\epsilon}{\varepsilon}
\newcommand{\eps}{\varepsilon}

\newcommand{\Pp}{\mathbb{P}}
\newcommand{\Pb}{\bm{\mathrm{P}}}
\newcommand{\Eb}{\bm{\mathrm{E}}}
\newcommand{\Varb}{\bm{\mathrm{Var}}}
\newcommand{\Gg}{\mathbb{G}}

\newcommand{\Ep}{\mathbb{E}}

\newcommand{\cc}{\mathrm{c}}

\newcounter{para}

\usepackage[
    backend=biber,
    style=numeric,        
    sorting=nyt,          
    giveninits=true,      
    maxbibnames=99        
]{biblatex}

\DeclareNameAlias{default}{family-given}

\DeclareFieldFormat[article, inbook, incollection, inproceedings, patent, thesis, unpublished]{title}{#1}

\renewbibmacro{in:}{%
  \ifentrytype{article}{}{\printtext{\bibstring{in}\intitleunit}}}

\appto{\bibsetup}{\sloppy}
\begin{document}

\title[Inhomogeneous Random Graphs]{Inhomogeneous Random Graphs: How Heavy Tails Reshape Connectivity}

\author[L.~Andreis]{Luisa Andreis} 
\address{Luisa Andreis, Università degli Studi di Torino, Palazzo Campana, Via Carlo Alberto 10, 10123, Turin, Italy}
\email{luisa.andreis@unito.it}

\author[L.~Avena]{Luca Avena} 
\address{Luca Avena, Dipartimento di Matematica e Informatica ``Ulisse Dini'', Universit\`a degli Studi di Firenze, Viale Morgagni 67/A, 50134, Florence, Italy}
\email{luca.avena@unifi.it}

\author[R.~Hazra]{RAJAT SUBHRA HAZRA}
\address{Rajat Subhra Hazra, Mathematical Institute, Leiden University, Einsteinweg 55, 2333
CC Leiden, The Netherlands}
\email{r.s.hazra@math.leidenuniv.nl}

\author[E.~Matteini]{Elena Matteini}
\address{Elena Matteini, Dipartimento di Matematica e Informatica ``Ulisse Dini'', Universit\`a degli Studi di Firenze, Viale Morgagni 67/A, 50134, Florence, Italy}
\email{elena.matteini@unifi.it}

\begin{abstract}
We study inhomogeneous random graphs on $n$ vertices with independent Pareto weights of parameter $\alpha>0$ and connection probabilities $\min (\varepsilon_n W_iW_j, 1)$. Our main focus is the infinite-mean regime $0<\alpha<1$. Under the scaling $n\varepsilon_n^\alpha=\lambda$, with $\lambda>0$ fixed, the global architecture of the graph is entirely governed by the Poisson point process $\Pi$ of rescaled extreme weights. In particular, the size of the largest component, divided by $n$, converges to a non-degenerate random variable expressed as an explicit functional of $\Pi$. Through the point process $\Pi$, we also characterize the exact limiting connectivity probability, uncovering a non-sharp connectivity threshold.
The second largest component and the diameter of the largest component are bounded with high probability. The distance between two uniformly chosen vertices in the largest component is at most three with high probability. We also prove convergence of the point process of rescaled degrees and identify the limiting distribution of the maximum degree. For completeness, we obtain expected degree asymptotics, typical degree limits and connectivity regimes for all $\alpha>0$. Due to the failure of classical branching process approximations in the infinite-mean regime, our approach relies fundamentally on analyzing the extreme weight environment $\Pi$.

\end{abstract}





\maketitle
\tableofcontents


\section{Introduction}\label{model}

Inhomogeneous random graphs provide a rich and treatable framework for modeling networks with heterogeneous connectivity. Since the seminal work~\cite{Bollobas2007Phase}, they have been extensively studied under several different assumptions, see~\cite{VanDerHofstad2024RandomV2} and references therein for a clear introduction to the model. We are interested in studying the setting where edges are independently drawn and $i,j\in [n]$ are connected with probability
\begin{equation*}
    p_{ij}=\min\left(\frac{k_n(W_i,W_j)}{n}, 1\right),
\end{equation*}
where $k_n$ is a symmetric, non-negative kernel (possibly depending on $n$) and $\{W_i\}_{i\in [n]}$ are i.i.d.\ random weights. In particular, we focus on the  rank-one kernel $k_n(x,y)=n\varepsilon xy$, where the parameter $\eps$ serves as a convenient scaling parameter. 
A closely related model has been studied as well in~\cite{avena2026inhomogeneous, catanzaro2026clustering}, see Section~\ref{subsec:literature} for comparisons. 
A key assumption underlying many classical results as in~\cite{Bollobas2007Phase, VanDerHofstad2024RandomV2}—including those on phase transitions and local weak convergence—is that the kernel $k_n$ is \emph{graphical} with respect to the weight distribution, which in particular requires suitable integrability conditions (see \cite[Definition 2.7]{Bollobas2007Phase}). These conditions ensure that the graph is governed by the contribution of typical vertices and enable the use of branching process approximations. Far less is known in other regimes.

In this paper, we consider the paradigmatic case in which the weights follow a Pareto distribution with parameter \(\alpha>0\), which determines the heaviness of the tail: the smaller \(\alpha\), the heavier the tail.
When $\alpha \leq 1$, and the mean weight is infinite, regularity conditions in the model break down and the graphical assumption ceases to be true. As a consequence, the graph is no longer controlled by typical vertices. For this reason, while the regime $\alpha>1$ has been extensively studied—revealing a rich structure depending on whether the variance is finite ($\alpha>2$) or infinite ($1<\alpha\leq 2$)—the infinite-mean case $\alpha \leq 1$ is more delicate. In this regime, classical techniques based on local weak convergence and branching process approximations lose their effectiveness, and new phenomena emerge. In this work, we identify a fundamentally different behavior in the regime $\alpha<1$: the graph lacks self-averaging and its macroscopic structure remains sensitive to the realization of  extreme weights. 
This affects the connectivity, which, in the infinite-mean regime, is shaped by the emergence of extreme vertices, marking a clear departure from classical results.

In detail: in Theorem \ref{thm:expecteddegree}, we characterize the expected degree of a typical vertex and identify the preferential scalings required to ensure the convergence in distribution of the degree. In Theorem \ref{thm:dust and connectivity threshold}, we investigate the connectivity of the graph and its geometric properties, with particular emphasis on the role of hubs in determining the network's diameter. It identifies as thresholds for the connectivity $\varepsilon=\Theta(\ln(n)/n)$ for $\alpha>1$ and $\varepsilon^\alpha=\Theta(n^{-1})$ for $\alpha \leq 1$. Figure \ref{fig:regimi_connettivita} provides a graphical representation of the main results established in Theorem \ref{thm:expecteddegree} and Theorem \ref{thm:dust and connectivity threshold}. 

For $\alpha > 1$, our model is consistent with the classical Erd\H{o}s--R\'enyi-like picture, see comparison with~\cite{Devroye2014Connectivity} in Section~\ref{subsec:phenomena_comparison}. In terms of connectivity this means that the threshold above which the graph is connected with high probability is $\varepsilon = \Theta(\log n / n)$, in this regime  even low-weight vertices are not isolated. In contrast, for $\alpha < 1$, global connectivity is triggered as soon as at least one vertex weights exceed the threshold $\varepsilon^{-1}$—forming what we call the Core of the graph. This occurs when $\varepsilon^\alpha n \to \infty$, effectively removing the need for any logarithmic correction. In this regime, the scaling $\varepsilon^\alpha = \lambda / n$ identifies a "critical window" rather than a sharp deterministic threshold as in \cite{Devroye2014Connectivity}.

For the above reasons, we later focus on the infinite-mean regime under the relevant scaling $\varepsilon^\alpha = \lambda/n$. 
Theorem \ref{thm:giantsize} establishes that the size of the largest component, rescaled by $n$, converges in distribution to a non-degenerate random variable. This behavior reflects the dominant role of extreme weights, which induce persistent fluctuations and mark a fundamental departure from classical finite-mean theory, where the size of the giant component is typically deterministic.
Theorem \ref{thm:nosharp} shows that, in contrast with classical results, in the infinite mean regime there is not a sharp threshold for connectivity, meaning that the graph has both a positive probability to be connected or to contain an isolated vertex.
Theorem \ref{thm:2nd_component} shows that the second largest component, if it exists, has at most constant size. 
Proposition \ref{thm:extremeultrasmallworld} describes the extreme ultra-small-world nature of the infinite-mean regime, characterizing both the diameter and the typical distance within the giant component of $\mathbb{G}_n(\alpha,\varepsilon)$. Finally, Theorem \ref{thm:maxdegree} characterizes the law of the maximum degree.

The topology of the graph as well, under the relevant scalings ($\varepsilon^\alpha = \lambda/n$ for $\alpha < 1$ and $\varepsilon = \lambda/n$ for $\alpha > 1$), reveals fundamentally different scenarios  as $\alpha$ varies. 
Concerning the giant component, in the finite-variance regime ($\alpha > 2$), the classical Erd\H{o}s-R\'enyi behavior is recovered and the graph exhibits a giant linear component if $\lambda > \mathbb{E}[W^2]^{-1}$ with typical distances of order $\log n$ \cites[Corollary 3.2]{Bollobas2007Phase}[Theorem 6.1]{VanDerHofstad2024RandomV2}. 
As $\alpha$ decreases into the interval $(1, 2)$, a giant linear component is present for each $\lambda$ and the graph enters the ultra-small-world regime, characterized by typical distances of order $\log \log n$ \cites[Corollary 3.3]{Bollobas2007Phase}[Theorem 6.3]{VanDerHofstad2024RandomV2}{van2017scale}. Conversely, in the infinite-mean regime ($\alpha < 1$), we observe a structural collapse of distances, and an extreme-ultra-small-world behavior. In this setting, the typical distance between two vertices chosen uniformly at random in the giant component no longer depends on $n$, but remains bounded by a constant equal to 3. 

Another fundamental distinction between these regimes lies in the concentration of the giant component size. For $\alpha > 1$, $|\mathcal{C}_{max}|/n$ converges in probability to a deterministic value $\zeta \in [0,1)$, representing the survival probability of the associated limiting branching process \cite[Theorem 3.1]{Bollobas2007Phase}. 
Conversely, in the infinite-mean regime ($\alpha < 1$) $\mathbb{G}_n(\alpha,\varepsilon)$ exhibits a novel \emph{lack of self-averaging}, as was already noticed in~\cite{catanzaro2026clustering} for a closely related model. We establish, in Theorem \ref{thm:giantsize}, that under the scaling $\varepsilon^\alpha = \lambda/n$ the ratio $|\mathcal{C}_{max}|/n$ converges in distribution to a non-degenerate random variable. This limiting variable is functionally dependent on the underlying Poisson Point Process of the rescaled weights (see Section \ref{sec:giantsize}), reflecting how the global structure is entirely governed by extreme vertices rather than typical ones.

The paper is organized as follows. The remainder of the introduction is divided into three parts. Section~\ref{sec:modeldescription} describes the model and introduces the relevant notation, Section~\ref{sec:mainresults} states our main results, and Section~\ref{subsec:literature} compares our model and results with the existing literature.
The proofs are presented in the subsequent sections. Theorem~\ref{thm:dust and connectivity threshold} is proved in Section~\ref{connectivity}, 
Theorem~\ref{thm:giantsize} in Section~\ref{sec:giantsize},
Theorem~\ref{thm:nosharp} in Section~\ref{sec:nosharp}, Theorem~\ref{thm:2nd_component} and Proposition~\ref{thm:extremeultrasmallworld} in Section~\ref{sec:2nd_component_distance}, and Theorem~\ref{thm:maxdegree} in Section~\ref{sec:maxdegree}. Finally, Appendix~\ref{appendice1} contains the proof of Theorem~\ref{thm:expecteddegree}, together with several useful auxiliary results.

\subsection{Model definition and Notations}\label{sec:modeldescription}

We consider a class of inhomogeneous Erd\H{o}s--R\'enyi-type random graphs on
\(n\) vertices, with vertex set
\[
V=[n]:=\{1,2,\ldots,n\}.
\]
To each vertex \(i\in[n]\), we assign an independent weight, or fitness
variable, \(W_i\). The random variables \((W_i)_{i\in[n]}\) are i.i.d. with
Pareto distribution of shape parameter \(\alpha>0\) and scale parameter \(1\),
that is,
\begin{equation}\label{paretodistribution}
  \Pb(W>w)=w^{-\alpha}, \qquad w\geq 1.
\end{equation}

In particular,
\[
\Eb[W]=\frac{\alpha}{\alpha-1}, \qquad \alpha>1,
\]
and
\[
\Varb[W]=\frac{\alpha}{(\alpha-2)(\alpha-1)^2}, \qquad \alpha>2,
\]
while the corresponding moments are infinite outside these parameter ranges.

Conditionally on the weights, edges between distinct vertices are included
independently. More precisely, for \(i\neq j\),
\begin{equation}\label{defofp}
p_{ij}
:=
\mathbb{P}\bigl(i\leftrightarrow j \mid W_i,W_j\bigr)
=
\min(\varepsilon W_iW_j,1),
\end{equation}
where \(\varepsilon=\varepsilon_n>0\) is a parameter controlling the overall
edge density. Self-loops and multiple edges are excluded, so the resulting
graph is simple.

We denote this random graph by \(\mathbb{G}_n(\alpha,\varepsilon)\), and write
\(\Pp\) and \(\Ep\) for its law and the corresponding expectation,
respectively. Although these quantities depend on \(n\), we suppress this
dependence from the notation whenever no ambiguity arises.

The structure of $\mathbb{G}_n(\alpha, \varepsilon)$ is determined by the interplay between the decay of $\varepsilon$ and the heavy-tailed nature of the weights, the former makes the graph sparser and sparser, while the latter promotes the emergence of high-degree hubs, especially for small values of $\alpha$. The graph is characterised by two distinct levels of randomness: the stochastic realization of the fitness variables $W_i$, and the subsequent Bernoulli sampling of edges conditioned on those weights. As we shall see, the transition between finite and infinite-mean regimes for the weights ($\alpha=1$) and the finiteness of the variance ($\alpha=2$) represent critical thresholds for the network.

To fully understand the role of the $\alpha$ parameter, we propose the simulations shown in Figures \ref{fig:simulazioninostre1} and \ref{fig:simulazioninostre2}. When the Pareto distribution \eqref{paretodistribution} has infinite mean (i.e.~~when $\alpha\in(0,1]$) the model favors the presence of very heavy vertices, which act as central poles of attraction. However, as soon as $\alpha >1$, where the Pareto distribution has a finite mean, the  model changes behavior. We observe more homogeneity in the weights (we are under the assumptions for the law of large numbers to hold), although some variability in vertex weights remains. When $\alpha>2$ the Pareto distribution has a finite variance and the vertex weights become increasingly concentrated around the mean value $\alpha/(\alpha-1)$ as $\alpha$ grows. Consequently, the variability in weights decreases and the model gets closer to the  Erd\H{o}s--R\'enyi $\mathcal{G}(n,\varepsilon)$, as $\alpha$ increases. The specific values $\alpha=1$ and $\alpha=2$ show some criticality and some more specific features.

\begin{figure}                       
\centering
\includegraphics[scale=0.30]{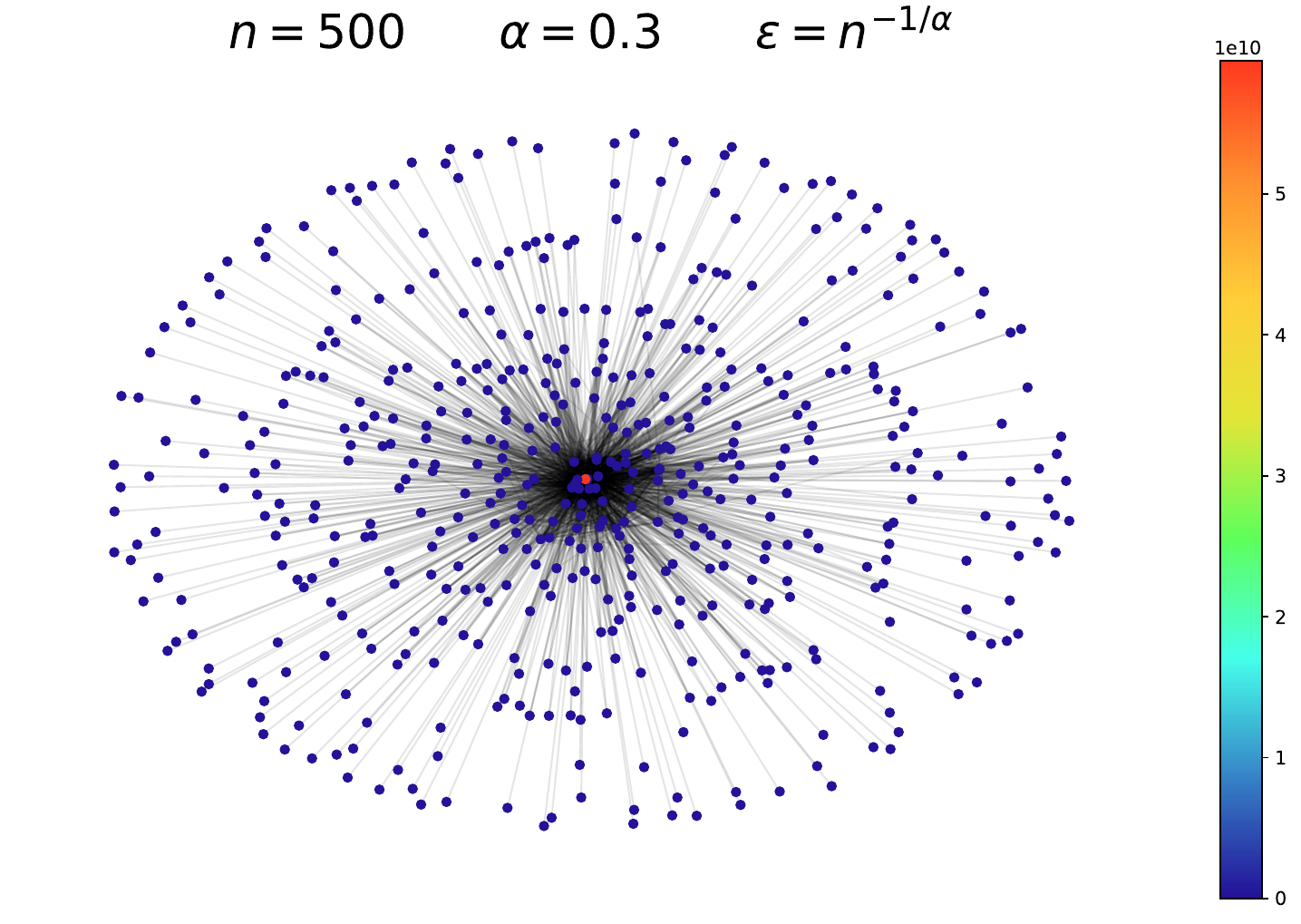}
\includegraphics[scale=0.30]{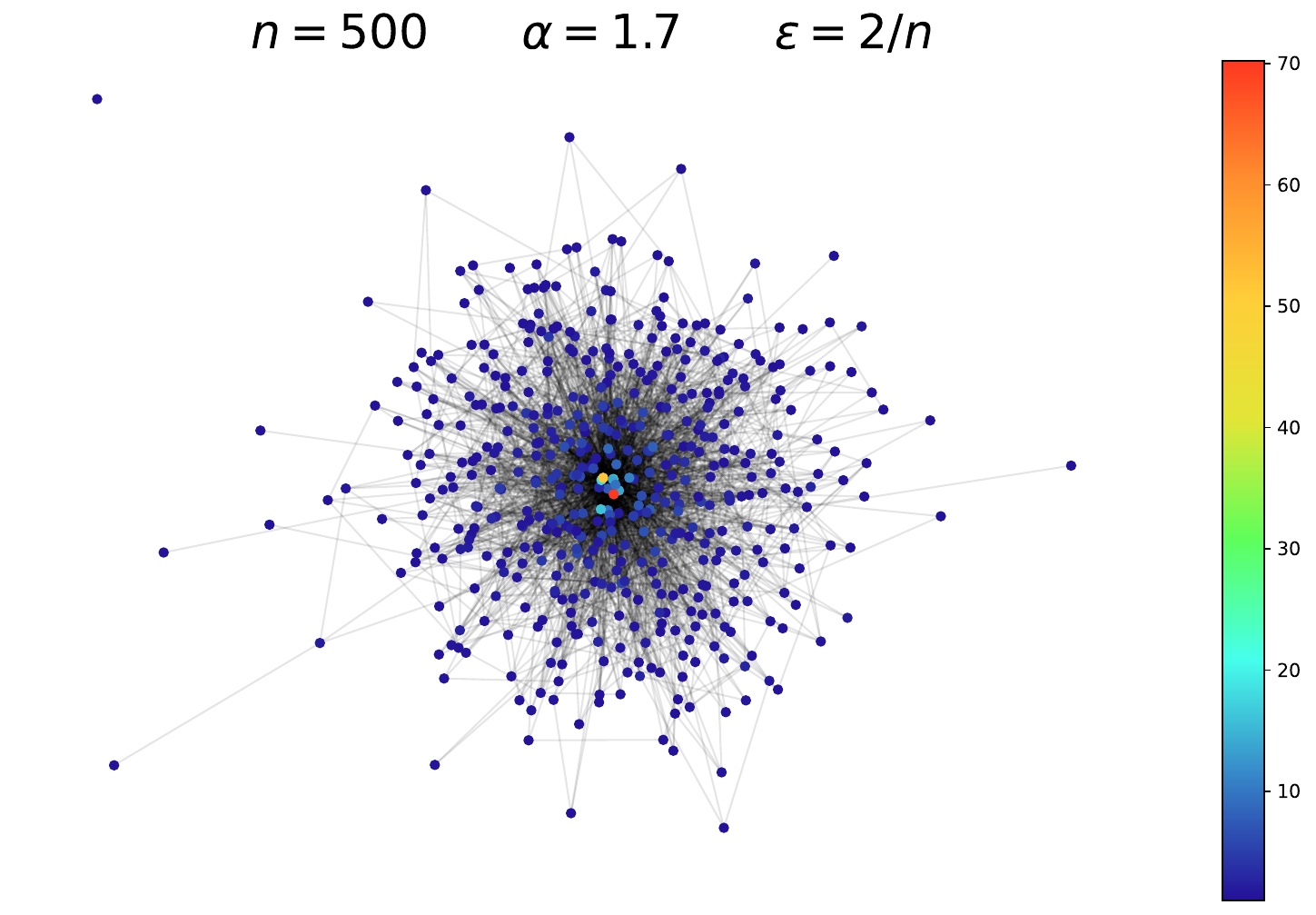}
\caption{In this simulation we work on $n=500$ vertices. We colored the vertices according to their weight. In the first image we used $\alpha=0.3$ and $\varepsilon=n^{-1/\alpha}$. In the second image we used $\alpha=1.7$ and $\varepsilon=2/n$. 
}
\label{fig:simulazioninostre1}
\includegraphics[scale=0.30]{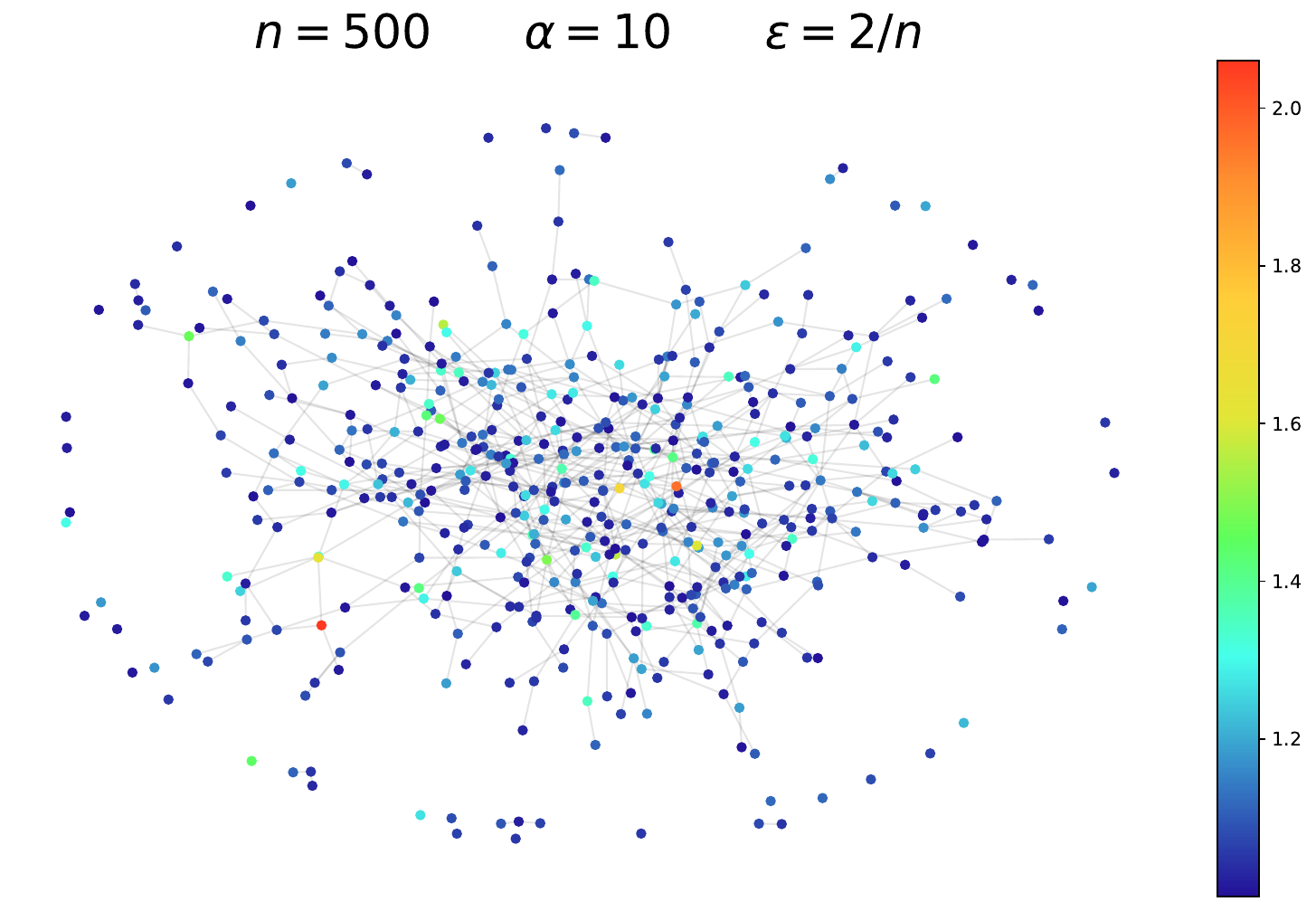}
\includegraphics[scale=0.30]{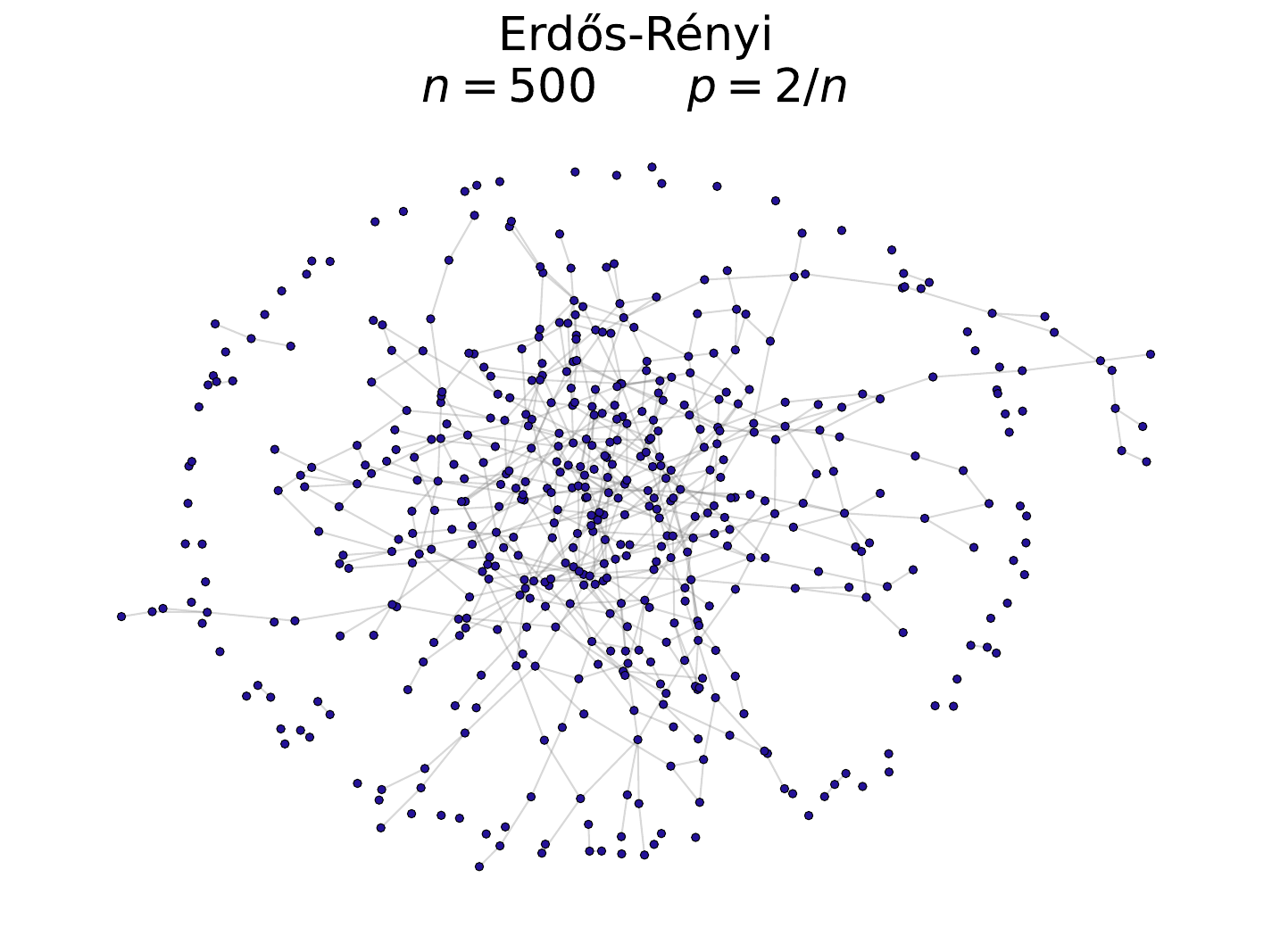}
\caption{We work in both case on $n=500$ vertices. The first is $\mathbb{G}_n(\alpha,\varepsilon)$ with $\alpha=10$ and $\varepsilon=2/n$. We colored the vertices according to their weight. The second is the  Erd\H{o}s--R\'enyi $\mathcal{G}(n,p)$ with $p=2/n$. 
}
\label{fig:simulazioninostre2}
\end{figure}

Regarding notation, whenever we are not referring specifically to the law of the graph \(\Gg_n(\alpha,\eps)\), we use \(\Pb\) and \(\Eb\) to denote a generic probability measure and the corresponding expectation. This convention is used mainly when describing limiting or universal objects. In particular, under \(\Pb\), \(W\sim\mathrm{Pareto}(\alpha)\) always denotes a representative Pareto random variable with shape parameter \(\alpha\), whose dependence on \(\alpha\) is omitted whenever clear from the context.
Convergence in distribution and convergence in probability are denoted by \(\xrightarrow{d}\) and \(\xrightarrow{\P}\), respectively. For a random variable \(X\), conditional expectation given \(X\) is denoted by \(\mathbb{E}_X[\cdot]:=\mathbb{E}[\cdot\mid X]\).
A sequence of events is said to occur \emph{with high probability} (w.h.p.) if its probability converges to \(1\) as \(n\to\infty\).
For \(i,j\in[n]\), we denote by \(\{i\longleftrightarrow j\}\) the event that \(i\) and \(j\) are directly connected by an edge. More generally, for \(A,B\subseteq[n]\), the event \(\{A\longleftrightarrow B\}\) indicates that there exists at least one edge joining a vertex in \(A\) to a vertex in \(B\). The cardinality of a set \(A\) is denoted by \(|A|\).
Finally, for functions \(f\) and \(g\), as \(n\to\infty\), we write \(f(n)\simeq g(n)\) if \(f(n)/g(n)\to1\), \(f(n)=o(g(n))\) if \(f(n)/g(n)\to0\), and \(f(n)=O(g(n))\) if there exists a constant \(C>0\) such that \(|f(n)|\leq C|g(n)|\) for all sufficiently large \(n\). We write \(f(n)=\Theta(g(n))\) if there exist constants \(0<c_1\leq c_2<\infty\) such that
\[
c_1\leq \liminf_{n\to\infty}\frac{f(n)}{g(n)}
\leq \limsup_{n\to\infty}\frac{f(n)}{g(n)}
\leq c_2.
\]

For an interval $I \subseteq (0,\infty]$, we denote by $\mathcal{M}_p(I)$ the space of point measures on $I$ equipped with the vague topology. If $M_n$ and $M$ are random variables in $\mathcal{M}_p(I)$ (i.e., random point measures in $\mathcal{M}_p(I)$), we write $M_n \Rightarrow M$ to denote convergence in distribution with respect to the vague topology.

We denote the degree of vertex $1$ and the number of isolated vertices as 
\begin{equation}\label{eq:degree_i}
D_n(1):=\sum_{j\neq 1} \ind(1 \longleftrightarrow j), \qquad    \mathcal{I}:=\sum_{i}\ind(D_n(i)=0).
\end{equation} 
By exchangeability $D_n(i)$ has the same distribution of $D_n(1)$, for any $i\in [n]$.
Sometimes we use the notation  $a_{ij}:=\ind(i \longleftrightarrow j)$, for $i\neq j$ to denote the  random variable that indicates the presence of the edge $\{i,j\}$. 

In what follows we define some sets of vertices with special features, since they will play a central role in all the paper.

\begin{definition}\label{core}
We define the \textbf{Core} of $\mathbb{G}_n(\alpha,\varepsilon)$ as 

\begin{equation*}
\HH_{\varepsilon}=\{i\in[n]: W_i\geq \varepsilon^{-1}\}.
\end{equation*}
We call \textbf{hub} an element of $\HH_{\varepsilon}$.
\end{definition}

\begin{definition}\label{core_delta}
Let $\delta >0$, we call \textbf{$\delta$--Core} the set of \textbf{weak hubs}, i.e.~
\begin{equation*}
\WW_\delta=\{i \in [n]: W_i > n^{\frac{1-\delta}{\alpha}}\}.
\end{equation*}
\end{definition}

\begin{definition}\label{Ndelta}
We define the \textbf{closed neighborhood} of $\WW_\delta$, denoted by $\widebar{\WW}_\delta$, as the set of vertices that are either part of the $\delta$--Core or directly connected to it:
\begin{equation*}
\widebar{\WW}_\delta := \{ i \in [n] : i \in \WW_\delta \text{ or } \exists j \in \WW_\delta \text{ s.t. } i \leftrightarrow j \}.
\end{equation*}
Equivalently, $\widebar{\WW}_\delta$ is the union of the weak hubs and their first neighbors. 
\end{definition}
\noindent Clearly, it holds $\HH_{\varepsilon} \subseteq \WW_\delta\subseteq \widebar{\WW}_\delta$. 
Connected components of the graph are denoted by $\{\CC_i\}_{i\geq 1}$ and we write $\{\CC_{(i)}\}_{i\geq 1}$ when we order them in decreasing order, breaking ties arbitrarily. $\CC_{(1)}$ is sometimes indicated with $\CC_{\max}$ as well. 
We  denote the diameter of a connected subgraph $S \subseteq \mathbb{G}_n(\alpha, \varepsilon)$ as $\text{diam}(S)$, defined as the maximum graph distance between any pair of vertices in $S$. In particular, we will investigate the diameter of the entire graph, $\text{diam}(\mathbb{G}_n(\alpha, \varepsilon))$, when it is fully connected, and the diameter of the largest component, $\text{diam}(\CC_{\max})$.

\subsection{Main results}\label{sec:mainresults}

Here we state the main results of the paper.  
Theorem \ref{thm:expecteddegree} and Theorem~\ref{thm:dust and connectivity threshold}  identify the relevant regimes of the scaling parameter $\eps$,  depending on $\alpha$ and $n$,  under which convergence of typical degree or, respectively, connectivity occurs. Although some of these results are somehow straightforward consequences, or mild extensions, of known results on the local weak limit of inhomogeneous random graphs (especially in the case of $\alpha>1$), we chose to present the general scaling to give a complete and exhaustive picture of the model. In this respect, see as well Figure~\ref{fig:regimi_connettivita} which illustrates how the relevant scaling regimes changes for different ranges of $\alpha$. In particular Theorem~\ref{thm:expecteddegree} appears to us a mild, yet interesting, extension of known results. For this reason its proof is contained, for the sake of completeness, in the Appendix~\ref{appendice1}.

From Theorem~\ref{thm:nosharp} on, we focus on one scale: the case of weights with infinite mean, \(\alpha<1\). We study in detail the connectivity window, the geometry of the graph in this regime, with particular emphasis on the role of hubs in determining the network's diameter. 

\subsubsection{Cross-parameters picture: how degree and connectivity change when $\alpha$ and $\eps$ vary}\label{subsec:general}
To start, we recall that a random variable $X$ has a mixed Poisson distribution with \emph{mixing distribution} $F$ when, for every $k \in \mathbb{N}_0$,
\begin{equation*}\label{mixPoisson}
\Pb(X=k)=\Eb\biggl[e^{-Z}\frac{Z^k}{k!}	\biggr],
\end{equation*}
where $Z$ is a random variable with distribution function $F$. By abuse of terminology, we say that $X$ has a mixed Poisson distribution with mixing parameter $Z$, and we mean that, conditioning on the event $\{Z=z\}$ for any $z\in\R_+$, $X$ is a Poisson random variable with a known parameter $z$. We have the following result on the degree of a fixed vertex, say vertex $1$ (by exchangeability they all have the same distribution). We denote it with $D_n(1)$, recall its definition in~\eqref{eq:degree_i}.

\begin{theorem}[Expected degree and asymptotic degree distribution]\label{thm:expecteddegree}
Consider the graph $\mathbb{G}_n(\alpha,\varepsilon)$ and let $D_n(1)$ be the degree of vertex $1$. Let $\varepsilon \to 0^+$ as $n\to+\infty$. It holds:
\begin{equation}\label{eq:mean_degree1}
 \mathbb{E}[D_n(1)] \simeq 
\begin{cases}
-\frac{n-1}{1-\alpha}\varepsilon^{\alpha}\ln{\varepsilon^{\alpha}}&\alpha\in(0,1),\\
\frac{n-1}{2}\varepsilon\ln^2{\varepsilon}&\alpha=1,\\
(n-1) \bigl(\frac{\alpha}{\alpha-1}\bigr)^2\varepsilon &\alpha>1.
\end{cases}
\end{equation}

Moreover, for $\lambda \in (0,\infty)$, as $n\to \infty$, the degree $D_n(1)$ converges in distribution to a mixed Poisson random variable $D_\infty$ with mixing parameter $\Lambda$, given by:
\begin{equation*}
\Lambda = 
\begin{cases} 
\frac{\lambda}{1-\alpha}W^\alpha  & \quad\text{if } \alpha \in (0,1) \text{ and } \varepsilon^\alpha = \lambda/n, \\
\lambda W  &\quad\text{if } \alpha = 1 \text{ and } -\varepsilon \ln \varepsilon = \lambda/n, \\
\lambda \frac{\alpha}{\alpha-1} W  &\quad\text{if } \alpha > 1 \text{ and } \varepsilon = \lambda/n,
\end{cases}
\end{equation*}
where $W\sim$ Pareto$(\alpha)$.

\end{theorem}

Notice that, under these specific choices of $\varepsilon$, the tail of $D_\infty$ satisfies, for large $k$,
\begin{equation*}
\Pb(D_\infty\geq k)\simeq 
\begin{cases}
\frac{\lambda}{1-\alpha}k^{-1} & \text{if } \alpha \in (0,1),\\
\lambda k^{-1} & \text{if } \alpha=1,\\
\lambda^\alpha\bigl(\frac{\alpha}{\alpha-1}\bigr)^\alpha k^{-\alpha} & \text{if } \alpha >1.
\end{cases}
\end{equation*}
These choices of $\varepsilon$ (up to a multiplication by a constant value) are the only ones that allow $D_n(1)$ to converge in distribution to a non-trivial random variable. We can see similar results in \cite{avena2026inhomogeneous}, \cite{Britton2006Generating} and \cite{VanDerHofstad2016RandomV1} for related models both in the $\alpha<1$ and $\alpha>1$ case. 
The theorem outlines a phase transition as soon as the Pareto weights have a finite mean. We see that while for $\alpha \in(0,1)$ the choice of the correct scaling depends on $\alpha$, as soon as $\alpha>1$, $\varepsilon_n$ stabilizes on $n^{-1}$. The case $\alpha=1$ represents a singular situation in which there is a logarithmic correction. We observe that if $\varepsilon_n\ln(\varepsilon_n)=-n^{-1}$, then $\varepsilon_n=\exp(L(-n^{-1}))$ where $L(z)$ is the Lambert function, i.e~the function that satisfies $z=L(z)e^{L(z)}$. In particular $\varepsilon_n=o(n^{-1})$. With these choices of the scaling $\varepsilon_n$, if $\alpha \in (0,1]$ the expected degree diverges. However, as soon as $\alpha>1$ (and thus $\mathbb{E}[W]<\infty$), the expected degree converges to $\lambda\mathbb{E}[W]^2$. The final part of the result specifies that the tails of $D_\infty$ decay polynomially, i.e.~~$\mathbb{G}_n(\alpha,\varepsilon)$ is a scale-free random graph. As it is the most computational and least informative of our proofs, the proof of Theorem \ref{thm:expecteddegree} is deferred to the Appendix, in Section \ref{degreeofavertex}.

Our second main result concerns the connectivity threshold of the graph $\mathbb{G}_n(\alpha, \varepsilon)$. 

\begin{theorem}[Connectivity regimes]\label{thm:dust and connectivity threshold}
Consider the graph $\mathbb{G}_n(\alpha,\varepsilon)$ with $n\to \infty$. We have the following:
\begin{itemize}

\item $\alpha \in (0,1]$
\begin{align*}
&\text{ if }\quad\varepsilon^\alpha n\to0\quad\text{then}\quad
\P(\mathcal{I}\geq 1)\to 1,\\
&\text{ if }\quad\varepsilon^\alpha n\to \infty\quad\text{then}\quad \mathbb{P}(\text{diam}(\mathbb{G}_n(\alpha,\varepsilon))\leq2)\to1.
\end{align*}
\item $\alpha>1$
\begin{align*}
&\text{ if }\quad\frac{\varepsilon n}{\ln{n}}\to0\quad\text{then}\quad
\P(\mathcal{I}\geq 1)\to 1,\\
&\text{ if }\quad\frac{\varepsilon n}{\ln{n}}\to \infty\quad\text{then}\quad 
\mathbb{P}(\mathbb{G}_n(\alpha,\varepsilon) \text{ is connected})\to 1,\\
&\text{ if }\quad\varepsilon^\alpha n\to \infty\quad\text{then}\quad \mathbb{P}(\text{diam}(\mathbb{G}_n(\alpha,\varepsilon))\leq2)\to1.
\end{align*}
\end{itemize}
\end{theorem}

This result establishes that the threshold for the almost sure appearance of isolated vertices coincides with the connectivity threshold, mirroring the Erd\H{o}s–R\'enyi model. Nevertheless, the proof reveals that the model behaves quite differently depending on the regime, resulting in distinct scenarios. In the subcritical regime $\alpha \leq 1$, the ``heavy'' vertices play a central role: behaving as large hubs of attraction, they are able to ensure connectivity as soon as $\varepsilon^\alpha n \to \infty$. Moreover, by passing through such vertices, distances in the graph can be substantially shortened: in this regime the diameter is remarkably small (at most $2$). In contrast, in the supercritical regime $\alpha > 1$, it is natural to expect properties sufficiently \textit{near} to those of the Erd\H{o}s--R\'enyi model $\mathcal{G}(n,\varepsilon)$: the graph exhibits connectivity, but can have a larger diameter. We present the proof of Theorem \ref{thm:dust and connectivity threshold} in Section \ref{connectivity}.

\begin{figure}[h] 
\centering
\begin{tikzpicture}[xscale=12, yscale=1.5]

\definecolor{isolated}{rgb}{1.0, 0.75,0.25} 
\definecolor{connected}{rgb}{0.6, 0.8, 1.0} 
\definecolor{diameter2}{rgb}{0.2, 0.4, 0.8} 

\node[anchor=east] at (-0.05, 3.5) {\textbf{$\alpha < 1$}};
\draw[thick, ->] (0,3.5) -- (1.05,3.5) node[right] {$\varepsilon_n$};

\fill[isolated] (0, 3.4) rectangle (0.16, 3.6);
\fill[diameter2] (0.16, 3.4) rectangle (1, 3.6);

\draw[line width=2.5pt, black] (0.16, 3.3) -- (0.16, 3.7) node[above, yshift=2pt] {$\varepsilon^\alpha = \frac{\lambda}{n}$};

\node[anchor=east] at (-0.05, 2) {\textbf{$\alpha = 1$}};
\draw[thick, ->] (0,2) -- (1.05,2) node[right] {$\varepsilon_n$};

\fill[isolated] (0, 1.9) rectangle (0.5, 2.1);
\fill[diameter2] (0.5, 1.9) rectangle (1, 2.1);

\draw[thick, dashed] (0.5, 1.8) -- (0.5, 2.2) node[above, yshift=2pt] {$\varepsilon = \frac{\lambda}{n}$};
\draw[line width=2.5pt, black] (0.32, 1.8) -- (0.32, 2.2) node[above, yshift=2pt] {$-\varepsilon \ln(\varepsilon)= \frac{\lambda}{n}$};

\node[anchor=east] at (-0.05, 0.5) {\textbf{$\alpha > 1$}};
\draw[thick, ->] (0,0.5) -- (1.05,0.5) node[right] {$\varepsilon_n$};

\fill[isolated] (0, 0.4) rectangle (0.67, 0.6);
\fill[connected] (0.67, 0.4) rectangle (0.84, 0.6);
\fill[diameter2] (0.84, 0.4) rectangle (1, 0.6);

\draw[line width=2.5pt, black] (0.5, 0.3) -- (0.5, 0.7) node[above, yshift=2pt] {$\varepsilon = \frac{\lambda}{n}$};
\draw[thick, dashed] (0.67, 0.3) -- (0.67, 0.7) node[above, yshift=2pt] {$\varepsilon = \frac{\lambda\ln n }{ n}$};
\draw[thick, dashed] (0.84, 0.3) -- (0.84, 0.7) node[above, yshift=2pt] {$\varepsilon^\alpha = \frac{\lambda}{n}$};

\begin{scope}[shift={(0,-1)}]
    \fill[isolated] (0,0) rectangle (0.05, 0.2) node[right, black, xshift=5pt, yshift=-5pt] {$\exists$ isolated vertex w.h.p.~[Theorem \ref{thm:dust and connectivity threshold}]};
    \fill[connected] (0, -0.4) rectangle (0.05, -0.2) node[right, black, xshift=5pt, yshift=-5pt] {$\mathbb{G}_n(\alpha,\varepsilon)$ is connected w.h.p.~[Theorem \ref{thm:dust and connectivity threshold}]};
    \fill[diameter2] (0, -0.8) rectangle (0.05, -0.6) node[right, black, xshift=5pt, yshift=-5pt] {$\mathbb{G}_n(\alpha,\varepsilon)$ is connected, $\exists$ hub and diam($\mathbb{G}_n(\alpha,\varepsilon))=2$ w.h.p.~[Theorem \ref{thm:dust and connectivity threshold}]};
    \draw[line width=2pt] (0, -1.1) -- (0.05, -1.1) node[right, black, xshift=5pt] {Degree convergence scaling [Theorem \ref{thm:expecteddegree}]};
\end{scope}
\end{tikzpicture}
\caption{\textbf{Connectivity regimes and diameter transitions for $\mathbb{G}_n(\alpha, \varepsilon)$.}\\ For $\alpha > 1$, the model recovers classical Erd\H{o}s--R\'enyi behavior, exhibiting a more nuanced structural evolution: the preferential scaling ($\varepsilon = \lambda/n$, with $\lambda\in (0,\infty)$) is separated from the connectivity threshold ($\varepsilon =\ln(n)/n$) and the ultra-small-world property is not reached immediately, appearing only at higher densities when extreme weight vertices act as central hubs of attraction, effectively connecting the entire graph. Conversely, for $\alpha < 1$, the scenario changes drastically: the degree scaling ($\varepsilon^\alpha = \lambda/n$, with $\lambda\in (0,\infty)$) coincides with the connectivity threshold, as diameter 2 is triggered as soon as the graph becomes connected. The critical case $\alpha = 1$ acts as a boundary where a logarithmic correction $\varepsilon \ln \varepsilon = -\lambda/n$ is required for degree convergence.}
\label{fig:regimi_connettivita}
\end{figure}

\subsubsection{Infinite mean weights: topology of the graph in the relevant scale}

Our remaining results focus on the infinite mean case, in the relevant scaling of $\varepsilon^\alpha=\lambda/n$. This scaling allows degree convergence (see Theorem \ref{thm:expecteddegree}) and also acts as a threshold for connectivity (see Theorem \ref{thm:dust and connectivity threshold}). 
Theorem \ref{thm:giantsize} shows that the size of the giant component is intrinsically random in the large limit, since it does not concentrate around a deterministic value. The presence of extreme hubs and the high variability of the weights lead to a loss of self-averaging for many macroscopic properties of the graph, including connectivity, the fraction of isolated vertices and the giant component size.

\begin{theorem}[Infinite mean: the asymptotic size of the giant is random]\label{thm:giantsize}
Fix $\lambda \in (0,\infty)$.  Let $\mathbb{G}_n(\alpha,\varepsilon)$ with $\alpha<1$ and $\varepsilon^\alpha=\lambda/n$. Let $\Pi$ be a Poisson Point Process on $(0,\infty]$ with intensity measure $\nu(dt)=\lambda \alpha t^{-(\alpha+1)}dt$ and let $W\sim \text{Pareto}(\alpha)$ be a Pareto random variable independent of $\Pi$. It holds
\begin{equation*}
\frac{|\mathcal{C}_{max}|}{n} \xrightarrow{d} 1-\Xi
\end{equation*}
where
\begin{equation*}
\Xi = \Eb\left[\prod_{z \in \text{supp}(\Pi)} (1 - \min(Wz,1))  \;\middle|\;\Pi \right].
\end{equation*} 
\textcolor{red}{}
\end{theorem}
With a straightforward computation one sees that $\Eb[\Xi]=\frac{\lambda}{1-\alpha} \Gamma\left(-1, \frac{\lambda}{1-\alpha}\right)$.

Theorem \ref{thm:giantsize} highlights the lack of concentration of the giant component size in the infinite-mean regime. We now give an interpretation of this phenomenon, based on the strategy of the proof. The limit $1-\Xi$ is an intrinsically random quantity and it depends on the order statistics of the Pareto weights.
More precisely, \(\Xi\) can be interpreted as the asymptotic probability that a typical vertex, with Pareto weight \(W\), is not connected to any of the vertices with large weights, namely those with weight on the same scale of the weight of the hubs, that is $\asymp \eps^{-1}$. The extremal weights of these vertices are described, in the limit, by the Poisson Point Process (PPP) \(\Pi\) and are asymptotically independent of the weight of any fixed vertex with finite weight. From this perspective, the giant component may be viewed as consisting of the vertices with weight $\asymp \eps^{-1}$, together with their immediate neighbors.
Note that, in this framework, $\lambda$ is related to the asymptotic expected number of hubs, being indeed a parameter of  the intensity of $\Pi$, namely $\nu(dt)=\lambda \alpha t^{-(\alpha+1)} dt$. The greater the number of hubs, the larger the giant component.
We prove Theorem \ref{thm:giantsize} in Section \ref{sec:giantsize}.

As a consequence of Theorem \ref{thm:giantsize}, Theorem \ref{thm:nosharp} will continue in the infinite-mean case the analysis of Theorem \ref{thm:dust and connectivity threshold} for the connectivity threshold. This contrasts with classical results, where a single 
$\lambda_*$
 produces a sharp threshold (see \cite{Devroye2014Connectivity}): here, for each value of $\lambda$, the event that the graph is connected occurs with asymptotically probability strictly far from $0$ and $1$. In particular, as the proof shows, the graph is connected exactly when a hub appears. 

\begin{theorem}[Infinite mean: no sharp connectivity threshold]\label{thm:nosharp} Fix $\lambda \in (0,\infty)$. 
Let $\mathbb{G}_n(\alpha,\varepsilon)$ with $\alpha<1$ and $\varepsilon^\alpha=\lambda/n$. Then, the following holds:
\begin{equation*}
    \mathbb{P}(\mathbb{G}_n(\alpha,\varepsilon)\text{ is connected})\to1-e^{-\lambda}.
\end{equation*}
\end{theorem}
The proof of Theorem \ref{thm:nosharp}  is in Section \ref{sec:nosharp}.

The following result complements Theorem \ref{thm:giantsize} by describing the size of the second-largest component of \(\mathbb{G}_n(\alpha,\varepsilon)\) in the infinite-mean regime, when \(\varepsilon^\alpha=\lambda/n\).

\begin{theorem}[Infinite mean: second largest component size]\label{thm:2nd_component} Fix $\lambda \in (0,\infty)$.
Let $\mathbb{G}_n(\alpha,\varepsilon)$ with $\alpha<1$ and $\varepsilon^\alpha=\lambda/n$. Let $k>(1-\alpha)^{-1}$. Then, as $n\to\infty$, it holds
\begin{equation*}
\mathbb{P}(|\mathcal{C}_{(2)}|\geq k)\to 0.
\end{equation*}
\end{theorem}

Then, the second largest component, if it exists, has constant size. In particular, w.h.p.\ there are no components with $k$ vertices, for each natural $k>(1-\alpha)^{-1}$.

Let us briefly comment on this result. While Theorem \ref{thm:giantsize} shows that the system is dominated by a component of linear size, Theorem~\ref{thm:2nd_component} establishes that all remaining vertices, if any, form components of size $O(1)$. The structure of this “dust” is governed by \(\alpha\): only components of size \(k\leq (1-\alpha)^{-1}\) can survive. In particular, when \(\alpha<1/2\), the graph consists asymptotically of the giant component and isolated vertices. As \(\alpha\) approaches the finite-mean boundary at \(1\), increasingly large finite-size components may occur. We prove Theorem \ref{thm:2nd_component} in Section \ref{sec:2nd_component}.

The following short proposition, which we prove in Section \ref{sec:extremeultrasmallworld}, describes the extreme-ultra-small world nature of the infinite-mean regime. In particular, it proves that the diameter of the giant component is bounded by a constant and that the typical distance in $\Cmax$ is at most $3$, w.h.p..

\begin{prop}[Extreme ultra-small-world: diameter and typical distance in the giant component]\label{thm:extremeultrasmallworld}
Fix $\lambda \in (0,\infty)$. Let $\mathbb{G}_n(\alpha, \varepsilon)$ with $\alpha < 1$ and $\varepsilon^\alpha = \lambda/n$. Let $K: = \lfloor (1-\alpha)^{-1} \rfloor$. As $n \to \infty$, 
\begin{equation*}
\mathbb{P}(\text{diam}(\Cmax) \le 2K + 3)\to 1.
\end{equation*}
Moreover, letting $u$ and $v$ be two vertices chosen uniformly at random in $\Cmax$, it holds:
\begin{equation*}
    \mathbb{P}(d(u,v) \le 3) \to 1 \quad \text{as } n \to \infty.
\end{equation*}
\end{prop}

Finally, our last result deals with the distribution of the maximum degree in $\mathbb{G}_n(\alpha,\varepsilon)$. To do so, we need to introduce some additional notation.

\begin{definition}[Rescaled degree statistics]\label{Psi_n}
Let $\mathcal{M}_p((0,1])$ be the space of the point measures on $(0,1]$ equipped with the vague topology and let
\begin{equation*}
\Psi_n:=\sum_{i \in [n] \atop D_n(i)\geq 1} \delta_{\frac{D_n(i)}{n}},
\end{equation*}
where $D_n(i)$ is the degree of vertex $i\in [n]$, defined in~\eqref{eq:degree_i}.
\end{definition}
$\Psi_n$ is a point process in $.\mathcal{M}_p((0,1])$ that register the rescaled degrees statistics of all but isolated vertices.

\begin{definition}\label{h}
For $x>0$, let
       \begin{equation*}
        h(x):=\Eb[\min(xW,1)]=
        \begin{cases}
         \frac{x^{\alpha}-\alpha x}{1-\alpha}   &x < 1;\\
         1   &x\geq 1,
        \end{cases}
    \end{equation*}
    where $W\sim Pareto(\alpha)$.
    Moreover, if $\mu \in \mathcal{M}_p((0,\infty])$ is a point measure, let
\begin{equation*} 
\mu[h]:=\sum_{z \in \mu}\delta_{h(z)}
\end{equation*}
be the transformed point measure using the continuous function $h$.
\end{definition}
The value $h(x)$, for $x>0$, can be interpreted as the asymptotic probability that a typical vertex (with weight $W$) is connected with a hub with weight $\simeq x/\eps$. Given the above definition, we can state our result on extremal statistics of degrees in our graph.

\begin{theorem}[Extreme Degrees]\label{thm:maxdegree}
Fix $\lambda \in (0,\infty)$. Let $\mathbb{G}_n(\alpha,\varepsilon)$ with $\alpha<1$ and $\varepsilon^\alpha=\lambda /n$. Let $\Psi_n$ as in Definition~\ref{Psi_n} and $h$ as in Definition~\ref{h}. 
We have
\begin{equation*}
\Psi_n \Rightarrow \Pi[h],
\end{equation*}
where the convergence is in distribution, with respect to vague topology on $\mathcal{M}_p((0,1])$ and $\Pi\in \mathcal{M}_p((0,\infty])$ has intensity measure $\nu(dt)=\lambda \alpha t^{-(\alpha+1)}dt$.  In particular, the maximum degree $\Delta_{(1)}=\max\limits_{i \in [n]} D_n(i)$ satisfies
\begin{equation*}
\frac{\Delta_{(1)}}{n}\xrightarrow{d} h(Z_{(1)})
\end{equation*}
where $Z_{(1)}=\max\{z \in \text{supp}(\Pi)\}$ follows a Fréchet distribution modified by $\lambda$, i.e. 

\begin{equation*}
\Pb\left(Z_{(1)}\leq x\right)=e^{-\lambda x^{-\alpha}}.
\end{equation*}
\end{theorem}

The result above shows that, as expected, the largest degrees are deterministically related to largest weights. However, due to the saturation effect given by the largest possible degree being $n-1$ and the fact that this is achieved as soon as one vertex is a hub in $\mathcal{H}_{\eps}$ (which is non-empty with positive probability), the scale of the largest degree is $n$ and the limiting random variable has an atom in $1$, since this value is achieved with probability $1-e^{-\lambda}$. We prove Theorem \ref{thm:maxdegree} in Section \ref{sec:maxdegree}. 
\begin{oss}[Limiting distribution of the maximum degree]\label{oss:max_degree}
The rescaled maximum degree $\Delta_{(1)}/n$ converges in distribution to the law of the random variable $h(Z_{(1)})$. Since the maximal atom of the point process $\Pi$, $Z_{(1)}$, follows a Fréchet distribution modified by the scale parameter $\lambda$, we have that the law of $h(Z_{(1)})$ is absolutely continuous with respect to Lebesgue on $(0,1)$ and has an atom in $1$. Indeed, 
\begin{equation*}
    \Pb(Z_{(1)} \le x) = \Pb\big(\Pi([x, \infty)) = 0\big) = e^{-\nu([x, \infty))} = \exp\left( - \int_x^\infty \lambda \alpha t^{-\alpha-1} dt \right) = e^{-\lambda x^{-\alpha}},
\end{equation*}
and for $y \in (0,1)$
\begin{equation*}
    \mathbb{P}\left( \frac{\Delta_{(1)}}{n} \le y \right) \to \Pb\big(h(Z_{(1)}) \le y\big) = \Pb\big(Z_{(1)} \le h^{-1}(y)\big) = \exp\left( -\lambda \big(h^{-1}(y)\big)^{-\alpha} \right),
\end{equation*}
while
\begin{equation*}
    \mathbb{P}\left(\frac{\Delta_{(1)}}{n-1}=1\right)\to1-e^{-\lambda}.
\end{equation*}
\end{oss}

\subsection{Discussion: universal features and related models}

\subsubsection{Similar random graph models}\label{subsec:literature}

From a mathematical perspective, several related models have been extensively analyzed in the literature. Some fundamental references are the Chung-Lu model \cite{Chung2002Average} and the Norros-Reittu model \cite{Norros2006Conditionally}, where the edge probabilities are defined as 
\begin{equation*}
    p_{ij} = \min\left(\frac{W_i W_j}{\sum_{k=1}^n W_k}, 1\right) \quad \quad \text{and} \quad \quad p_{ij} = 1 - e^{-\frac{W_i W_j}{\sum_{k=1}^n W_k}}
\end{equation*}
respectively. Unlike the model discussed in this work, both the Chung-Lu and Norros-Reittu frameworks do not feature an external tuning parameter $\varepsilon$. Instead, they are defined through the total weight sum $\sum_{k} W_k$, which acts as a self-regulating factor. When $\alpha \le 1$, this normalization automatically adjusts the scale of the edge probabilities, leading to structurally different behaviors between the finite and infinite-mean cases. A specific choice of $\varepsilon$ replacing $\sum_k W_k$ (i.e. $\varepsilon=n^{-1/\alpha}$ for $\alpha \in (0,1)$ and $\varepsilon=n^{-1}$ for $\alpha>1$, as detailed in Theorem \ref{thm:expecteddegree}) leads our model to exhibit behaviors similar to those of the aforementioned frameworks in their respective regimes. However, the introduction of an independent parameter $\varepsilon$ allows for a more general analysis, providing the flexibility to explore different edge density regimes that are not constrained by the realization of the weights. This flexibility is crucial in identifying the threshold where the emergence of hubs triggers global connectivity and a collapse in the graph's diameter (see Theorem \ref{thm:dust and connectivity threshold}).  

Another well-known framework is the Generalized Random Graph (GRG), discussed in~\cite{VanDerHofstad2016RandomV1, VanDerHofstad2024RandomV2}, with
\begin{equation*}
    p_{ij} = \frac{  W_i W_j}{\sum_k W_k +  W_i W_j}
\end{equation*} 
and previously studied in \cite{Britton2006Generating} in the form
\begin{equation*}
    p_{ij} = \frac{ W_i W_j}{n + W_i W_j} \quad \text{if $\alpha>1$} \quad\quad \text{ and }\quad\quad p_{ij} = \frac{ W_i W_j}{n^{\frac{1}{\alpha}} + W_i W_j}\quad \text{if $\alpha \in (0,1)$}.
\end{equation*}

In the physics literature, a model with a Norros–Reittu-type exponential kernel in the infinite-mean regime ($\alpha < 1$) was introduced in \cite{Garuccio2023Multiscale} as a scale-invariant framework under additive coarse-graining. Such a scale invariance has been further studied in the recent \cite{avena2026renormalisation}. Yet, the rigorous analysis of the emerging structures in the original ensemble presented in \cite{Garuccio2023Multiscale}  has been started in \cite{avena2026inhomogeneous}, incorporating the scaling parameter $\varepsilon$ directly into the exponential edge probability:
\begin{equation}\label{varp}
 p_{ij} = 1 - e^{-\varepsilon W_i W_j}.
\end{equation}

Further properties of this setting were recently explored numerically in \cite{catanzaro2026spectra} for the spectral behavior under $\varepsilon = n^{-1/\alpha}$. Moreover, in the same setting,  \cite{catanzaro2026clustering} investigates the question of clustering and geometry. The authors underline that this model is an example of a random graph ensemble that has a high clustering coefficient not caused by an underlying geometrical embedding in contrast to the case of e.g. random hyperbolic graphs \cite{BoFoMu16, Gas24}. They show that the fraction of isolated vertices and, as a consequence, the local clustering coefficient have random limits, providing another manifestation of the lack of self-averaging in this regime. 

In this work, we go beyond the analysis of isolated vertices by proving that a complete characterization of the global component structure of the graph is indeed driven by the random environment of heavy weights. In particular, our model, $\mathbb{G}_n(\alpha, \varepsilon)$, is closely related to the form in \eqref{varp}, but instead of using the exponential kernel, it uses its first-order linear approximation $\min(\varepsilon W_iW_j,1)$. This choice is still natural since it introduces a sharp "saturation" threshold that allows for a clear geometric characterization of the hubs in the graph. While the exponential kernel is smooth, the $\min$ function identifies a weight scale $W_i \ge \varepsilon^{-1}$ above which connections become certain, enabling the analysis of the connectivity threshold presented in Theorem \ref{thm:dust and connectivity threshold}, which is not directly accessible with smooth kernels. Despite these technical differences, both kernels yield remarkably similar qualitative and universal features. Our main goal is to explore this model across different regimes of $\alpha$, in particular in the infinite-mean regime where little is known in the literature.

\subsubsection{Connectivity and largest component}\label{subsec:phenomena_comparison}
The connectivity of weighted random graphs has been repeatedly studied in the literature, typically under strong regularity assumptions on the connection kernel. A notable example is the work of Devroye and Fraiman~\cite{Devroye2014Connectivity}, who study models of the form
\begin{equation*}
p_{ij} = \min\left(\frac{k(X_i,X_j)\log n }{ n}, 1\right),
\end{equation*}
where the kernel $k$ is irreducible and satisfies boundedness conditions such as $\lambda_2 \in L^\infty$, with $\lambda_2(x) = \big( \int k(x,y)^2 \, d\mu(y) \big)^{1/2}$. In this setting, connectivity is governed by $\lambda_* = \mathrm{ess\,inf}\,\lambda(x)$, where $\lambda(x) = \int k(x,y)\,d\mu(y)$, and a sharp threshold occurs at $\lambda_* = 1$. In this work, since our weights are heavy-tailed, the resulting kernel is unbounded, and our model therefore falls outside this framework for all values of $\alpha$. Moreover, in the case of $\alpha<1$ the connectivity threshold is not sharp. Notice that this same phenomenon happens in the critical $1/2$-case in the random hyperbolic graph~\cite{BoFoMu16}. 

Another very natural question is on the largest component of the graph, since the seminal result on the phase transition for the Erd\H{o}s-R\`enyi graph in the sparse regime~\cite{Erdos1960Evolution}. In most of the classical random graph processes, one can identify regimes in which one component proportional to the size of the full network appears and this proportion usually concentrates (very) fast around the typical limit (see~\cite{AnKoPa21, AnKoLaPa23,  JoKoDi26, JoZw26, MoRe98, OCo98}). By contrast, in our regime the giant component exhibits a striking lack of self-averaging, converging in distribution to a random limit (see Theorem \ref{thm:giantsize}).

\subsubsection{Maximum degree} 
The question of the largest degree is a natural one and has been studied in several contexts of random graph theory. In the classical Erd\H{o}s-R\'enyi sparse random graph, the light tails of the degree distribution cause the (properly scaled) maximum degree to concentrate around two deterministic values; see~\cite{Bollobas2001Random}. The same phenomenon occurs in random geometric graphs, as shown by Penrose~\cite{Pen03}.

In the setting of heavy-tailed weighted graphs, a genuinely continuous extreme-value distribution emerges instead. To the best of our knowledge, this has only been addressed by~\cite{BhSc22}, in a number of related but distinct models. They consider a geometric random graph, along with the Chung-Lu model and a multi-edge version of the Norros-Reittu model (models \textbf{V} and \textbf{IV}, respectively, in their notation), which are closely related to our framework. They study the distribution of the largest degree under the appropriate scaling. Two key differences separate their results from ours. First, in the Chung-Lu model, they establish convergence of the rescaled largest degree to a Fr\'echet distribution (the natural limiting extreme-value law in this setting) only when the weights have finite second moment ($\alpha > 2$ in our notation). Second, in the multi-edge version of the Norros-Reittu model, they obtain convergence to a Fr\'echet distribution for every $\alpha > 0$; however, since multiple edges are permitted, no saturation effect arises, and the limit remains Fr\'echet throughout.

A similar phenomenon has been observed in the setting of hyperbolic random graphs, see~\cite{Gas24}. There, the author proves that the ordering between degree and distance from the center (here, distance replacing weight) is preserved, analogous to our Theorem~\ref{thm:maxdegree}. In that setting, when the curvature parameter is critically equal to $1/2$, the largest degrees scale with $n$, and a saturation effect at $n-1$ can likewise be observed; see~\cite[Theorem~3.3]{Gas24}. In the random hyperbolic graph this saturation is however weaker than in our case, since the largest degree of the graph (scaled with $n$) converges to a random variable whose distribution is absolutely continuous with respect to the Lebesgue measure on $(0,1]$, in contrast to our case where the limit has an atom in $1$, see Remark~\ref{oss:max_degree}.

\subsubsection{Possible Extensions}
\label{sub:generalization}
We now discuss how the main results would adapt under alternative modeling assumptions. Let us generalize the edge probabilities, giving them the form
\begin{equation*}
    p_{ij} = g(\varepsilon W_i W_j),
\end{equation*}
with $g: (0, \infty) \to (0, 1]$ an increasing function. Equivalently, this corresponds to working with a kernel $k_n(x,y) = n g(\varepsilon x y)$. In this paper, we consider $g_1(x) = \min(x, 1)$. Other well-known choices in the literature include $g_2(x) = 1 - e^{-x}$ (Norros--Reittu) and $g_3(x) = \frac{x}{1+x}$ (GRG), see Section~\ref{subsec:literature}.

The connection function $g$ impacts the model dynamics both at the local and global levels. In particular, $g_1$, $g_2$, and $g_3$ all share the same degree scalings, established in Theorem~\ref{thm:expecteddegree}, up to prefactors. On the other hand, global connectivity is fundamentally altered if $g$ never saturates to $1$. 
The macroscopic results of Theorems \ref{thm:giantsize}, \ref{thm:2nd_component}, and \ref{thm:maxdegree}, as well as Proposition~\ref{thm:extremeultrasmallworld}---including the non-self-averaging giant component and the modified-Fréchet degree dynamics---remain structurally robust even with a general $g$.
It is worth mentioning that in Theorem \ref{thm:maxdegree}, saturation induces a Dirac point mass at $1$ in the limiting maximum degree, while for non-saturating functions $g$ (where $g(x)<1$ for all $x$), this atom disappears, making the result closer to the one on random hyperbolic graphs~\cite{Gas24}. Related questions concerning universality are further addressed in \cite{avena2026renormalisation} within a network renormalization framework.

Moreover, we highlight that as long as we simply shift the support of the weight distribution $f_W$ to $[a, \infty)$ for some $a > 0$, the control over the dust remains and the results are expected to hold in the same way. On the other hand, the setup involving a general slowly varying function, i.e., 
\begin{equation*}
    \mathbb{P}(W > w) \simeq L(w) w^{-\alpha}\quad \text{as } w \to \infty
\end{equation*} 
with $L(wa)/L(w)\to 1$ as $w\to \infty$ for all $a>0$, requires a slightly more delicate treatment. Here, the critical scaling $\varepsilon_n$ must be defined implicitly via $n \mathbb{P}(W > 1/\varepsilon_n) \to \lambda$ (or $n \varepsilon_n^\alpha L(1/\varepsilon_n) \to \lambda$). A fully rigorous extension to general slowly varying functions remains open, as it would require handling the implicit sequence $\varepsilon_n$; nevertheless, it seems reasonable to hypothesize that Karamata's integration theorem preserves the essential scaling behavior of the local moments.


\section{Connectivity Threshold. Proof of Theorem \ref{thm:dust and connectivity threshold}} \label{connectivity}

In this section, we present the proof of Theorem \ref{thm:dust and connectivity threshold}, 
which addresses the connectivity threshold for all $\alpha>0$. In Section~
\ref{connectivityregime}, we deal with the supercritical regime by stating 
Lemmas~\ref{connectivity1} and \ref{conalpha}, which provide two different 
sufficient conditions for the graph to be connected w.h.p. In Section~\ref{dustregime}, we analyze the subcritical case to show that w.h.p.\ there is at least one isolated vertex. 
Finally, in Subsection \ref{finalpart}, we prove Theorem \ref{thm:dust and connectivity threshold} 
by combining the results of both the supercritical and subcritical regimes for 
every value of $\alpha>0$.


\subsection{Connectivity regime}\label{connectivityregime}
\begin{lemma}[Connectivity regime]\label{connectivity1}
Let $\alpha>0$. If $\varepsilon n/\ln(n)\to \infty$ the graph is connected w.h.p., i.e.~
\begin{equation*}
    \mathbb{P}\left(\mathbb{G}_n(\alpha,\varepsilon) \text{ is connected}\right)\to1\quad\quad\text{ as }n\to\infty.
\end{equation*}
\end{lemma}

\begin{proof}
Since $\min(\varepsilon W_i W_j,1)\geq \min(\varepsilon,1)=\varepsilon$, we have that the edge probability in $\mathbb{G}_n(\alpha,\varepsilon)$ stochastically dominates the one in the Erd\H{o}s--R\'enyi model \(G \sim \mathcal{G}(n, \varepsilon)\). Since the event ``the graph is connected'' is increasing, we obtain 
\begin{equation*}
    \mathbb{P}\big(\mathbb{G}_n(\alpha, \varepsilon) \text{ is connected}\big) \geq \mathbb{P}(G \text{ is connected}) \to 1,
\end{equation*} indeed \(\varepsilon = \omega(\ln(n)/n)\)  (see \cite{Erdos1960Evolution} or \cite[Theorem 5.8]{VanDerHofstad2016RandomV1}). 
\end{proof}

\begin{prop}[Hubs Bring Connectivity]\label{core_bounds}
Let $\mathcal{H}_\varepsilon$ be the set of hubs as in Definition \ref{core}. Let $\alpha > 0$ and $\varepsilon > 0$. For every $n$ fixed, it holds:
\begin{equation*}
    \mathbb{P}(\mathbb{G}_n(\alpha, \varepsilon) \text{ is connected}) \geq \mathbb{P}(\text{diam}(\mathbb{G}_n(\alpha, \varepsilon)) \leq 2)  \geq \mathbb{P}(H_\varepsilon \neq \emptyset).
\end{equation*}
\end{prop}
\begin{proof}
 We say that a graph contains a "big star" if there is a vertex that is connected to every other vertex and we define the event \[S=\{\text{the graph contains a big star}\}.\] 
 We now show that if an element, say $1$, is a hub, then the graph contains a "big star". For a fixed $n$, we may write:
\begin{align*}
    \mathbb{P}\left(S \;\Big|\; 1 \in \HH_{\varepsilon}\right)&\geq \mathbb{P}\left(1\leftrightarrow j \text{ }\forall j \geq 2 \;\Big|\; W_1\geq \frac{1}{\varepsilon}\right)\\
    &=\mathbb{E}\left[\prod_{j\geq 2}\min(\varepsilon W_1W_j,1) \;\Big|\; W_1\geq \frac{1}{\varepsilon}\right]\\
    &\geq\mathbb{E}\left[\prod_{j\geq 2}\min(W_j,1)\;\Big|\; W_1\geq \frac{1}{\varepsilon}\right]\\
    &=\mathbb{E}[\min(W_2,1)]^{n-1}=1.
\end{align*}
By vertex exchangeability, the same holds for each vertex, i.e. $\mathbb{P}(S | i \in \HH_{\varepsilon})=1$ for all $i \in [n]$. Thus, one has $\mathbb{P}(S^\cc \cap \{ i \in \HH_\varepsilon\})=0$ for all $i \in [n]$ and, with a union bound, we obtain
\begin{equation*}
    \mathbb{P}(S^\cc \cap \{\HH_\varepsilon \neq \emptyset\})=\mathbb{P}\left(\bigcup_{i \in [n]}(S^\cc \cap \{i \in \HH_\varepsilon\}) \right)\leq \sum_{i \in [n]}\mathbb{P}(S^\cc \cap \{i \in \HH_\varepsilon \})=0.
\end{equation*}
So, clearly $\mathbb{P}(S \cap \{\HH_\varepsilon\neq \emptyset\})=\mathbb{P}(\HH_\varepsilon \neq \emptyset)-\mathbb{P}(S^\cc \cap \{\HH_\varepsilon\neq \emptyset\})=\mathbb{P}(\HH_\varepsilon \neq \emptyset)$ that gives 
\begin{align*}
\mathbb{P}(\mathbb{G}_n(\alpha,\varepsilon)\text{ is connected})&\geq\mathbb{P}(\text{diam}(\mathbb{G}_n(\alpha,\varepsilon))\leq 2)\geq\mathbb{P}(S)\\
&\geq \mathbb{P}(S \cap\{\HH_{\varepsilon}\neq \emptyset\})= \mathbb{P}(\HH_\varepsilon \neq \emptyset).
\end{align*}
\end{proof}

\begin{corollary}[Connectivity regime and diameter]\label{conalpha} 
Let $\alpha>0$ with $\varepsilon^\alpha n\to \infty$. Then $\HH_\varepsilon$ is non-empty w.h.p. In particular, as $n\to \infty$, it holds
\begin{equation*}
    \mathbb{P}\left(\mathbb{G}_n(\alpha,\varepsilon) \text{ is connected}\right)\geq \mathbb{P}(\text{diam}(\mathbb{G}_n(\alpha,\varepsilon))\leq 2)\to1.
\end{equation*}
\end{corollary}

\begin{proof}
Since $|\HH_{\varepsilon}|\sim \text{Bin}(n,\varepsilon^\alpha)$, it holds $\mathbb{P}(|\mathcal{H}_\varepsilon|=0)=(1-\varepsilon^\alpha)^n\leq e^{-n\varepsilon^\alpha}\to0$, i.e. the Core $\mathcal{H}_\varepsilon$ is non-empty w.h.p. Using Proposition \ref{core_bounds} we can easily conclude. 
\end{proof}

\subsection{Dust regime}\label{dustregime}

We start introducing some elements useful for the proof of the subcritical regime. Let $h(n)=1+g(n)$ with $g(n)=\ln(n)^{-1}$ and let 
    \begin{equation*}
        \mathcal{S}_h:=\{i \in [n]: W_i\leq h(n)\}
    \end{equation*}
    be the set of  vertices with small weights, i.e.\ sufficiently close to the minimum $1$. We observe that the set $\mathcal{S}_h$ is w.h.p. non-empty since $|\mathcal{S}_h|\sim \text{Bin}(n,p_n)$ with $p_n=1-(1+\ln(n)^{-1})^{-\alpha}\simeq \alpha \ln(n)^{-1}$. Then, let
    \begin{equation*}
        \mathcal{I}_h:=\sum_{i \in [n]}\ind(W_i\leq h(n))\ind(i \text{ is isolated})=\sum_{i \in \mathcal{S}_h}\ind(i \text{ is isolated})
    \end{equation*}
    be the number of vertices in $S_h$ that are \textbf{isolated} in $\mathbb{G}_n(\alpha,\varepsilon)$. Our aim is to show that, w.h.p., the graph contains at least one isolated vertex.  In particular, we want to show $\mathbb{P}(\mathcal{I}_h\geq 1)\to 1$ as $n\to \infty$. For any $w\geq 1$ and $m\in \N$, let $\mathbb{E}_w[p^m_{ij}]:=\mathbb{E}[p^m_{ij}|W_i=w]$ be the $m$-th moment of the edge probability $p_{ij}=\min(\varepsilon W_iW_j,1)$ conditioned on having one weight equal to $w$. A simple computation, see Lemma~\ref{norma2}, gives
\begin{align}\label{A}
\hyperref[table]{\mathbb{E}_x[p_{12}^m]}=
\begin{cases}
\frac{m}{m-\alpha}(\varepsilon x)^\alpha+\frac{\alpha}{\alpha-m}(\varepsilon x)^m 	&\alpha \neq m,\\
-(\varepsilon x)^m\ln((\varepsilon x)^m)+(\varepsilon x)^m 					&\alpha=m,
\end{cases}
\end{align}
if $x\leq 1/\varepsilon$, 1 otherwise. To prove that $\mathbb{G}_n(\alpha,\varepsilon)$ has an isolated vertex w.h.p. we use Lemma~\ref{smallverticesmoment2} in this specific setting.
Hence, our goal is to ensure the that two conditions in Lemma~\ref{smallverticesmoment2} to obtain the existence of an isolated vertex w.h.p.\ are satisfied.

\begin{lemma}[Dust regime]\label{dust}
With the following choices
\begin{enumerate}
\item 	$\varepsilon^\alpha n\to 0$ when $\alpha \in (0,1]$;
\item 	$\varepsilon n/\ln(n)\to 0$	 when $\alpha >1$,
\end{enumerate}
the graph contains an isolated vertex w.h.p., i.e.~~$\mathbb{P}(\mathcal{I}\geq1)\to1.$ In particular, the graph is not connected w.h.p..
\end{lemma}

\begin{proof}
From  Lemma~\ref{smallverticesmoment2} it is sufficient to show $\mathbb{E}[\mathcal{I}_h] \to \infty$ and $\phi_n \to 0$ as $n \to \infty$, where $\phi_n$ is defined as:
\begin{equation*}
\phi_n := n \left( -2\mathbb{E}_1[p_{12}] + \mathbb{E}_1[p_{12}^2] + 2\frac{\mathbb{E}_{h(n)}[p_{12}]}{1-\mathbb{E}_{h(n)}[p_{12}]} \right).
\end{equation*}
The asymptotic behavior of each term in the expression of $\phi_n$ in $\mathbb{G}_n(\alpha,\varepsilon)$ is analyzed across the different regimes of $\alpha$, using \eqref{A}, and it is summarized in Table~\ref{table:asymptotics_phi_detailed}. In particular, it holds
\begin{equation*}
\phi_n\simeq 
\begin{cases}
\frac{2}{2-\alpha}\varepsilon^\alpha n\to 0	&\alpha \in (0,1),\\
2\varepsilon n\to 0 	&\alpha=1,\\
2\mathbb{E}[W] \frac{\varepsilon n}{\ln(n)}\to 0 	&\alpha>1
\end{cases}
\end{equation*}
as $n\to \infty$ under the appropriate scaling hypotheses for $\varepsilon$ in each regime. 
  Moreover,
\begin{align}\label{primomomentoi}
\notag\mathbb{E}[\mathcal{I}_h]=&\mathbb{E}\left[\sum_{i \in [n]}\ind(W_i \leq h(n))\ind(\text{i is isolated})\right]\\
\notag&=n\mathbb{E}\left[\ind(W_1 \leq h(n))\prod_{j\geq 2} (1-\min(\varepsilon W_1W_j,1)) \right]\\
\notag&\geq n\mathbb{P}(W_1\leq h(n))\mathbb{E}[1-\min(\varepsilon h(n) W_2,1)]^{n-1}\\
\notag&= n \left(1-\left(1+\frac{1}{\ln(n)}\right)^{-\alpha}\right)  (1-\mathbb{E}_{h(n)}[p_{12}])^{n-1}\\
\notag&\geq c\frac{n}{\ln(n)}(1-\mathbb{E}_{h(n)}[p_{12}])^{n-1}\geq c\frac{n}{\ln(n)}e^{-(n-1)\frac{D}{2}}\\
\notag &= ce^{-D(n-1)/2+\ln(n)-\ln(\ln(n))}\\
&= c \exp\left(\ln(n)\left( 	1-\frac{1}{2}\frac{D (n-1)}{\ln n}-\frac{\ln(\ln(n))}{\ln n}\right) 	\right),
\end{align}
where $c>0$ is a suitable constant depending on $\alpha$ and 
\begin{equation*}
    D=2\frac{\mathbb{E}_{h(n)}[p_{12}]}{1-\mathbb{E}_{h(n)}[p_{12}]}\simeq 
    \begin{cases}
        \frac{2}{1-\alpha}\varepsilon^\alpha  &\alpha \in (0,1),\\
        -2\varepsilon \ln{\varepsilon}  &\alpha=1,\\
        2\mathbb{E}[W]\varepsilon  &\alpha>1
    \end{cases}
\end{equation*}
is the term listed in Table \ref{table:asymptotics_phi_detailed} deriving from the bound $1-x\geq \exp(-x/(1-x))$. Our hypotheses ensure that $D(n-1)/\ln(n)\to 0$ in each case: for $\alpha>1$ it is trivially true, for the case $\alpha=1$ we have $\varepsilon n\to 0 \implies -\varepsilon \ln(\varepsilon)n/\ln(n)\to0$ and for $\alpha\in(0,1)$ the hypothesis $\varepsilon^\alpha n\to 0$ implies $\varepsilon^\alpha n/\ln(n)\to 0$. Thus \eqref{primomomentoi} tends to $\infty$, which finally yields $\mathbb{E}[\mathcal{I}_h]\to \infty$ as desired.
\end{proof}

\begin{oss}[Comments on the proof of Lemma \ref{dust}]\label{obs2}
   Beyond the analytical computations, $\phi_n$ quantifies the deviation of $\mathbb E[\mathcal I_h^2]$ from $\mathbb E[\mathcal I_h]^2$. Roughly speaking, the condition $\phi_n \to 0$ is related to the fact that the isolation of a specific small vertex does not significantly affect the probability of isolation of other small vertices as the system size grows.

  The restriction to the set $S_h$ is essential in this respect. Vertices with large weights may act as common neighbors for many pairs of vertices, generating non-negligible correlations between their isolation events. By imposing a uniform upper bound on the weights, these effects are suppressed and the contribution of joint isolation events becomes asymptotically negligible.

  The mechanism through which $\phi_n \to 0$ depends on the regime of $\alpha$ (see Table \ref{table:asymptotics_phi_detailed}). When $\alpha > 1$, the second-order interactions are too weak to influence the leading order, and $\phi_n$ is dominated by the first-order terms. Conversely, in the $\alpha < 1$ regime, the contribution of the second moment $A$ becomes significant, reflecting a more "rigid" structure where local dependencies are more persistent and harder to decouple. Finally, $\alpha = 1$ acts as the critical threshold between these two structural behaviors.
\end{oss}

\subsection{Proof of Theorem \ref{thm:dust and connectivity threshold}}\label{finalpart}

Here we see the proof of Theorem \ref{thm:dust and connectivity threshold}

\begin{proof}[Proof of Theorem \ref{thm:dust and connectivity threshold}]
We analyze each case. For $\alpha \in (0,1]$, Lemma \ref{dust} ensures that if $\varepsilon^\alpha n\to 0$, then $\mathbb{P}(\mathcal{I}\geq1)\to 1$. In particular, w.h.p.\ it is disconnected. Lemma \ref{conalpha} on the other hand guarantees that if $\varepsilon^\alpha n\to \infty$, then $\mathbb{P}(\mathbb{G}_n(\alpha,\varepsilon)\text{ is connected})\geq\mathbb{P}(\text{diam}(\mathbb{G}_n(\alpha,\varepsilon)\leq 2)\to1$. This completes the case $\alpha \in (0,1]$ identifying $\varepsilon=\Theta(n^{-1/\alpha})$ as threshold for the connectivity of the graph. 
On the other hand, with $\alpha>1$, Lemma \ref{dust} ensures that if $\varepsilon n/\ln(n)\to 0$, then $\mathbb{P}(\mathcal{I}\geq 1)\to 1$ and Lemma \ref{connectivity1} shows that if $\varepsilon n/\ln(n)\to \infty$, then $\mathbb{P}(\mathbb{G}_n(\alpha,\varepsilon)\text{ is connected})\to 1$. This identifies $\varepsilon=\Theta(\ln{n}/n)$ as the connectivity threshold for $\alpha>1$ in analogy with the homogeneous case. Lemma \ref{conalpha} with $\alpha>1$ only specifies that when $\varepsilon^\alpha n\to \infty$, then $\text{diam}(\mathbb{G}_n(\alpha,\varepsilon))\leq 2$ w.h.p..
\end{proof}


\section{Infinite Mean: The Giant Size is Random. Proof of Theorem \ref{thm:giantsize}}\label{sec:giantsize}

In this section, we prove Theorem \ref{thm:giantsize}, working in the infinite mean setting ($\alpha$<1) with $\varepsilon^\alpha=\lambda/n$, $\lambda \in (0,\infty)$. Since $|\HH_\varepsilon|\sim \text{Bin}(n,\lambda/n)$, it is clear that the Core $\HH_\varepsilon$ can be empty with a positive probability. Thus, in order to work with a non-empty set, we need to weaken the assumptions and work instead with the set of \emph{weak hubs} $\WW_\delta$, as introduced in Definition \ref{core_delta}. We recall, from Definition~\ref{core_delta} and Definition~\ref{Ndelta}, the two sets
\begin{equation*}
    \WW_\delta:=\{i \in [n]: W_i>n^{\frac{1-\delta}{\alpha}} \}\subseteq
\widebar{\WW}_\delta := \{ i \in [n] : i \in \WW_\delta \text{ or } \exists j \in \WW_\delta \text{ s.t. } i \leftrightarrow j \}.
\end{equation*}
First, in Section~\ref{sec:coredelta_prop} we give basic definitions and we show some properties of the set $\WW_\delta$. In particular, Lemma \ref{isolatedornot} says that a vertex $i \in [n]$ either is directly connected to $\WW_\delta$ or it is isolated w.h.p.. As a consequence, $|\mathcal{C}_{max}|/n$ can effectively be analyzed through the study of $|\widebar{\WW}_\delta|/n$. See Figure \ref{fig:graph_structure} for a graphical sketch of the structure of $\mathbb{G}_n(\alpha,\varepsilon)$ in this regime.
Then, in Section \ref{PPPconvergence} we show that the  Point Process $\Pi_n$ of rescaled weights of weak hubs converges weakly to a PPP. In Section~\ref{convergenzafunzionali}, we analyze a specific functional $F:\mathcal{M}_p((0,\infty])\to \R$ and we show that $F(\Pi_n)\to F(\Pi)$.
Later, using the previous results, in Section~\ref{proofgiantsize} we prove Lemma~\ref{lemmaginatsize} that characterizes the random limit of $|\widebar{\WW}_\delta|/n$. 
Finally, in Section \ref{finallyherethetheorem}, we prove Theorem \ref{thm:giantsize}, putting all the pieces together. 
\begin{figure}[h]
\centering
\begin{tikzpicture}[scale=0.6]
    \draw[fill=teallight!80, draw=black] (0,0) circle (5cm);
    \draw[fill=tealmid!80, draw=black] (0,0) circle (3.4cm);    
    \node[black] at (3.3, 2.3) {\small $\widebar{\WW}_\delta $};
    \draw[fill=tealdeep!80, draw=black, thick] (0,0) circle (1.2cm);
    
    \node[black] at (1.7, 0.5) {\small $\WW_\delta$ };

    \foreach \a in {0,60,120,180,240,300}{
        \node[circle, fill=black, inner sep=1.6pt] (c\a) at (\a:0.7cm) {};
    }

    \foreach \x in {0,60,120,180,240,300}{
        \foreach \y in {0,60,120,180,240,300}{
            \draw[black, very thin] (c\x) -- (c\y);
        }
    }

    \foreach \a in {40, 70, 120, 150, 190, 220, 280, 330}{
        \node[circle, fill=black, inner sep=1.6pt] (v\a) at (\a:2.4cm) {};
        \draw[black, thin, shorten >=0.7cm] (v\a) -- (0,0); 
    }

    \foreach \a in {10, 55, 100, 140, 200, 250, 310, 350}{
        \node[circle, fill=black, inner sep=1.6pt] at (\a:4.3cm) {};
    }

\end{tikzpicture}
\caption{Conceptual structure of $\mathbb{G}_n(\alpha,\varepsilon)$ for $\alpha < 1$ and $\varepsilon^\alpha=\lambda/n$. The $\delta$-Core $\WW_\delta$ forms a central clique that attracts a fraction (with mean $1-\xi_\lambda$) of the vertices, while the remaining fraction (with mean $\xi_\lambda$) remains isolated.} 
\label{fig:graph_structure}
\end{figure}
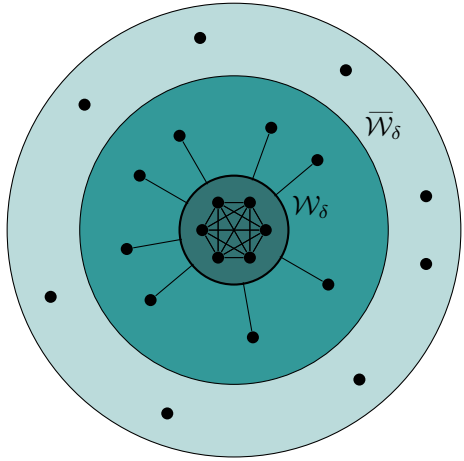

 \subsection{Basic definitions and properties of $\WW_\delta$}\label{sec:coredelta_prop}

In all the section we will work with the following assumptions: we fix $\lambda \in (0,\infty)$ and
consider $\mathbb{G}_n(\alpha,\varepsilon)$ with $\alpha<1$ and $\varepsilon^\alpha=\lambda/n$. We use $\delta\in (0,1/2)$. We start stating some basics definitions and notation useful in all the following subsections.  Let $\mathcal{M}_p((0,\infty])$ be the space of all point measures on $(0,\infty]$ equipped with the vague topology. 

\begin{definition}[$\Pi_n$]\label{defofPin} 
We introduce 
\begin{equation*}
\Pi_n=\sum_{j \in \WW_\delta}\delta_{\varepsilon W_i},    
\end{equation*}
 the point process in $\mathcal{M}_p((0,\infty])$ of rescaled weights of the vertices in $\WW_\delta$. \end{definition}
\begin{definition}\label{def:PPP_limit} Fix $\lambda\in (0,\infty)$ and $0<\alpha<1$.
    Let $\Pi \in \mathcal{M}_p((0,\infty])$ be a Poisson Point Process (PPP)  with intensity measure $\nu$ given by $\nu(dt) = \lambda \alpha t^{-(\alpha+1)}dt.$ 
\end{definition}
\begin{definition}\label{psieF}
We define the following functionals from $\mathcal{M}_p((0,\infty])$ to $\mathbb{R}$: for all  $\mu \in \mathcal{M}_p((0,\infty])$
\begin{align*}
&(i)\quad &\psi_w(\mu) &= \ind\left( \mu[w^{-1}, \infty] = 0 \right) \exp\left(-\int_0^{w^{-1}}-\ln(1 - xw) \mu(dx)\right), \quad \forall w \geq 1;\\
&(ii) \quad &F(\mu) &= \Eb[\psi_W(\mu)],\quad  \text{ where }W\sim Pareto(\alpha). 
\end{align*}
\end{definition} 
In particular, notice that, for $\mu \in \mathcal{M}_p((0,\infty])$, we may rewrite
\begin{equation}\label{F}
\psi_w(\mu) = \prod_{x \in \mathrm{supp}(\mu)} (1 - \min(xw,1)),
\quad
\text{and}
\quad
F(\mu) = \Eb\left[\prod_{x \in \mathrm{supp}(\mu)} (1 - \min(xW,1))\right].
\end{equation}

\begin{lemma}\label{coredeltaexists}
Let $\delta \in (0,1/2)$. As $n\to\infty$, it holds
\begin{equation*}
    \mathbb{P}(\mathcal{W}_\delta\text{ is a non-empty clique})\to 1.
\end{equation*} 
\end{lemma}

\begin{proof}
By definition $|\WW_\delta|\sim \text{Bin}(n,n^{-(1-\delta)})$, so trivially $\mathbb{E}[|\WW_\delta|]=n^{\delta}$ and $\mathbb{P}(|\WW_\delta|=0)=(1-n^{\delta-1})^n\to 0$. Moreover, if $i,j \in \WW_\delta$, 
\begin{equation*}
p_{ij}=\min(\varepsilon W_iW_j,1)\geq \min\left(\lambda^{\frac{1}{\alpha}}n^{-\frac{1}{\alpha}}n^{\frac{2(1-\delta)}{\alpha}},1\right)=\min\left(\lambda^{1/\alpha}n^{\frac{1-2\delta}{\alpha}},1\right)=1
\end{equation*}
for $n$ sufficiently large, since $\delta < 1/2$. Thus, it trivially holds
\begin{equation*}
    \mathbb{P}(\mathcal{W}_\delta \text{ is a non-empty clique}) = \mathbb{E}\left[\prod_{i,j \in \mathcal{W}_\delta \atop i \neq j}p_{ij}\text{ }\Biggl | \text{ }\mathcal{W}_\delta \neq \emptyset\right]\mathbb{P}(\mathcal{W}_\delta \neq \emptyset)=\mathbb{P}(\mathcal{W}_\delta \neq \emptyset)\to1
\end{equation*}
as $n\to \infty$, which completes the proof.
\end{proof}

\begin{lemma}\label{isolatedornot}
Let $\delta \in (0,1/2)$. The following holds
\begin{enumerate}
\item $\mathbb{P}(1 \text{ is isolated})\to \xi_\lambda$,
\item $\mathbb{P}(1 \in \widebar{\WW}_\delta)\to 1-\xi_\lambda$,
\end{enumerate}
where \begin{equation*}
\xi_\lambda := \Eb\left[e^{-\frac{\lambda}{1-\alpha}W^\alpha}\right] = \frac{\lambda}{1-\alpha} \Gamma\left(-1, \frac{\lambda}{1-\alpha}\right).
\end{equation*} 
In particular, it holds
\begin{equation*}
    \mathbb{P}(\text{1 is either isolated or }1\in \widebar{\WW}_\delta)\to 1
\end{equation*}
as $n\to \infty$.
\end{lemma}

\begin{proof}[Proof of Lemma \ref{isolatedornot}]
  $(i)$ follows from Theorem \ref{thm:expecteddegree}, by the fact that $D_n(1) \xrightarrow{d} \mathrm{Poisson}\left(\frac{\lambda}{1-\alpha}W_1^\alpha\right)$; indeed, as $n\to \infty$ 
\begin{equation*}
\mathbb{P}(\text{1 is isolated})=\mathbb{P}(D_n(1)=0) \to \Eb\left[e^{-\frac{\lambda}{1-\alpha}W^\alpha}\right].
\end{equation*}
For $(ii)$, we observe that $\mathbb{P}(i \in \widebar{\WW}_\delta)=1-\mathbb{P}(1 \notin \widebar{\WW}_\delta)$ with
\begin{align*}
\mathbb{P}(1 \notin \widebar{\WW}_\delta)&=\mathbb{E}\left[\ind(W_1 < n^{\frac{1-\delta}{\alpha}})\prod_{j \geq 2}\left(1-\min(\varepsilon W_1W_j,1)\ind(W_j > n^{\frac{1-\delta}{\alpha}})\right)\right]\\
&=\mathbb{E}\left[\ind(W_1 < n^{\frac{1-\delta}{\alpha}})\mathbb{E}_{W_1}\left[ 1-\min(\varepsilon W_1W_2,1)\ind(W_2 > n^{\frac{1-\delta}{\alpha}}) \right]^{n-1} \right].
\end{align*}
We now apply the Dominated Convergence Theorem to conclude. For a fixed value of $w\geq 1$, it holds that $\ind(w < n^{\frac{1-\delta}{\alpha}})\to 1$ as $n \to \infty$. Therefore, it remains to show that
\begin{align*}
\mathbb{E}_{w}\left[ 1-\min(\varepsilon wW_2,1)\ind(W_2 > n^{\frac{1-\delta}{\alpha}}) \right]^{n-1}\to e^{-\frac{\lambda}{1-\alpha}w^\alpha}.
\end{align*}
We compute:
\begin{align*}
\mathbb{E}_{w}\left[ \min(\varepsilon wW_2,1)\ind(W_2 > n^{\frac{1-\delta}{\alpha}})\right] &= \alpha \int_{n^{\frac{1-\delta}{\alpha}}}^{\infty}\frac{\min(\varepsilon w x,1)}{x^{\alpha+1}}dx\\
&= \alpha \int_{n^{\frac{1-\delta}{\alpha}}}^{\frac{1}{\varepsilon w}}\frac{\varepsilon w x}{x^{\alpha+1}}dx + \alpha \int_{\frac{1}{\varepsilon w}}^{\infty} \frac{1}{x^{\alpha+1}}dx\\
&= A+B.
\end{align*}
We observe that, for every fixed $w\geq1$, the term $A$ is well-defined for $n$ sufficiently large, since
\begin{equation*}
\frac{1}{\varepsilon w} > n^{\frac{1-\delta}{\alpha}} \iff w < n^{-\frac{1-\delta}{\alpha}}n^{\frac{1}{\alpha}}\lambda^{-\frac{1}{\alpha}} \iff w < n^{\frac{\delta}{\alpha}}\lambda^{-\frac{1}{\alpha}}.
\end{equation*}
Evaluating the integrals, we find:
\begin{equation*}
B = \left[-x^{-\alpha} \right]^{\infty}_{\frac{1}{\varepsilon w}} = (\varepsilon w)^\alpha = \frac{\lambda}{n}w^\alpha,
\end{equation*}
and
\begin{align*}
A &= \varepsilon w \left[ \frac{\alpha}{1-\alpha}x^{-\alpha+1} \right]^{\frac{1}{\varepsilon w}}_{n^{\frac{1-\delta}{\alpha}}} = \frac{\alpha}{1-\alpha}\varepsilon w \left((\varepsilon w)^{\alpha-1}-n^{\frac{1-\alpha}{\alpha}(1-\delta)} \right)\\
&= \frac{\alpha}{1-\alpha}\left( \frac{\lambda w^\alpha}{n} - \lambda^{\frac{1}{\alpha}}n^{\frac{\alpha-1}{\alpha}\delta-1}w \right) = \frac{\alpha}{1-\alpha}\frac{\lambda w^\alpha}{n} + o\left(\frac 1n\right).
\end{align*}
Since $n(A+B) \to \frac{\lambda}{1-\alpha}w^\alpha$, the left-hand side of \eqref{bla} becomes
\begin{equation} \label{bla}
\left(1-\frac{\lambda w^\alpha}{1-\alpha}\frac{1}{n}(1+o(1))\right)^{n-1},
\end{equation}
which yields the desired result. Moreover, we observe that
\begin{equation}\label{isolatoono2}
\mathbb{P}(\text{1 is either isolated or 1}\in \widebar{\WW}_\delta)=\mathbb{P}(\text{1 is isolated})+\mathbb{P}(1 \in \widebar{\WW}_\delta)+o(1)\to 1.
\end{equation}
\end{proof}
An easy consequence of the previous lemma are the two following asymptotics:
\begin{enumerate}    
\item $\lim_{n \to \infty} \frac{\mathbb{E}[\mathcal{I}]}{n} = \xi_\lambda$,  
\item $\lim_{n \to \infty}\frac{\mathbb{E}[|\widebar{\WW}_\delta|]}{n} = 1 - \xi_\lambda$.\end{enumerate}
\begin{corollary}\label{isolatedornot2}
Let $\delta \in (0,1/2)$. It holds
\begin{equation*}
    \frac{|\mathcal{C}_{max}|}{n}-\frac{|\widebar{\WW}_\delta|}{n}\xrightarrow{\text{ }\P\text{ }}0\quad\quad\text{as }n\to\infty.
\end{equation*}
\end{corollary}

\begin{proof}
Since the entire graph $\mathbb{G}_n(\alpha,\varepsilon)$ has at least one edge w.h.p., trivially $\mathcal{C}_{max}$ does not consist of an isolated vertex w.h.p. Therefore, we observe that
\begin{equation*}
\mathcal{C}_{max}\setminus\widebar{\WW}_\delta\subseteq\{i \in [n]: i \notin \widebar{\WW}_\delta, i \text{ not isolated}  \}
=\{i \in [n]: i \in \widebar{\WW}_\delta \text{ or } i \text{ isolated}  \}^\cc.
\end{equation*}
Thus, from equation \eqref{isolatoono2}, as $n\to \infty$ we have
\begin{equation*}
    \mathbb{E}[|\mathcal{C}_{max}\setminus\widebar{\WW}_\delta|]\leq n-n \mathbb{P}(\text{1 is either isolated or 1}\in\widebar{\WW}_\delta )=n-n(1-o(1)).
\end{equation*}
Using Markov's inequality, for each $\eta>0$ we obtain
\begin{align}\label{cmaxx}
    \mathbb{P}\left(\frac{|\mathcal{C}_{max}\setminus\widebar{\WW}_\delta|}{n}\geq \eta\right)
    \leq\frac{n-n (1-o(1))}{n\eta}\to0,
\end{align}
i.e. $|\mathcal{C}_{max}\setminus\widebar{\WW}_\delta|/n\xrightarrow{\P} 0$. 
Now we show that $\widebar{\WW}_\delta \subseteq \mathcal{C}_{\max}$ w.h.p.. If, by contradiction, $\widebar{\WW}_\delta \not\subseteq \mathcal{C}_{\max}$, by the fact that $\widebar{\WW}_\delta$ is connected w.h.p., we have that $\mathcal{C}_{\max} \subseteq \mathcal{C}_{\max} \setminus \widebar{\WW}_\delta$ w.h.p.. Thus, $|\mathcal{C}_{\max}|/n \xrightarrow{\mathbb{P}} 0$, as seen in \eqref{cmaxx}. Moreover, $|\widebar{\WW}_\delta|/n \le |\mathcal{C}_{\max}|/n \xrightarrow{\mathbb{P}} 0$. Since it is a bounded quantity, its expectation goes to $0$ as well. However, this contradicts the fact that $\mathbb{E}[|\widebar{\WW}_\delta|]/n \to 1 - \xi_\lambda > 0$. Hence, $\widebar{\WW}_\delta \subseteq \mathcal{C}_{\max}$ w.h.p., which implies $|\mathcal{C}_{\max}\setminus\widebar{\WW}_\delta| = |\mathcal{C}_{\max}| - |\widebar{\WW}_\delta|$ w.h.p. Thus, we have the desired result.
\end{proof}


\subsection{Poisson Point Process approximation}\label{PPPconvergence}

Recall the point measure $\Pi_n$
from Definition~\ref{defofPin}. Recall as well the PPP $\Pi$ from Definition~\ref{def:PPP_limit}, with intensity measure $\nu(dt) = \lambda \alpha t^{-(\alpha+1)} dt$. 

\begin{lemma}\label{PPPCOLORLLARY}
As $n\to \infty$, it holds
\begin{equation*}    
\Pi_n \Rightarrow \Pi,
\end{equation*}
where $\Rightarrow $ means convergence in distribution of a sequence of random variables in the space $\mathcal{M}_p((0,\infty])$, endowed with the vague topology.

\end{lemma}

\begin{proof}
By construction, $\Pi_n$ is an element of $\mathcal{M}_p((0,\infty])$. To prove the result, it is sufficient to show the convergence of the corresponding Laplace functionals $\psi_{\Pi_n}(f) \to \psi_\Pi(f)$ for any $f \in C_0^+((0,\infty])$ (see \cite[Proposition~3.19]{resnick1987extreme}). In this context $C_0^+((0,\infty])$ is the space of continuous, non negative, compactly supported function on $(0,\infty]$. 

We have
\begin{align*}
 \psi_{\Pi_n}(f) &= \mathbb{E}\left[e^{-\sum_{i \in [n]}f(\eps W_i)\ind(i \in \WW_\delta)} \right]=\mathbb{E}\left[e^{-f(\eps W_1)\ind(1 \in \WW_\delta)} 	\right]^n\\
 &=\left(1-\mathbb{E}\left[(1-e^{-f(T_1)})\ind(1 \in \WW_\delta)\right]\right)^n=(1-\mathbb{E}[h_n(W_1)])^n
\end{align*}
for any $f \in C_0^+((0,\infty])$. Thus,
\begin{align*}
n\mathbb{E}[h_n(W_1)]&=n\alpha \int_{n^{\frac{1-\delta}{\alpha}}}^{\infty}\frac{1-e^{-f(\varepsilon t)}}{t^{\alpha+1}}dt\\
&=\alpha n\varepsilon^\alpha \int_{\varepsilon n^{\frac{1-\delta}{\alpha}}}^{\infty}\frac{1-e^{-f(y)}}{dy^{\alpha+1}}dy\to\alpha \lambda \int_0^{\infty}\frac{1-e^{-f(y)}}{y^{\alpha+1}}dy
\end{align*}
since $\varepsilon n^{\frac{1-\delta}{\alpha}} =\lambda^{\frac{1}{\alpha}}n^{\frac{-\delta}{\alpha}} \to 0$. We notice that the right hand side is finite, since the support of $f$ is bounded away from the origin. Using the standard limit $(1-x_n/n)^n\to e^{-x}$, where $x_n\to x$, we conclude
\begin{equation*}
\psi_{\Pi_n}(f)\to \exp\left(-\int_0^\infty (1-e^{-f(t)})\nu(dt) \right)
\end{equation*}
which is the Laplace functional of a PPP($\nu$).
\end{proof}

It is worth noting that if $\Pi_n$ would have been defined using the weights of all vertices $i\in [n]$, instead of only those with $i\in \WW_{\delta}$ the limit $\Pi$ would have been the same in this topology, which indeed do not consider everything that accumulates in the origin (the case of all small weights under the $\eps$ scaling).

\subsection{Continuity of the limiting functional}\label{convergenzafunzionali}
Recall $\psi_w$ and $F$ as in Definition \ref{psieF}. The goal of this section is to show $F(\Pi_n)\to F(\Pi)$, as $n\to\infty$. This result is stated in Lemma \ref{continuitàdiF}. We notice that $\psi_w$ and $F$ are not continuous functionals on the whole space $\mathcal{M}_p((0,\infty])$ under the vague topology. So, we want to approximate $\psi_w$ and $F$ with a sequence of vaguely continuous functionals $\psi_\eta$ (and $F_\eta$) by reducing the support of $f_w$ away from the boundaries $0$ and $w^{-1}$. Let us start with the following definition.

\begin{definition}[Vaguely continuous approximation]\label{def_approx}
For a fixed $w\geq1$, let $f_w(x)=-\ln(1-wx)\mathbb{1}(x\in [0,w^{-1}))$. For any $\eta > 0$ sufficiently small, we define the truncated and continuous approximation $f_w^{(\eta)} \in C_0^+((0, \infty])$ as follows:
\begin{equation*}
f_w^{(\eta)}(x) = 
\begin{cases} 
0 & \text{if } x \leq \frac{\eta}{2} \text{ or } x \geq w^{-1}-\frac{\eta}{2}, \\
f_w(x) & \text{if } \eta \leq x \leq w^{-1}-\eta, \\
\text{linear interpolation} & \text{otherwise,}
\end{cases}
\end{equation*}
such that $0 \leq f_w^{(\eta)}(x) \leq -\ln (1-wx)$ for all $x \in (0, w^{-1})$ and $f_w^{(\eta)}$ is continuous with $\text{supp}(f_w^{(\eta)}) \subseteq [\frac{\eta}{2}, w^{-1}-\frac{\eta}{2}]$ and coincides with $f_w$ in the set $[\eta,w^{-1}-\eta]$.
Accordingly, we define the approximated functionals:
\begin{enumerate}
    \item $\psi_w^{(\eta)}(\mu) = \ind(\mu[w^{-1},\infty]=0)\exp\left(-\int_0^{w^{-1}}f_w^{(\eta)}(x) \mu(dx)\right)$,
    \item $F_\eta(\mu) = \Eb[\psi_W^{(\eta)}(\mu)]$.
\end{enumerate}
\end{definition}

\begin{lemma}\label{B0}
As $n \to \infty$, it holds that $F_\eta(\Pi_n) \xrightarrow{d} F_\eta(\Pi)$.
\end{lemma}

\begin{proof}
By Definition \ref{def_approx}, the approximating function $f_w^{(\eta)}$ is continuous and has compact support strictly contained in $(0, \infty]$. This means that, if $\mu_n\to \mu$ vaguely, $\mu_n(f_w^{(\eta)})\to\mu(f_w^{(\eta)})$ as $n\to \infty$. We, then consider
\begin{equation*}
\psi_w^{(\eta)}(\mu_n) = \ind(\mu_n[w^{-1},\infty]=0)\exp\left(-\int_0^{w^{-1}}f_w^{(\eta)}(x) \mu_n(dx)\right).
\end{equation*}
We observe that the indicator function $\ind(\mu_n[w^{-1},\infty]=0)$ is not always continuous with respect to the vague topology, as it is sensitive to the presence of atoms at the boundary point $w^{-1}$.
Then, for each $w \geq 1$, the map $\mu \mapsto \psi_w^{(\eta)}(\mu)$ is continuous at $\mu$ whenever $\mu(\{w^{-1}\}) = 0$. Since $\mu$ has at most countably many atoms, the set of $w$ such that $\mu(\{w^{-1}\}) > 0$ is at most countable. Hence, 
\[
\psi_w^{(\eta)}(\mu_n) \to \psi_w^{(\eta)}(\mu)
\]
for almost every $w \ge 1$ with respect to the Lebesgue measure. Since the Pareto distribution of $W$ is absolutely continuous with respect to the Lebesgue measure, this convergence holds almost surely with respect to $W$. Moreover, $0 \le \psi_w^{(\eta)} \le 1$, so by the Dominated Convergence Theorem we obtain
\[
F_\eta(\mu_n) \to F_\eta(\mu).
\]
since $F_\eta(\mu)=\Eb[\psi_w^{(\eta)}(\mu)]$. This shows that $F_\eta$ is continuous with respect to the vague topology. By Lemma \ref{PPPCOLORLLARY}, the point process $\Pi_n$ converges weakly to $\Pi$ with respect to the vague topology on $\mathcal{M}_p((0,\infty])$. Applying the Continuous Mapping Theorem (see, e.g., \cite{resnick1987extreme}, Section $3.5$, page $152$) yields
\[
F_\eta(\Pi_n) \xrightarrow{d} F_\eta(\Pi),
\]
as desired.
\end{proof}

\begin{lemma}\label{B1}
As $\eta\to0$, it holds that $F_\eta(\Pi) \xrightarrow{\Pb} F(\Pi)$.
\end{lemma}

\begin{proof}
For every $w \geq 1$ and $\eta>0$ sufficiently small, let
\[
R_{w,\eta}=\int_0^\eta f_w(x) \Pi(dx)
\quad\text{and}\quad
S_{w,\eta}=\int_{w^{-1}-\eta}^{w^{-1}} f_w(x) \Pi(dx).
\]
By Definition \ref{def_approx}, we have
\begin{equation}\label{C.1}
0 \le \Pi(f_w)-\Pi(f_w^{(\eta)}) \le R_{w,\eta}+S_{w,\eta}.
\end{equation}
  We first show that both $R_{w,\eta}$ and $S_{w,\eta}$ vanish in expectation as $\eta \to 0$.
Using the Campbell formula for a Poisson Point Process with intensity $\nu$, we obtain
\[
\Eb[R_{w,\eta}]
= \int_0^\eta f_w(x)\,\nu(dx)
= \alpha \lambda \int_0^\eta \frac{-\ln(1-xw)}{x^{\alpha+1}}dx.
\]
Since $-\ln(1-xw)\le C xw$ for $x$ small, we get
\[
\Eb[R_{w,\eta}]
\le C w \alpha \lambda \int_0^\eta x^{-\alpha}dx
= C w \frac{\alpha\lambda}{1-\alpha} \eta^{1-\alpha} \to 0
\]
as $\eta \to 0$, since $\alpha <1$. Similarly, performing the change of variable $u=1-xw$, we obtain
\[
\Eb[S_{w,\eta}]
= \alpha \lambda w^\alpha \int_0^{w\eta} \frac{-\ln(u)}{(1-u)^{\alpha+1}} du.
\]
For $\eta$ small, the denominator is bounded by a constant and thus
\[
\Eb[S_{w,\eta}]
\le C \int_0^{w\eta} -\ln(u)\,du
= C\big[ u - u\ln u \big]_{0}^{w\eta}
\to 0.
\]
  Therefore, from \eqref{C.1}
\[
\Eb[|\Pi(f_w)-\Pi(f_w^{(\eta)})|] \to 0.
\]
  Using that $|e^{-x}-e^{-y}|\le |x-y|$ for $x,y\ge 0$, we deduce $|\psi_w(\Pi)-\psi_w^{(\eta)}(\Pi)|\leq|\Pi(f_w)-\Pi(f_w^{(\eta)})|$, that implies
\begin{equation}\label{C.2}
\Eb[|\psi_w(\Pi)-\psi_w^{(\eta)}(\Pi)|] \to 0
\quad \text{for every } w\ge 1
\end{equation}
as $\eta\to0$. We now pass to the functional $F$. Observe that
\[
|F(\Pi)-F_\eta(\Pi)|
= \left| \Eb[\psi_W(\Pi)-\psi_W^{(\eta)}(\Pi)\mid \Pi] \right|
\le \Eb[|\psi_W(\Pi)-\psi_W^{(\eta)}(\Pi)| \mid \Pi].
\]
Taking expectations and applying Fubini's theorem, we obtain
\[
\Eb[|F(\Pi)-F_\eta(\Pi)|]
\le \Eb \left[ \Eb[|\psi_W(\Pi)-\psi_W^{(\eta)}(\Pi)|\mid W] \right].
\]

  Since $0 \le \psi_w, \psi_w^{(\eta)} \le 1$ and from \eqref{C.2}, we conclude by the Dominated Convergence Theorem that
\[
\Eb[|F(\Pi)-F_\eta(\Pi)|] \to 0.
\]
  Finally, by Markov's inequality,
\[
F_\eta(\Pi) \xrightarrow{\Pb} F(\Pi).
\]
\end{proof}

\begin{lemma}\label{B2}
It holds
\begin{equation*}
\lim_{\eta \to 0} \limsup_{n \to \infty} \mathbb{P}( |F(\Pi_n)-F_\eta(\Pi_n)| \geq \theta) = 0
\end{equation*}
for each positive $\theta$.
\end{lemma}

\begin{proof}[Proof of Lemma \ref{B2}]
Very similar to the proof of Lemma \ref{B1} but here we also have to take care of uniformity on $n$. $\forall w\geq 1$, $\forall n \in \mathbb{N}$ and $\forall \eta>0$ sufficiently small, let
\begin{equation*}
R_{w,\eta,n}=\int_0^\eta f_w(x) \Pi_n(dx)\quad\quad \text{ and } \quad \quad
S_{w,\eta,n}=\int_{w^{-1}-\eta}^{w^{-1}} f_w(x) \Pi_n(dx).
\end{equation*}
Using that $f_w(x)=-\ln(1-xw)\leq C xw$ for $x$ near 0,  we get
\begin{align}\label{B.1.A}
\notag\mathbb{E}[R_{w,\eta,n}]&=\mathbb{E}\left[\sum_{i \in [n]}f_w(\varepsilon W_i)\ind(\varepsilon W_i\leq \eta) \right]\leq \mathbb{E}\left[\sum_{i \in [n]} C \varepsilon W_i w \ind(\varepsilon W_i<\eta)\right]\\
\notag&=C\varepsilon w n \mathbb{E}[W\ind(\varepsilon W<\eta)]=C\varepsilon w n \alpha \int_1^{\frac{\eta}{\varepsilon}} \frac{x}{x^{\alpha+1}}dx\\
\notag &\leq C\varepsilon w n \alpha \int_0^{\frac{\eta}{\varepsilon}} \frac{x}{x^{\alpha+1}}dx=C \varepsilon^\alpha w n  \frac{\alpha}{1-\alpha}\eta^{1-\alpha}\\
&= C w\lambda \frac{\alpha}{1-\alpha} \eta^{1-\alpha}\to 0
\end{align}
since $\alpha <1$. Similarly, we get:
\begin{align}\label{B.1.B}
\notag\mathbb{E}[S_{w,\eta,n}]&=\mathbb{E}\left[ \sum_{i \in [n]} f_w(\varepsilon W_i)\ind(\varepsilon W_i \in [w^{-1}-\eta,w^{-1}])\right]\\
\notag&=n\alpha\int_{w^{-1}\varepsilon^{-1}-\eta\varepsilon^{-1}}^{w^{-1}\varepsilon^{-1}}\frac{-\ln(1-\varepsilon xw)}{x^{\alpha+1}}dx\\
\notag&=n\alpha \int_0^{\eta w}\frac{-\ln(y)}{\left(\frac{1-y}{\varepsilon w}\right)^{\alpha+1}}\frac{1}{\varepsilon w}dy
=\alpha n \varepsilon^{\alpha}w^{\alpha}\int_0^{\eta w}\frac{-\ln(y)}{(1-y)^{\alpha+1}}dy\\
\notag&\leq C \alpha \lambda w^\alpha \int_0^{\eta w}-\ln(y) dy
=C \alpha \lambda w^{\alpha}\left[y-y\ln y \right]^{\eta w}_0\\
&=C \alpha \lambda w^{\alpha}(\eta w-\eta w\ln (\eta w) )\to 0\quad\quad \text{ as }\eta\to0.
\end{align}
In particular, this implies
\begin{equation*}
\sup_{n\to\infty} \mathbb{E}[R_{w,\eta,n}+S_{w,\eta,n}]\leq g_w(\eta)\to 0 \quad\quad\text{ as } \eta\to 0
\end{equation*}
where $g_w(\eta)$ does not depend on $n$ and is defined appropriately from \eqref{B.1.A} and \eqref{B.1.B}. This implies that, uniformly on $n$,
\begin{align*}\label{B.1.0}
\notag\mathbb{E}[|\Pi_n(f_w)-\Pi_n(f_w^{(\eta)})|]&=\mathbb{E}[\Pi_n(f_w)-\Pi_n(f_w^{(\eta)})]\leq\mathbb{E}[R_{w,\eta,n}+S_{w,\eta,n}]\leq g_w(\eta)\to 0 
\end{align*}
as $\eta\to0$. Since $|\psi_w(\Pi_n)-\psi_w^{(\eta)}(\Pi_n)|\leq |e^{-\Pi_n(f_w)}-e^{-\Pi_n(f_w^{(\eta)})} |\leq|\Pi_n(f_w)-\Pi_n(f_w^{(\eta)})|$, we also have
\begin{equation}\label{B.1.1}
\limsup_{n\to \infty}\mathbb{E}[|\psi_w(\Pi_n)-\psi_w^{(\eta)}(\Pi_n)|]\leq g_w(\eta)\to0
\end{equation}
which, using Markov's inequality, implies
\begin{align*}
\lim_{\eta\to0}\limsup_{n\to \infty}\mathbb{P}(|\psi_w(\Pi_n)-\psi^{(\eta)}_w(\Pi_n)|\geq \theta)\leq\frac{g_w(\eta)}{\theta}\to 0,
\end{align*}
for every $\theta>0$. We now pass to the functional $F$. We observe that
\[
|F(\Pi_n)-F_\eta(\Pi_n)|
= \left| \mathbb{E}[\psi_W(\Pi_n)-\psi_W^{(\eta)}(\Pi_n)| \Pi_n] \right|
\le \mathbb{E}[|\psi_W(\Pi_n)-\psi_W^{(\eta)}(\Pi_n)| \mid\Pi_n].
\]
Taking expectations and using Fubini's theorem, we obtain
\[
\mathbb{E}[|F(\Pi_n)-F_\eta(\Pi_n)|]
\le \mathbb{E}\left[\mathbb{E}[|\psi_W(\Pi_n)-\psi_W^{(\eta)}(\Pi_n)| \mid W]\right].
\]
Since $0 \le \psi_W, \psi_W^{(\eta)} \le 1$, the conditional expectation is bounded. Thus, by the Reverse Fatou's Lemma, 
\begin{equation}\label{B.1.2}
\limsup_{n\to\infty}\mathbb{E}[|F(\Pi_n)-F_\eta(\Pi_n)|]
\le \mathbb{E}\left[\limsup_{n\to\infty}\mathbb{E}[|\psi_W(\Pi_n)-\psi_W^{(\eta)}(\Pi_n)| \mid W]\right].
\end{equation}
Using the Dominated Convergence Theorem and \eqref{B.1.1}, the right-hand side of \eqref{B.1.2} vanishes as $\eta \to 0$. Hence, by Markov's inequality,
\[
\lim_{\eta \to 0} \limsup_{n \to \infty}
\mathbb{P}(|F(\Pi_n)-F_\eta(\Pi_n)| \geq \theta) = 0.
\]
\end{proof}

\begin{lemma}\label{continuitàdiF}
    Let $\Pi_n$ as in \ref{defofPin}. Let $\Pi$ be a Poisson Point Process with intensity measure $\nu(t)=\alpha \lambda t^{-\alpha-1}dt$. Let $F$ be a functional in $\mathcal{M}_p((0,\infty])$, as defined in \eqref{F}. Then it holds:
    \begin{equation*}
    F(\Pi_n)\xrightarrow{d} F(\Pi).
    \end{equation*}
\end{lemma}

\begin{proof}
    Let $F_\eta$ as in Definition \ref{def_approx}. We have

\begin{align*}
    &F_\eta(\Pi_n) \xrightarrow[n\to\infty]{d} F_\eta(\Pi)\quad\quad\quad &\text{from Lemma \ref{B0}},\\
    &F_\eta(\Pi) \xrightarrow[\eta\to0]{\Pb} F(\Pi)\quad\quad &\text{from Lemma \ref{B1}},\\
    &\forall \theta>0\quad\lim_{\eta \to 0} \limsup_{n \to \infty} \mathbb{P}( |F(\Pi_n)-F_\eta(\Pi_n)| \geq \theta) = 0\quad\quad\quad &\text{from Lemma \ref{B2}}.
\end{align*}
Thus, we can directly apply Billingsley's Theorem (see Theorem 4.2 in \cite{billingsley2013convergence}, 1968) to conclude. 
\end{proof}

\subsection{Lack of self-averaging}\label{proofgiantsize}

In this subsection, we characterize the random limit of \(|\widebar{\WW}_\delta|/n\) through its complement. Using the Poisson point process \(\Pi\) that describes the asymptotic rescaled weights in \(\WW_\delta\) (Lemma~\ref{PPPCOLORLLARY}), we express \(|\widebar{\WW}_\delta^{\mathrm c}|/n\) as the probability that a typical vertex \(W\sim\mathrm{Pareto}(\alpha)\) is not connected to \(\operatorname{supp}(\Pi)\). The result is stated in Lemma \ref{lemmaginatsize}, after the concentration estimates of Lemmas \ref{lemmaA} and \ref{lemmaB}. We use the notation introduced at the beginning of Section \ref{sec:giantsize}, in particular \(\Pi_n,\Pi\) and the functionals \(\psi_w,F\).

\begin{definition}\label{hatp}
Given the weights $\mathbf{W}=(W_i)_{i\in[n]}$, we define
\begin{equation*}
\hat{p}_i(\mathbf{W}) = \ind(i \notin \WW_\delta) \prod_{j \in \WW_\delta} \left( 1 - \min(\varepsilon W_i W_j, 1) \right).
\end{equation*}
\end{definition}
For each $i\in [n]$, the term $\hat{p}_i(\mathbf{W})$ in the above definition is the conditional probability for vertex $i$ not to be connected to the $\delta$-Core $\WW_\delta$. Sometimes we will simply write $\hat{p}_i$ instead of $\hat{p}_i(\mathbf{W})$. In the following lemma we prove that the sum of such terms is very close to the cardinality of the set $\widebar{\WW}_\delta^{\mathrm c}$.

\begin{lemma}\label{lemmaA}
Let $\delta\in(0,1/2)$. As $n \to \infty$, it holds:
\begin{equation*}
\frac{|\widebar{\WW}_\delta^\cc|}{n} - \frac{1}{n}\sum_{i \in [n]} \hat{p}_i(\mathbf{W}) \xrightarrow{\mathbb{P}} 0.
\end{equation*}
\end{lemma}

\begin{proof}
Let $X_i=\ind(i \nleftrightarrow \WW_\delta)$ for all $i \in [n]$. Conditionally on $\mathbf{W}$, the indicators $(X_i)_{i\in[n]}$ are independent Bernoulli random variables with  success probabilities \((\hat{p}_i(\mathbf{W}))_{i\in [n]}\). We observe that:
\begin{equation*}
\mathbb{E}\left[ \frac{|\widebar{\WW}_\delta^\cc|}{n} \mid \mathbf{W} \right] =
\frac{1}{n}\sum_{i \in[n]} \mathbb{E}[X_i \mid \mathbf{W}]=
 \frac{1}{n} \sum_{i=1}^n  \hat{p}_i.
\end{equation*}
Moreover
\begin{align*}
\mathbb{V}\left[\frac{|\widebar{\WW}_\delta^\cc|}{n}\mid \mathbf{W} \right]&=\frac{1}{n^2}\sum_{i \in [n]}\mathbb{V}(X_i\mid \mathbf{W})+\frac{1}{n^2}\sum_{i,j \in [n] \atop i\neq j}Cov(X_i,X_j\mid \mathbf{W})\\
&= \frac{1}{n^2}\sum_{i \in [n]}\mathbb{V}(X_i \mid \mathbf{W})= \frac{1}{n^2}\sum_{i \in [n]} \hat{p}_i(1- \hat{p}_i)\leq \frac{1}{n}\to 0
\end{align*}
So it brings to
\begin{align*}
\mathbb{P}\left(\left|\frac{|\widebar{\WW}_\delta^\cc|}{n}-\frac{1}{n}\sum_{i} \hat{p}_i\right|> \eta \right)&= \mathbb{E}\left[\mathbb{P}\left(\left|\frac{|\widebar{\WW}_\delta^\cc|}{n}-\frac{1}{n}\sum_{i} \hat{p}_i\right|> \eta \mid \mathbf{W}\right) \right]\\
&\leq \frac{\mathbb{E}\left[\mathbb{V}(\frac{|\widebar{\WW}_\delta^\cc|}{n} \mid \mathbf{W}) \right]}{\eta^2}\leq \frac{1}{n\eta^2}\to 0
\end{align*}
for every $\eta>0$.
\end{proof}

\begin{lemma}\label{lemmaB}
 Let $\delta\in(0,1/2)$. As $n \to \infty$, it holds:
\begin{equation*}
\frac{1}{n}\sum_{i \in [n]} \hat{p}_i(\mathbf{W}) -F(\Pi_n)\xrightarrow{\mathbb{P}} 0.
\end{equation*}
\end{lemma}

\begin{proof}
  Conditionally on $\mathbf{W}$, we can write $\hat{p}_i = \hat{p}_i(\mathbf{W})$ as follows:
\begin{align*}
\hat{p}_i &= \ind(i \notin \WW_\delta) \prod_{j \in \WW_\delta} (1 - \min(\varepsilon W_i W_j, 1)) \\
&= \ind(i \notin \WW_\delta) \ind(\Pi_n[W_i^{-1}, \infty) = 0) \exp\left( \int_0^{W_i^{-1}} \ln(1-\eps xW_i) \Pi_n(dx) \right) \\
&= \ind(i \notin \WW_\delta) \psi_{W_i}(\Pi_n).
\end{align*}
Let $W \sim \text{Pareto}(\alpha)$ be an auxiliary random variable and independent of $\Pi_n$. For better readability, let
\begin{equation*}
Z_n = \frac{1}{n} \sum_{i \in [n]} \hat{p}_i(\mathbf{W}) - F(\Pi_n).
\end{equation*}
We have:
\begin{align}\label{A0}
\mathbb{E}[Z_n] &= \frac{1}{n} \sum_{i \in [n]} \mathbb{E}[\hat{p}_i(\mathbf{W})] - \mathbb{E}[F(\Pi_n)] = \mathbb{E}[\hat{p}_1(\mathbf{W})] - \mathbb{E}[F(\Pi_n)] \\
\notag &= \mathbb{E}[\ind(1 \notin \WW_\delta) \psi_{W_1}(\Pi_n)] - \mathbb{E}[\psi_W(\Pi_n)] \\
\notag &= \mathbb{E}[\ind(1 \notin \WW_\delta) (\psi_{W_1}(\Pi_n) - \psi_W(\Pi_n))] - \mathbb{E}[\ind(1 \in \WW_\delta) \psi_W(\Pi_n)] \\
\notag &= A - B.
\end{align}
Since $\ind(1 \in \WW_\delta) = \ind(W_1 \geq n^{\frac{1-\delta}{\alpha}}) \to 0$ for every fixed value of $W_1$ and $\psi_W(\Pi_n) \in (0,1)$, the Dominated Convergence Theorem implies $B \to 0$.
For the term $A$, observe that on the event $\{1 \notin \WW_\delta\}$, the weight $W_1$ does not contribute to the point process $\Pi_n$. We can thus substitute $\Pi_n$ with the point process
\begin{equation*}
\hat{\Pi}_{n-1} = \sum_{\substack{j \in \WW_\delta \\ j \geq 2}} \delta_{\varepsilon W_j},
\end{equation*}
which is independent of both $W_1$ and $W$. We can then write:
\begin{align*}
A &= \mathbb{E}[\ind(1 \notin \WW_\delta)(\psi_{W_1}(\hat{\Pi}_{n-1}) - \psi_W(\hat{\Pi}_{n-1}))] \\
&= \mathbb{E}[\psi_{W_1}(\hat{\Pi}_{n-1}) - \psi_W(\hat{\Pi}_{n-1})] - \mathbb{E}[\ind(1 \in \WW_\delta)(\psi_{W_1}(\hat{\Pi}_{n-1}) - \psi_W(\hat{\Pi}_{n-1}))] \\
&= A_1 - A_2 \to 0,
\end{align*}
where $A_1 = 0$ by symmetry and $A_2 \to 0$ by the Dominated Convergence Theorem. Thus, $\mathbb{E}[Z_n] \to 0$. 
  We now consider the second moment to establish concentration:
\begin{align}\label{Z_n^2}
\mathbb{E}[Z_n^2] &= \frac{1}{n^2} \sum_{i \in [n]} \mathbb{E}[\hat{p}_i^2] + \frac{1}{n^2} \sum_{i \neq j} \mathbb{E}[\hat{p}_i \hat{p}_j] - \frac{2}{n} \sum_{i \in [n]} \mathbb{E}[\hat{p}_i F(\Pi_n)] + \mathbb{E}[F(\Pi_n)^2] \\
\notag&= \frac{\mathbb{E}[\hat{p}_1^2]}{n} + \frac{n(n-1)}{n^2} \mathbb{E}[\hat{p}_1 \hat{p}_2] - 2\mathbb{E}[\hat{p}_1 F(\Pi_n)] + \mathbb{E}[F(\Pi_n)^2] \\
\notag&= o(1) + \mathbb{E}[\hat{p}_1 \hat{p}_2] - 2\mathbb{E}[\hat{p}_1 F(\Pi_n)] + \mathbb{E}[F(\Pi_n)^2].
\end{align}
Proceeding exactly as we did for the first moment, we can decouple the dependencies in the cross-terms by substituting $\Pi_n$ with the restricted point process $\hat{\Pi}_{n-2} = \sum_{j \in \mathcal{W}_\delta, j \notin \{1,2\}} \delta_{\varepsilon W_j}$ (which is independent of both $W_1$ and $W_2$) in the set $\mathbb{1}(1,2 \notin \mathcal{W}_\delta)$, while the probability of the complement vanishes by the Dominated Convergence Theorem. So we obtain $\mathbb{E}[\hat{p}_1\hat{p}_2] = \mathbb{E}[F(\Pi_n)^2] + o(1)$. 
By the exact same argument using $\hat{\Pi}_{n-1}$, we also have $\mathbb{E}[\hat{p}_1 F(\Pi_n)] = \mathbb{E}[F(\Pi_n)^2] + o(1)$. 
Substituting these limits back into the expansion, we get:
\begin{equation*}
    \mathbb{E}[Z_n^2] = o(1) + \mathbb{E}[F(\Pi_n)^2] - 2\mathbb{E}[F(\Pi_n)^2] + \mathbb{E}[F(\Pi_n)^2] = o(1).
\end{equation*}
It follows that $\mathbb{V}(Z_n) \to 0$, hence by Chebyshev's inequality, we have that $|Z_n|\to 0$ in probability as $n\to\infty$.
\end{proof}

Finally we are able to address the following lemma, that characterizes the random limit of $|\widebar{\WW}_\delta|/n$. Recall the PPP $\Pi$ from Definition~\ref{def:PPP_limit}.

\begin{lemma}\label{lemmaginatsize}
 Let $\delta \in (0,1/2)$, then it holds
\begin{equation*}
\frac{|\widebar{\WW}_\delta|}{n}\xrightarrow{\text{ }d\text{ }} 1-\Xi
\end{equation*}
where
\begin{equation*}
\Xi = \Eb\left[\prod_{z \in \text{supp}(\Pi)} (1 - \min(Wz,1))  \;\middle|\;\Pi \right],
\end{equation*}
for $W\sim Pareto(\alpha)$.
\end{lemma}

\begin{proof}[Proof of Lemma \ref{lemmaginatsize}]
To prove that $|\widebar{\WW}_\delta|/n \xrightarrow{d} 1 - \Xi$, we work with its complement and instead prove
\begin{equation*} \label{finale}
\frac{|\widebar{\WW}_\delta^\cc|}{n} \xrightarrow{d} \Xi,
\end{equation*}
since $|\widebar{\WW}_\delta^\cc| = n - |\widebar{\WW}_\delta|$ clearly holds. From Lemma \ref{lemmaA} we know that
\begin{equation*}
\frac{|\widebar{\WW}_\delta^\cc|}{n} - \frac{1}{n} \sum_{i \in [n]} \hat{p}_i(\mathbf{W}) \xrightarrow{\mathbb{P}} 0.
\end{equation*}
This essentially states that the study of the convergence in distribution of $|\widebar{\WW}_\delta^\cc|/n$ can be reduced to that of $\frac{1}{n}\sum \hat{p}_i$. Indeed, by Slutsky's Theorem, if we show that $\frac{1}{n}\sum\hat{p}_i \xrightarrow{d} \Xi$, it then follows that $|\widebar{\WW}_\delta^\cc|/n\xrightarrow{d} \Xi$ as well. Moreover, with Lemma \ref{lemmaB} we also have that to address $\frac{1}{n}\sum \hat{p}_i$ one can study $F(\Pi_n)$, since
\begin{equation*}
\frac{1}{n} \sum_{i \in [n]} \hat{p}_i(\mathbf{W}) - F(\Pi_n) \xrightarrow{\mathbb{P}} 0.
\end{equation*}
This relies on the fact that, conditionally on the configuration of $\WW_\delta$, the variables $\hat{p}_i$ behave approximately as independent, so that an averaging effect, present in the definition of $F$, takes place. Finally, in Subsection \ref{convergenzafunzionali} (see Lemma \ref{continuitàdiF}) we have proved 
\begin{equation*}
F(\Pi_n) \xrightarrow{d} F(\Pi) =: \Xi.
\end{equation*}
Combining the previous steps via Slutsky's Theorem one gets
\begin{equation*}
\frac{|\widebar{\WW}_\delta^\cc|}{n}\xrightarrow{d} F(\Pi).
\end{equation*}
From \eqref{F} it is clear that $\Xi=F(\Pi)$ can be expressed also as
\begin{equation*}
\Xi = \Eb\left[\prod_{z \in \text{supp}(\Pi)} (1 - \min(Wz,1))  \;\middle|\;\Pi \right].
\end{equation*}
\end{proof}

\subsection{Proof of Theorem \ref{thm:giantsize}}\label{finallyherethetheorem}
Let $\WW_\delta$ and $\widebar{\WW}_\delta$ as in Definition \ref{core_delta} and \ref{Ndelta}.

\begin{proof}[Proof of Theorem \ref{thm:giantsize}]
 Lemma \ref{coredeltaexists} ensures that $\WW_\delta$ is non-empty w.h.p. As $n\to \infty$, from Lemma \ref{lemmaginatsize} we know that
 \begin{equation*}
\frac{|\widebar{\WW}_\delta|}{n}\xrightarrow{d} 1-\Xi
\end{equation*}
where
\begin{equation*}
\Xi = \Eb\left[\prod_{z \in \text{supp}(\Pi)} (1 - \min(Wz,1))  \;\middle|\;\Pi \right].
\end{equation*}
 Moreover, Corollary \ref{isolatedornot2} shows
 \begin{equation*}
    \frac{|\mathcal{C}_{max}|}{n}-\frac{|\widebar{\WW}_\delta|}{n}\xrightarrow{\P}0.
\end{equation*}
Thus, by Slutsky's Theorem, it holds
\begin{equation*}
    \frac{|\mathcal{C}_{max}|}{n}\xrightarrow{d} 1-\Xi.
\end{equation*}
It worth noticing that, since 
\begin{equation*}
    \frac{\mathbb{E}[|\mathcal{C}_{max}|]}{n}\to 1-\xi_\lambda
\end{equation*} (see Corollary \ref{isolatedornot2}) and $|\mathcal{C}_{max}|/n\leq 1$, it trivially holds
\begin{equation*}
    \Eb[\Xi]=\xi_\lambda
\end{equation*}
with $\xi_\lambda = \mathbb{E}\left[e^{-\frac{\lambda}{1-\alpha}W^\alpha}\right] = \frac{\lambda}{1-\alpha} \Gamma\left(-1, \frac{\lambda}{1-\alpha}\right)$. 
\end{proof}

\section{Infinite Mean: No sharp Connectivity Threshold. Proof of Theorem \ref{thm:nosharp}}\label{sec:nosharp}

In this section, we present the proof of Theorem~\ref{thm:nosharp} as a direct consequence of Theorem~\ref{thm:giantsize}. In particular, the probability of connection is controlled by the probability of having a hub, i.e. $\HH_\varepsilon \neq \emptyset$, which immediately leads us to $\text{diam}(\mathbb{G}_n(\alpha,\varepsilon))\leq 2$. We refer to notation introduced in Section~\ref{sec:modeldescription}.

\begin{proof}[Proof of Theorem \ref{thm:nosharp}]
By Proposition \ref{core_bounds}, we already know that for a fixed $n$ it holds
\begin{equation*}
    \mathbb{P}(\mathbb{G}_n(\alpha, \varepsilon) \text{ is connected}) \geq \mathbb{P}(\text{diam}(\mathbb{G}_n(\alpha, \varepsilon)) \leq 2)  \geq \mathbb{P}(\HH_\varepsilon \neq \emptyset).
\end{equation*}
Moreover it holds $|\HH_{\varepsilon}|\sim \text{Bin}(n,\varepsilon^\alpha)$ and, in particular,
\begin{equation*}
\mathbb{P}(|\HH_{\varepsilon}|=0)=\left(1-\frac{\lambda}{n}\right)^n\to e^{-\lambda}.
\end{equation*}
Thus, 
\begin{equation}\label{liminf}
 \liminf_{n\to\infty}\mathbb{P}(\mathbb{G}_n(\alpha,\varepsilon)\text{ is connected})\geq 1-e^{-\lambda}.   
\end{equation}
Now we look for a good upper bound. By definition, the graph $\mathbb{G}_n(\alpha,\varepsilon)$ is connected if and only if its largest connected component contains all $n$ vertices, i.e., $|\mathcal{C}_{\max}| = n$. Therefore
\begin{equation*}
    \{\mathbb{G}_n(\alpha,\varepsilon) \text{ is connected}\} = \left\{ \frac{n - |\mathcal{C}_{\max}|}{n} = 0 \right\}.
\end{equation*}
Recall that $\frac{n - |\mathcal{C}_{\max}|}{n} \xrightarrow{d} \Xi$ as $n \to \infty$ from Theorem \ref{thm:giantsize}. Since $\{0\}$ is a closed subset of $\mathbb{R}$, applying Portmanteau's Theorem \cite[Theorem 2.1]{billingsley2013convergence} yields
\begin{equation}\label{limsup}
    \limsup_{n\to\infty}\mathbb{P}(\mathbb{G}_n(\alpha,\varepsilon) \text{ is connected}) = \limsup_{n\to\infty}\mathbb{P}\left(\frac{n - |\mathcal{C}_{\max}|}{n} = 0\right) \le \Pb(\Xi=0),
\end{equation}
where 
\begin{equation*}
    \Xi = \Eb\left[\prod_{z \in \text{supp}(\Pi)} (1 - \min(Wz,1)) \;\middle|\;\Pi \right]
\end{equation*} for $W \sim \text{Pareto}(\alpha)$. We now show that a.s. $\left\{ \Xi=0\right\}=\left\{\exists z \in \text{supp}(\Pi) : z\geq 1 \right\}$. Trivially, if $\exists z \geq 1$, $z \in \text{supp}(\Pi)$, then $\min(Wz,1)\geq \min(z,1)=1$ and thus $\Xi=0$. Otherwise, let us assume by contradiction that $M:=\max(z \in \text{supp}(\Pi) )<1$ and show that $\Xi>0$ a.s.. We have
\begin{align*}
    \Xi&=\Eb\left[ \prod_{z \in \text{supp}(\Pi)}(1-\min(Wz,1)) \Biggl| \Pi     \right]\\
    &\geq\int_{1}^{1/M} \prod_{z \in \text{supp}(\Pi)}(1-\min(uz,1))f_W(du)\\
    &=\int_1^{1/M}e^{-\sum\limits_{z \in \text{supp}(\Pi)}-\ln(1-uz)}f_W(du)\\
    &\geq \int_1^{1/M}e^{-\sum\limits_{z \in \text{supp}(\Pi)} \frac{uz}{1-uz}}f_W(du)\\
    &\geq \int_1^{1/M}e^{-\frac{uA}{1-uM}}f_W(du)
\end{align*}
where we have used the inequality $-\ln(1-x)\leq x/(1-x)$ for $x<1$ and introduced
\begin{equation*}
    A:=\sum_{z \in \text{supp}(\Pi)}z.
\end{equation*}
We observe that the term $A$ is finite a.s. since
\begin{align*}
    \Eb[A]=\int_0^1 z \nu(dz)=\lambda \alpha \int_0^1 z^{-\alpha}dz=\frac{\lambda\alpha}{1-\alpha}<\infty.
\end{align*}
Now, let us introduce $1<u_0<1/M$. Thus, a.s. it holds
\begin{align*}
    \Xi&\geq \int_1^{1/M}e^{-\frac{uA}{1-uM}}f_W(du)\\
    &\geq \int_1^{u_0}e^{-\frac{uA}{1-uM}}f_W(du)\\
    &\geq e^{-\frac{u_0A}{1-u_0M}}\int_1^{u_0}f_W(du)\\
    &=e^{-\frac{u_0A}{1-u_0M}}\Pb(1\leq W\leq u_0)>0.
\end{align*}
That is a contradiction. Therefore, we get
\begin{align*}
    \Pb(\Xi=0)=\Pb(\exists z \in \text{supp}(\Pi): z\geq 1)=1-\Pb(\Pi[1,\infty]=0)=1-e^{-\lambda},
\end{align*}
where $\Pi[1,\infty]\sim \text{Poisson}(\lambda)$ is the number of atoms of $\Pi$ in the interval $[1,\infty]$. Thus, from \eqref{limsup} we also have
\begin{equation*}
    \limsup_{n\to\infty}\mathbb{P}(\mathbb{G}_n(\alpha,\varepsilon)\text{ is connected})\leq 1-e^{-\lambda}
\end{equation*}
that together with \eqref{liminf} yields the desired result. 
\end{proof}

\section{Infinite mean: second largest component and distances}\label{sec:2nd_component_distance}
\subsection{Second largest component: proof of Theorem \ref{thm:2nd_component}}\label{sec:2nd_component}

  In this section, we present the proof of Theorem \ref{thm:2nd_component}. We recall from Definitions \ref{core_delta} and \ref{Ndelta} that
\begin{equation*}
    \WW_\delta=\left\{i \in [n]: W_i>n^{\frac{1-\delta}{\alpha}} \right\}\quad\quad \text{and}\quad\quad
\widebar{\WW}_\delta := \{ i \in [n] : i \in \WW_\delta \text{ or } \exists j \in \WW_\delta \text{ s.t. } i \leftrightarrow j \}.
\end{equation*}
If $S\subseteq [n]$, let $G[S]$ be the subgraph of $\mathbb{G}_n(\alpha,\varepsilon)$ induced by the vertex set $S$.

\begin{proof}[Proof of Theorem \ref{thm:2nd_component}]
  Let $\delta \in (0,1/2)$. Lemma \ref{coredeltaexists} ensures that $\WW_\delta$ is a non-empty clique w.h.p.. For $k\in \N$, we introduce the random variable
\begin{equation} \label{X_K}
    X_k := \sum_{S \subset [n] \atop |S|=k} \ind(G[S] \text{ is connected and } S \cap \widebar{\WW}_\delta = \emptyset) 
\end{equation}
to count the number of distinct sets of $k$ vertices that are not directly connected to $\WW_\delta$ but form a connected subgraph. From the proof of Corollary \ref{isolatedornot2}, it is clear that $\widebar{\WW}_\delta \subseteq \mathcal{C}_{max}$. Since any connected component of size at least $k$ contains a connected subgraph of size exactly $k$, it holds $\{|C_{(2)}| \ge k\} \subseteq \{X_k \ge 1\}$. Our goal is to show that for any $k > (1-\alpha)^{-1}$, $X_k = 0$ w.h.p. This clearly implies that 
\begin{equation}\label{D}
\mathbb{P}(|\mathcal{C}_{(2)}|\geq k)\to 0,
\end{equation}
too. To prove that $X_k = 0$ w.h.p., it suffices to show that $\mathbb{E}[X_k] \to 0$ as $n \to \infty$, since by Markov's inequality we have $\mathbb{P}(X_k \geq 1) \leq \mathbb{E}[X_k]$. By weights exchangeability, we can write
\begin{align}\label{mean_final}
    \mathbb{E}[X_k] = \binom{n}{k} \mathbb{P}(G[1, \dots, k] \text{ is connected and } \{1, \dots, k\} \cap \widebar{\WW}_\delta = \emptyset)=\binom{n}{k}\hat{p},
\end{align}
where $\hat{p}$ denotes the latter probability. Let $\tilde{W} = \max(W_1, \dots, W_k)$ and $\tilde{i} \in [1,\dots, k]$ such that $W_{\tilde{i}}=\tilde{W}$. The probability that no vertex in the set $[k]$ is connected to the core $\WW_\delta$ is bounded by the probability that the vertex $\tilde{i}$ with the maximum weight $\tilde{W}$ is not connected. In particular, conditioning on $W_1,\dots, W_k$ we can write
\begin{align*}
    \mathbb{E}_{W_1,\dots,W_k}[\mathbb{1}(1,\dots,k\notin \widebar{\WW}_\delta)]&=\mathbb{P}(1,\dots,k \notin \widebar{\WW}_\delta \mid W_1,\dots, W_k)\\
    &<  \mathbb{P}(\tilde{i}\notin \widebar{\WW}_\delta \mid W_1,\dots, W_k)\\
    \notag&=  \ind(\tilde{i} \notin \WW_\delta)\mathbb{E}\left[\prod_{j \geq k+1} 1-\min(\varepsilon \tilde{W} W_j,1)\ind(j \in \WW_\delta)\mid \tilde{W} \right]\\
    \notag&\leq \mathbb{E}\left[ 1-\min(\varepsilon \tilde{W} W_{k+1},1)\ind(W_{k+1}>n^{\frac{1-\delta}{\alpha}})\mid \tilde{W} \right]^{n-k}\\
    \notag&\leq\left(1-c_{\lambda,\alpha}\frac{\tilde{W}^\alpha}{n} \right)^{n-k}
    \leq e^{-d_{\lambda,\alpha}\tilde{W}^\alpha}
\end{align*}
where in the second to last step we have used the same computation as in \eqref{bla} and introduced suitable positive constants $c_{\lambda,\alpha}>0$ and $d_{\lambda,\alpha}>0$.
Thus, it holds
\begin{align*}
    \mathbb{E}\left[\tilde W^{2(k-1)}\mathbb{E}_{W_1,\dots,W_k}[\mathbb{1}(1,\dots,k \notin\widebar{\WW}_\delta)]\right] \leq \mathbb{E}[\tilde{W}^{2(k-1)}e^{-d_{\lambda,\alpha} \tilde{W}^\alpha}]=C<\infty
\end{align*}
since the exponential decay of $e^{-d_{\lambda,\alpha}\tilde{W}^\alpha}$ dominates the polynomial tail of the maximum $\tilde{W}$. Now, let $\mathcal{T}_k$ be the set of all possible trees on $k$ vertices. Since $\{G[1, \dots, k] \text{ is connected}\} = \bigcup_{T \in \mathcal{T}_k} \{T \subseteq G[1, \dots, k]\}$, by applying a union bound and using the fact that $|\mathcal{T}_k| = k^{k-2}$ (Cayley's Formula), we obtain:
\begin{align*}
  \hat{p} &=\mathbb{E}\left[ \ind(G[1,\dots, k ]\text{ is connected}) \ind(1,\dots,k \notin \widebar{\WW}_\delta)\right]\\
    &\leq  \sum_{T \in \mathcal{T}_k} \mathbb{E}\left[ \prod_{(i,j) \in E(T)} \min(\varepsilon W_iW_j,1) \mathbb{E}_{W_1,\dots,W_k}\left[\ind(1,\dots,k \notin\widebar{\WW}_\delta)\right]\right]\\
    &\leq \sum_{T \in \mathcal{T}_k} \mathbb{E}\left[  \varepsilon^{k-1} \tilde{W}^{2(k-1)}\mathbb{E}_{W_1,\dots,W_k}\left[\ind(1,\dots,k \notin\widebar{\WW}_\delta)\right]\right]\\
    &= k^{k-2}\varepsilon^{k-1} \mathbb{E}\left[   \tilde{W}^{2(k-1)}\mathbb{E}_{W_1,\dots,W_k}\left[\ind(1,\dots,k \notin\widebar{\WW}_\delta)\right]\right],
\end{align*}
that brings to
\begin{align*}
\mathbb{E}[X_k]\leq n^kk^{k-2}\varepsilon^{k-1}\mathbb{E}[\tilde{W}^{2(k-1)}\mathbb{E}_{W_1,\dots,W_k}\left[\mathbb{1}(W_1,\dots,W_k \notin \widebar{\WW}_\delta)]\right]\leq k^{k-2}C\lambda^{\frac{k-1}{\alpha}} \cdot n^{k-\frac{k-1}{\alpha}}.
\end{align*}
As $n \to \infty$, the right-hand side vanishes if $k - \frac{k-1}{\alpha} < 0$, which is equivalent to $k > (1-\alpha)^{-1}$. This brings to \ref{D} and concludes the proof.
\end{proof}

\subsection{Extreme ultra small world phenomenon: proof of Proposition \ref{thm:extremeultrasmallworld}} \label{sec:extremeultrasmallworld}

In this section we prove Proposition \ref{thm:extremeultrasmallworld} regarding the extreme ultra small world nature of the infinite mean regime in $\mathbb{G}_n(\alpha,\varepsilon)$. We show that the diameter in the giant component of $\mathbb{G}_n(\alpha,\varepsilon)$ is bounded by a constant w.h.p. and that the typical distance is at most $3$ w.h.p.

\begin{proof}[Proof of Proposition \ref{thm:extremeultrasmallworld}]
Let $\delta \in (0,1/2)$. From Lemma \ref{coredeltaexists}, $\mathcal{W}_\delta$ is not empty w.h.p. We first show that the typical distance in the giant component is at most $3$. Let $u$ and $v$ be two vertices chosen uniformly at random from $\Cmax$. Given a realisation $G_n$ of the graph, it holds:
\begin{equation*}
    \mathbb{P}(u, v \in \widebar{\mathcal{W}}_\delta \mid {G}_n) = \left( \frac{|\Cmax \cap \widebar{\mathcal{W}}_\delta|}{|\Cmax|} \right)^2 = \left( 1 - \frac{|\Cmax \setminus \widebar{\mathcal{W}}_\delta|}{|\Cmax|} \right)^2.
\end{equation*}
Theorem \ref{thm:giantsize} ensures that $|\mathcal{C}_{max}|/n\xrightarrow{d}1-\Xi$ and Corollary \ref{isolatedornot2} ensures that $|\Cmax \setminus \widebar{\WW}_\delta| / n \xrightarrow{\mathbb{P}} 0$. Thus, both $u$ and $v$ belong to $\widebar{\WW}_\delta$ w.h.p., since
\begin{equation*}
\frac{|\mathcal{C}_{max}\setminus \widebar{\mathcal{W}}_\delta|}{|\mathcal{C}_{max}|}=\frac{|\mathcal{C}_{max}\setminus \widebar{\mathcal{W}}_\delta|}{n}\cdot \frac{n}{|\mathcal{C}_{max}|}\xrightarrow{\text{ }\mathbb{P}\text{ }}0.
\end{equation*}
By Definition \ref{Ndelta}, this means there exist vertices $u', v' \in \mathcal{W}_\delta$ such that $u\leftrightarrow u'$ and $v\leftrightarrow v'$ (where possibly $u = u'$ or $v = v'$). Moreover, by Lemma \ref{coredeltaexists}, $\mathcal{W}_\delta$ forms a clique w.h.p., implying that the edge $\{u', v'\}$ is present. Thus, the path $u \leftrightarrow u' \leftrightarrow v' \leftrightarrow v$ exists, proving that the distance between $u$ and $v$ is at most $3$ w.h.p.

Lastly, we want to control the diameter.  We consider the vertices in $\Cmax \setminus \overline{\mathcal{W}}_\delta$. Let $X_k$ as in \eqref{X_K} for $k\geq 1$ and let $K = \lfloor (1-\alpha)^{-1} \rfloor$.  From the proof of Theorem \ref{thm:2nd_component}, $X_{K+1} = 0$ w.h.p. Since every vertex $u \in \Cmax \setminus \overline{\mathcal{W}}_\delta$ must eventually connect to $\widebar{\WW}_\delta$ to belong to the giant component, it must be part of a "branch" of size at most $K$ attached to $\overline{\mathcal{W}}_\delta$. 
Consequently, the maximum distance from any vertex $u \in \Cmax \setminus \overline{\mathcal{W}}_\delta$ to the set $\overline{\mathcal{W}}_\delta$ is bounded by $K$. This implies that any vertex in the giant component is at distance at most $K+1$ from the clique $\mathcal{W}_\delta$. Since the distance between any two distinct vertices within the clique is exactly $1$, the maximum distance between any pair of vertices $u,v \in \Cmax$ is bounded by $(K+1) + 1 + (K+1) = 2K + 3$.
\end{proof}

\section{Infinite Mean: Extreme Degrees. Proof of Theorem \ref{thm:maxdegree}}\label{sec:maxdegree}

In this section we present the proof of Theorem \ref{thm:maxdegree}. We recall that $D_n(i)$ indicates the degree of the vertex $i$ and that
\begin{equation*}
    \WW_\delta=\left\{i \in [n]: W_i>n^{\frac{1-\delta}{\alpha}} \right\}
\end{equation*}
is the set of weak hubs as introduced in Definition \ref{core_delta}. Lemma \ref{coredeltaexists} specifies that $\mathcal{W}_\delta$ is a non-empty clique if $\delta \in (0,1/2)$. Moreover, we recall, from Section \ref{sec:giantsize}, that  $\Pi_n:=\sum_{i \in \mathcal{W}_\delta}\delta_{\varepsilon W_i}$ converges weakly to $\Pi$ in $\mathcal{M}_p((0,\infty])$.
To start, we introduce the following definitions.

\begin{definition}
Let $\mathbb{G}_n(\alpha,\varepsilon)$ with $\alpha<1$ and $\varepsilon^\alpha=\lambda/n$ with $\lambda \in (0,\infty)$. Let $\delta \in (0,1/2)$. Let $k \in \mathbb{N}$, $k\geq 1$. We write 
\begin{equation*}
    W_{(1)} > W_{(2)} > \dots > W_{(k)}
\end{equation*} to indicate the k largest weights in $\mathbf{W}=(W_i)_{i \in [n]}$ and, for $u \in [k]$, we write $v_u$ to indicate the vertex with the $u$-th largest weight, i.e. $W_{v_u}=W_{(u)}$.
Similarly, we write 
\begin{equation*}
    Z_{(1)}>Z_{(2)}>\dots >Z_{(k)}
\end{equation*} to indicate the k largest points in $\text{supp}(\Pi)$. 
\end{definition}

The latter are well defined almost surely. In fact, as $\nu((0,\infty])=\infty$ and $\nu$ is absolutely continuous with respect to the Lebesgue measure, $\Pi$ is almost surely an infinite and simple point process (see, e.g., \cite{resnick1987extreme}, Section $3.3.1$ and Exercise $3.3.3$), ensuring its atoms are distinct.

\begin{definition}
Similarly we write 
\begin{equation*}
    \Delta_{(1)}\geq \Delta_{(2)}\geq \dots\geq \Delta_{(k)}
\end{equation*} to indicate the k largest degrees in $\mathbb{G}_n(\alpha,\varepsilon)$. 
\end{definition}

Letting $W$ independent of $(W_i)_{i\in[n]}$, we also recall that the function 
\begin{equation*}
h(x):=\mathbb{E}[\min(Wx,1)]=
\begin{cases}
         \frac{x^{\alpha}-\alpha x}{1-\alpha}   &x < 1;\\
         1   &x\geq 1
        \end{cases}
\end{equation*}
as defined in Definition~\ref{h}, represents the asymptotic probability that a typical vertex $W$ is linked to a point $x>0$. We notice that it directly follows by the computation in Appendix (see Corollary \ref{expectedprobgivenaweight} using $\varepsilon=1$).

To prove Theorem \ref{thm:maxdegree}, we rely on two preliminary lemmas. In the first one (Lemma~\ref{lem:uniform_concentration}), we establish the uniform concentration of the rescaled node degrees around their expected values $h(\varepsilon W_i)$. In the second one (Lemma~\ref{hofinitoinomi}), we derive the joint limiting distribution of the top $k$ extreme degrees and show that, with high probability, they preserve the ordering of the largest vertex weights.

\begin{lemma}
\label{lem:uniform_concentration}
For any $\eta > 0$, as $n \to \infty$,
$$
\mathbb{P}\left( \max_{1 \le i \le n} \left| \frac{D_n(i)}{n} - h(\varepsilon W_i) \right| > \eta \right) \le 2n e^{-2(n-1)\eta^2} \longrightarrow 0.
$$
In particular, for any sequence of vertices $v_u$ corresponding to the $u$-th largest weight $W_{(u)}$ with $1 \le u \le k$,
$$
\frac{D_n(v_u)}{n} - h(\varepsilon W_{(u)}) \xrightarrow{\mathbb{P}} 0.
$$
\end{lemma}
\begin{proof}
Fix a vertex $i \in [n]$ and condition on its weight $W_i$. For each $j \neq i$, the conditional probability that an edge exists between $i$ and $j$ given $W_i$ is
$$
p_{ij} := \mathbb{P}(i \leftrightarrow j \mid W_i) = \mathbb{E}_{W_j}[\min(1, \varepsilon W_i W_j)] = h(\varepsilon W_i).
$$
Consequently,
$$
D_n(i) \mid W_i \sim \text{Bin}\Big(n-1, h(\varepsilon W_i)\Big).
$$
Applying Hoeffding's inequality yields, for any $\eta > 0$,
$$
\mathbb{P}\left( \left| \frac{D_n(i)}{n-1} - h(\varepsilon W_i) \right| > \eta \;\middle|\; W_i \right) \le 2 e^{-2(n-1)\eta^2}.
$$
Taking expectations with respect to $W_i$ implies that the unconditional probability satisfies the exact same bound:
$$
\mathbb{P}\left( \left| \frac{D_n(i)}{n-1} - h(\varepsilon W_i) \right| > \eta \right) \le 2 e^{-2(n-1)\eta^2}.
$$
Applying a union bound over all $n$ vertices gives
$$
\mathbb{P}\left( \max_{1 \le i \le n} \left| \frac{D_n(i)}{n-1} - h(\varepsilon W_i) \right| > \eta \right) \le \sum_{i=1}^n 2 e^{-2(n-1)\eta^2} = 2n e^{-2(n-1)\eta^2} \longrightarrow 0,
$$
as $n \to \infty$. This proves the uniform concentration over the entire graph, from which the convergence for any specific $v_u$ follows. Note that replacing $n-1$ with $n$ in the denominator introduces an error of order $O(1/n)$ uniform in $i$, which vanishes asymptotically. Thus, we can replace $n-1$ with $n$ without problems. 
\end{proof}

\begin{lemma}[Joint Distribution of Extreme Degrees]\label{hofinitoinomi}
Let $u \in \mathbb{N}$. Then the $u$-th largest degree is realized by the vertex with the $u$-th largest weight. In particular, for a fixed integer $k \ge 1$, as $n \to \infty$, it holds:
\begin{equation*}
   \frac{\vec{\Delta}^{(k)}}{n}:=\left( \frac{\Delta_{(1)}}{n}, \dots, \frac{\Delta_{(k)}}{n} \right) \xrightarrow{d} \left( h(Z_{(1)}), \dots, h(Z_{(k)}) \right).
\end{equation*}
\end{lemma}

\begin{proof}
From Lemma \ref{lem:uniform_concentration}, we know that for each $u \in \{1, \dots, k\}$ it holds:
\begin{equation*}
    \frac{D_n(v_u)}{n} - h(\varepsilon W_{v_u}) \xrightarrow{\mathbb{P}} 0.
\end{equation*}
Thus, since $k$ is a finite constant, we may write
\begin{equation*}
    \left(\frac{D_n(v_1)}{n}, \dots, \frac{D_n(v_k)}{n}\right) - \big(h(\varepsilon W_{(1)}), \dots, h(\varepsilon W_{(k)})\big) \xrightarrow{\mathbb{P}} \mathbf{0}.
\end{equation*}
Moreover, from Lemma \ref{PPPCOLORLLARY}, we know that $\Pi_n=\sum_{i \in \mathcal{W}_\delta}\delta_{\varepsilon W_i}$ converges weakly to $\Pi$ on $\mathcal{M}_p((0,\infty])$. Thus, since the mapping from a point measure in $\mathcal{M}_p((0,\infty])$ to the vector of its $k$ largest atoms is continuous with respect to the vague topology, (\cite[Proposition~3.13,  Corollary~4.19]{resnick1987extreme}) from the Continuous Mapping Theorem it holds
\begin{align*}
\left(\varepsilon W_{(1)}, \dots, \varepsilon W_{(k)}\right) \xrightarrow{d} (Z_{(1)}, \dots, Z_{(k)})
\end{align*}
and thus
\begin{align*}\label{convergenceofh}
\left(h(\varepsilon W_{(1)}), \dots, h(\varepsilon W_{(k)})\right) \xrightarrow{d} (h(Z_{(1)}), \dots, h(Z_{(k)}))
\end{align*}
holds, by the continuity of $h$. Applying the multivariate Slutsky's Theorem, this brings to
\begin{equation}\label{eq:heaviest_degrees}
    \left(\frac{D_n(v_1)}{n}, \dots, \frac{D_n(v_k)}{n}\right) \xrightarrow{d} \big(h(Z_{(1)}), \dots, h(Z_{(k)})\big).
\end{equation}
We now want to say that $D_n(v_u)=\Delta_{(u)}$ for $u \in [k]$ w.h.p., i.e. that the $u$-th largest degree is realized by the vertex with the $u$-th largest weight w.h.p.. In particular, we now show that for each $u \in [k]$
\begin{equation}\label{order}
    \mathbb{P}\left(D_n(v_u)\geq D_n(j) \text{ }\forall j \notin \{v_1, \dots, v_u \}\right) \to 1.
\end{equation}
i.e. each of the $v_u$ reach a degree greater than or equal to the degree of the vertices with smaller weight. We distinguish two cases: when $\varepsilon W_{(u)}\geq 1$, i.e. when $v_u$ is a hub, or when $\varepsilon W_{(u)}<1$. We have:
\begin{align}\label{A+B}
    \mathbb{P}\left(D_n(v_u)< D_n(j) \text{ }\exists j \notin \{v_1, \dots, v_u \} \right)= A+B
\end{align}
where
\begin{align*}
    &A:=\mathbb{E}\left[\mathbb{1}\left(D_n(v_u)< D_n(j) \text{ }\exists j \notin \{v_1, \dots, v_u \}\right) \mathbb{1}(\varepsilon W_{(u)}\geq 1)\right]\\
    &B:=\mathbb{E}\left[\mathbb{1}\left(D_n(v_u)< D_n(j) \text{ }\exists j \notin \{v_1, \dots, v_u \}\right) \mathbb{1}(\varepsilon W_{(u)}< 1)\right]
\end{align*}
Let us start with the term in $A$. When $\varepsilon W_{(u)}\geq 1$, the probability of connection between $v_u$ and the other vertices becomes $1$ even with fixed $n$. Thus, $v_u$ reaches the maximum possible degree $n-1$ and trivially $D_n(j)\leq n-1=D_n(v_u)$ $\forall j \notin \{v_1, \dots, v_u\}$. In particular, we may write $A=0.$
Otherwise, we observe that since $(\varepsilon W_{(u)}, \varepsilon W_{(u+1)})\xrightarrow{d} (Z_{(u)}, Z_{(u+1)})$, by the continuity of $h$ and the fact that the process $\Pi$ does not put positive mass at $1$ we also have that $[h(\varepsilon W_{(u)})-h(\varepsilon W_{(u+1)})]\mathbb{1}(\varepsilon W_{(u)}<1)$ converges in distribution to $[h(Z_{(u)})-h(Z_{(u+1)})]\mathbb{1}(Z_{(u)}<1)$. 
 When $Z_{(u)}<1$, $h(Z_{(u)})-h(Z_{(u+1)})$ is strictly positive, since $h$ is strictly increasing in $(0,1)$. In particular, for each $\eta>0$, there exists a sufficiently small $\delta>0$, $\delta=\delta(\eta)$ such that
\begin{equation*}
    \Eb\left[\mathbb{1}(h(Z_{(u)})-h(Z_{(u+1)})\leq\delta)\mathbb{1}(Z_{(u)}<1)\right]<\eta.
\end{equation*}
We now define the events
\begin{align*}
    &\mathcal{B}_1:=\{h(\varepsilon W_{(u)})-h(\varepsilon W_{(u+1)})\leq\delta , \varepsilon W_{(u)}<1 \}\\
    &\mathcal{B}_2:=\{h(\varepsilon W_{(u)})-h(\varepsilon W_{(u+1)})>\delta , \varepsilon W_{(u)}<1 \}
\end{align*}
and split $B$ in two more parts writing $B=B_1+B_2$ where
\begin{align*}
    &B_1=\mathbb{E}\left[\mathbb{1}\left(D_n(v_u)< D_n(j) \text{ }\exists j \notin \{v_1,\dots, v_u\} \right) \mathbb{1}(\mathcal{B}_1)\right]\\
   &B_2=\mathbb{E}\left[\mathbb{1}\left(D_n(v_u)< D_n(j) \text{ }\exists j \notin \{v_1,\dots, v_u\} \right) \mathbb{1}(\mathcal{B}_2)\right].
\end{align*}
For the first, we have that 
\begin{equation*}
    \limsup_{n\to\infty}B_1 \le \limsup_{n\to\infty}\mathbb{P}(\mathcal{B}_1) \le \Eb\left[\mathbb{1}(h(Z_{(u)})-h(Z_{(u+1)})\le\delta)\mathbb{1}(Z_{(u)}< 1)\right] < \eta
\end{equation*}
Instead, for the second, working on $\mathcal{B}_2$, we may write:
\begin{align*}
\min_{j \notin \{v_1, \dots, v_u \}}\frac{D_n(v_u)-D_n(j)}{n}&\geq h(\varepsilon W_{(u)})-h(\varepsilon W_{(u+1)})-2\max_{j \in [n]} \left |\frac{D_n(j)}{n}-h(\varepsilon W_j) \right |\\
&\geq \delta -2\max_{j \in [n] } \left |\frac{D_n(j)}{n}-h(\varepsilon W_j) \right |.
\end{align*}
From Lemma \ref{lem:uniform_concentration} it is clear that the last term can be made greater than $\delta/2$ w.h.p., concentrating the $\max$ term sufficiently close to $0$ depending on the choice of $\delta$. Thus, $B_2\to 0$ as $n\to \infty$. Finally, considering all the terms together, from \eqref{A+B} we have
\begin{align*}
    \limsup_{n\to \infty}\mathbb{P}(D_n(v_u)<D_n(j) \text{ } \exists j \notin \{v_1,\dots, v_u\})=\limsup_{n\to \infty} (A+B_1+B_2)< \eta
\end{align*}
Since $\eta$ can be chosen arbitrarily small, we finally obtain \eqref{order} for any $u \in [k]$. This is sufficient to say that
\begin{equation*}
    \mathbb{P}\left(D_n(v_u)=\Delta_{(u)} \text{ }\forall u \in [k]\right)\to 1 \quad\quad\quad\text{ as } n \to \infty.
\end{equation*}
From \eqref{eq:heaviest_degrees} we finally have
\begin{equation*}
\left( \frac{\Delta_{(1)}}{n}, \dots, \frac{\Delta_{(k)}}{n} \right) \xrightarrow{d} \left( h(Z_{(1)}), \dots, h(Z_{(k)}) \right).
\end{equation*}
\end{proof}

We are now ready to see the proof of Theorem \ref{thm:maxdegree}
\begin{proof}[Proof of Theorem \ref{thm:maxdegree}]
To show the weak convergence for the vague topology, it suffices to prove that $\Psi_n(f) \xrightarrow{d} \Pi[h](f)$ for any continuous, non negative test function $f$ with compact support of type $[y, 1]$, for some $y > 0$ \cite[Proposition $3.19$]{resnick1987extreme}. For any $k>1$ we may write
\begin{equation*}
    \Psi_n(f)=\sum_{i=1}^k f\left(\frac{\Delta_{(i)}}{n} \right)+\sum_{i\geq k+1} f\left(\frac{\Delta_{(i)}}{n} \right).
\end{equation*} 
By Lemma \ref{hofinitoinomi}, it holds that
\begin{equation*}
\left(\frac{\Delta_{(1)}}{n},\dots, \frac{\Delta_{(k)}}{n} \right)\xrightarrow{d}\left(h(Z_{(1)}), \dots, h(Z_{(k)}) \right)
\end{equation*}
which implies, by the Continuous Mapping Theorem,
\begin{equation*}
    \sum_{i=1}^k f\left(\frac{\Delta_{(i)}}{n} \right)\xrightarrow{d} \sum_{i=1}^k f\left(h(Z_{(i)} )\right).
\end{equation*}
On the other side, since $\Delta_{(k+1)}/n\xrightarrow{d} h(Z_{(k+1)})$ too, we get
\begin{align*}
   \lim_{n\to \infty}\mathbb{P}\left( \sum_{i\geq k+1} f\left(\frac{\Delta_{(i)}}{n} \right)>0\right)\leq \lim_{n\to \infty}\mathbb{P}\left(\frac{\Delta_{(k+1)}}{n}\geq y \right)=\Pb(h(Z_{(k+1)})\geq y)\xrightarrow[k\to \infty]{}0
\end{align*}
using the properties of the Poisson Point Process $\Pi$ that is not able to accumulate infinite points in a compact $[a,\infty]$, with $a>0$. Thus
\begin{equation*}
    \limsup_{k\to \infty}\lim_{n\to \infty}\mathbb{P}\left( \sum_{i\geq k+1} f\left(\frac{\Delta_{(i)}}{n} \right)>0\right)=0.
\end{equation*}
By Billingsley's convergence theorem \cite[Theorem 3.2]{billingsley2013convergence}, combining the convergence of the first terms with the vanishing bound on the tail, it yields $\Psi_n(f) \xrightarrow{d} \Pi_{[h]}(f)$ for all $f \in C^+_c((0,1])$.
\end{proof}

\appendix

\section{}\label{appendice1}
\subsection{Estimates of key observables}\label{modelfirstpropoerties}
Throughout this section, fix $\eps>0$ and $\alpha>0$ and consider the graph $\mathbb{G}_n(\alpha,\varepsilon)$. For any pair of distinct vertices $i\neq j\in [n]$, recall $p_{ij}$ defined in \eqref{defofp}. This is a random variable under $\P$ and, by exchangeability, $p_{ij}$ has the same distribution of $p_{12}$. We denote with
    \begin{equation}\label{p1}
        \widebar{p}:=\E[p_{12}]\quad\quad\text{ and }\quad\quad\mathbb{E}_w[p_{12}]:=\mathbb{E}[p_{12}|W_1=w]
    \end{equation}
    the probability of having an edge and the probability of having an edge knowing that the weight of one of the two adjacent vertices is a constant $w>1$, respectively. For any $m>0$, the following are true, being $\{W_i\}_i$ independent Pareto distributed:
    \begin{align*}
&\E_x[p_{12}^m]=\alpha\int_1^\infty\frac{\min(\varepsilon x y,1)^m}{y^{\alpha+1}}dy;\\
&||p_{12}||_m^m:=\E [p_{12}^m]=\alpha^2\int_1^\infty\int_1^\infty\frac{\min(\varepsilon x y,1)^m}{(xy)^{\alpha+1}}dxdy.
\end{align*}
Because of the form of our kernel and of the Pareto distribution, the above integrals can be explicitly computed. We do so in the following lemma.

\begin{lemma}[Explicit expression of $L_m$-norm]\label{norma2}
Let $m>0$, then it holds
\begin{align*}
\mathbb{E}_x[p_{12}^m]=
\begin{cases}
\frac{m}{m-\alpha}(\varepsilon x)^\alpha+\frac{\alpha}{\alpha-m}(\varepsilon x)^m 	&\alpha \neq m;\\
-(\varepsilon x)^m\ln((\varepsilon x)^m)+(\varepsilon x)^m 					&\alpha=m
\end{cases}
\end{align*}
if $x\leq 1/\varepsilon$, $\mathbb{E}_x[p_{12}^m]=1$ otherwise. Moreover 
\begin{align*}
||p_{12}||_m^m=
\begin{cases}
\left(\frac{\alpha}{\alpha-m}	\right)^2\varepsilon^m-\frac{m}{m-\alpha}\varepsilon^\alpha \ln(\varepsilon^\alpha)+\left(1-(\frac{\alpha}{\alpha-m})^2	\right)\varepsilon^\alpha 	&\alpha\neq m;\\
\frac{1}{2}\varepsilon^m \ln^2(\varepsilon^m)-\varepsilon^m\ln(\varepsilon^m)+\varepsilon^m 	&\alpha=m.
\end{cases}
\end{align*}
\end{lemma}

\begin{proof}
If $x>\varepsilon^{-1}$, then $\varepsilon x y\geq 1$ for all $y\geq1$ and $\min(\varepsilon xy,1)=1$. If $x\leq \varepsilon^{-1}$ 
it holds
\begin{align*}
\mathbb{E}_x[p_{12}^m]=&\int_1^{\infty}\frac{\alpha \min(\varepsilon x y,1)^m}{y^{\alpha+1}}dy
=\int_1^{\frac{1}{\varepsilon x}}\frac{\alpha (\varepsilon xy)^m}{y^{\alpha+1}}dy+\int_{\frac{1}{\varepsilon x}}^\infty\frac{\alpha}{y^{\alpha+1}}dy.
\end{align*}
When $\alpha\neq m$ this becomes
\begin{align*}
\mathbb{E}_x[p_{12}^m]=&\alpha (\varepsilon x)^m \left[\frac{y^{m-\alpha}}{m-\alpha}\right]^{\frac{1}{\varepsilon x}}_1+(\varepsilon x)^\alpha
=\frac{\alpha}{\alpha-m} (\varepsilon x)^m (1-(\varepsilon x)^{\alpha-m})+(\varepsilon x)^\alpha\\
=&\frac{\alpha}{\alpha-m}(\varepsilon x)^m +\frac{m}{m-\alpha}(\varepsilon x)^\alpha,
\end{align*}
instead, if $\alpha=m$, a log-correction kicks in 
\begin{align*}
\mathbb{E}_x[p_{12}^m]=m(\varepsilon x)^m \left[\ln(y)	\right]_1^{\frac{1}{\varepsilon x}}+(\varepsilon x)^{m}=-(\varepsilon x)^m \ln((\varepsilon x)^m)+(\varepsilon x)^{m}.
\end{align*}
Moreover, it holds that $||p_{12}||_m^m=\mathbb{E}[\mathbb{E}_{W_1}[p_{12}^m]]$, so if $\alpha \neq m$
\begin{align*}
||p_{12}||_m^m=&\int_1^{\frac{1}{\varepsilon}}\frac{\alpha}{x^{\alpha+1}}\mathbb{E}_x[p_{12}^m]dx+\int_{\frac{1}{\varepsilon}}^{\infty}\frac{\alpha}{x^{\alpha+1}}dx\\
=&\int_1^{\frac{1}{\varepsilon}}\frac{\alpha m}{m-\alpha}\varepsilon^\alpha x^{-1} +\frac{\alpha^2}{\alpha-m}\varepsilon^m x^{m-\alpha-1} dx +\varepsilon^\alpha\\
=&-\frac{\alpha m}{m-\alpha}\varepsilon^\alpha \ln(\varepsilon) +\left(\frac{\alpha^2}{\alpha-m}\right)^2 \varepsilon^m (1-\varepsilon^{\alpha-m})+\varepsilon^\alpha\\
=&-\frac{m}{m-\alpha}\varepsilon^\alpha \ln(\varepsilon^\alpha)+\left(\frac{\alpha}{\alpha-m}\right)^2\varepsilon^m +\left(1-\left(\frac{\alpha}{\alpha-m}\right)^2\right)\varepsilon^\alpha,
\end{align*}
and if $\alpha=m$
\begin{align*}
||p_{12}||_m^m=&\int_1^{\frac{1}{\varepsilon}}\frac{m\mathbb{E}_x[p_{12}^m]}{x^{m+1}}dx+\varepsilon^m\\
=&\int_1^{\frac{1}{\varepsilon}}-m^2\varepsilon^m x^{-1}\ln(\varepsilon) -m^2\varepsilon^mx^{-1}\ln(x)+m\varepsilon^m x^{-1}dx +\varepsilon^m\\
=& m^2\varepsilon^m\ln^2(\varepsilon)-m^2\varepsilon^m\left[\frac{1}{2}\ln^2(x)	\right]_1^{\frac{1}{\varepsilon}}-\varepsilon^m \ln(\varepsilon^m)+\varepsilon^m\\
=&\varepsilon^m\ln^2(\varepsilon^m)-\frac{\varepsilon^m\ln^2(\varepsilon^m)}{2}-\varepsilon^m\ln(\varepsilon^m)+\varepsilon^m\\
=&\frac{\varepsilon^m\ln^2(\varepsilon^m)}{2}-\varepsilon^m\ln(\varepsilon^m)+\varepsilon^m.
\end{align*}
\end{proof}

\begin{corollary}\label{expectedprobgivenaweight}
Consider the graph $\mathbb{G}_n(\alpha,\varepsilon)$ with $\alpha>0$. Let $i$ and $j$ be two distinct vertices. Let $x\geq1$ be a fixed constant and let $\mathbb{E}_x[p_{ij}]$ as in \eqref{p1} to indicate the probability of having the edge $\{i,j\}$ given $W_i=x$. It holds

\begin{equation*}
\mathbb{E}_x[p_{ij}]=
\begin{cases}
\frac{\alpha}{\alpha-1}\varepsilon x+\frac{1}{1-\alpha}(\varepsilon x)^\alpha		&\alpha \neq 1;\\
-x \varepsilon \ln(\varepsilon x) +\varepsilon x 						&\alpha=1\\
\end{cases}
\end{equation*}
if $x\leq 1/\varepsilon$. Otherwise $\mathbb{E}_x[p_{12}]=1$.
\end{corollary}

\begin{corollary}\label{probabilityofanedge}
Consider the graph $\mathbb{G}_n(\alpha,\varepsilon)$ with $\alpha>0$ and consider two arbitrary, fixed and distinct vertices $i,j$. Let $\widebar{p}$ as in \eqref{p1} to indicate the probability of having the edge $\{i,j\}$. It satisfies

\begin{align*}
\widebar{p}=
\begin{cases}
\left(\frac{\alpha}{\alpha-1}\right)^2\varepsilon-\frac{1}{1-\alpha}\varepsilon^\alpha\ln(\varepsilon^\alpha)+\left(1-\left(\frac{\alpha}{\alpha-1}\right)^2	\right)\varepsilon^\alpha	&\alpha \neq 1;\\
\frac{1}{2}\varepsilon \ln^2(\varepsilon)-\varepsilon \ln\varepsilon+\varepsilon	&\alpha=1.
\end{cases}
\end{align*}
\end{corollary}

\subsection{Proof of Theorem \ref{thm:expecteddegree}}\label{degreeofavertex}
  In this section, we present the proof of Theorem \ref{thm:expecteddegree} regarding the asymptotic behavior of $\mathbb{E}[D_n(i)]$. Theorem \ref{thm:expecteddegree} also specifies the choices for $\varepsilon_n$ that lead to the convergence in distribution of $D_n(i)$ to a mixed Poisson random variable. The proof identifies these scalings naturally, demonstrating that they are the only ones (up to  constants) that allow the convergence of $\mathbb{E}[t^{D_n(i)}]$ to the probability generating function of a non trivial random variable, which turns out to have mixed Poisson distribution. 

  We recall that the degree $D_n(1)$ of vertex $1$ is
\begin{equation*}
D_n(1)=\sum_{j\geq 2}a_{1j}
\end{equation*}
where $a_{ij}=\ind(i\leftrightarrow j)$. By exchangeability $D_n(i)\overset{d}{=}D_n(1)$ for any $i\in [n]$.

\begin{proof}[Proof of Theorem \ref{thm:expecteddegree}]

First, note that
\begin{align*}
\mathbb{E}[D_n(1)]=\mathbb{E}\Biggl[\sum_{j\geq 2} a_{1j}\Biggl]=(n-1)\mathbb{P}(1 \leftrightarrow j)=(n-1)\hat{p},
\end{align*}
so we can easily use Corollary~\ref{probabilityofanedge} to conclude~\eqref{eq:mean_degree1}. Then, let us write, for any $t\in (0,\infty)$, for any $t \in (0,1)$
\begin{align*}
\mathbb{E}\bigl[t^{D_n(1))}\bigr]=&\mathbb{E}\Bigl[t^{\sum_{j\geq 2}a_{1j}}\Bigr]=\mathbb{E}\Biggl[\prod_{j\geq 2}t^{a_{1j}}\Biggr]=\mathbb{E}\Biggl[\mathbb{E}_{W_1}\Biggl[\prod_{j\geq 2}t^{a_{1j}}\Biggr]\Biggr]=\mathbb{E}\left[ \mathbb{E}_{W_1}[t^{a_{1j}}]^{n-1} 	\right],
\end{align*}
where $\mathbb{E}_{W_1}[\cdot]$ indicates the conditional expectation given the weight $W_1$. Since, given $W_1$ and $W_j$, $\{a_{1j}\}_j$ are independent Bernoulli with parameter $\min(\varepsilon W_1 W_j,1)$, we get
\begin{align*}
\mathbb{E}_{W_1}[t^{a_{1j}}]=&\mathbb{E}_{W_1}\left[(t-1)\min(\varepsilon W_1 W_j,1)+1\right]=1+(t-1)\mathbb{E}_{W_1}[p_{12}].
\end{align*}
Then, using Corollary \ref{expectedprobgivenaweight}, when $\alpha \in (0,1)$ we obtain
\begin{align*}
\mathbb{E}_{W_1}[t^{a_{1j}}]^{n-1}=&\left(1+(t-1)\left(\frac{\alpha}{\alpha-1}\varepsilon W_1+\frac{1}{1-\alpha}(\varepsilon W_1)^\alpha\right)\right)^{n-1}\\
=&\left(1+(t-1)\frac{1}{1-\alpha}(\varepsilon W_1)^\alpha(1+o(1))\right)^{n-1}\xrightarrow[\varepsilon^\alpha=\frac{\lambda}{n} \atop n\to \infty]{} e^{(t-1)\frac{\lambda}{1-\alpha}W_1^\alpha};
\end{align*}
for $\alpha >1$
\begin{align*}
\mathbb{E}_{W_1}[t^{a_{1j}}]^{n-1}=&\left(1+(t-1)\left(\frac{\alpha}{\alpha-1}\varepsilon W_1+\frac{1}{1-\alpha}(\varepsilon W_1)^\alpha\right)\right)^{n-1}\\
=&\left(1+(t-1)\frac{\alpha}{1-\alpha}(\varepsilon W_1)(1+o(1))\right)^{n-1}\xrightarrow[\varepsilon=\frac{\lambda}{n} \atop n\to \infty]{} e^{(t-1)\lambda\mathbb{E}[W]W_1};
\end{align*}
and for $\alpha=1$
\begin{align*}
\mathbb{E}_{W_1}[t^{a_{1j}}]^{n-1}=&\left(1+(t-1)\left(-W_1 \varepsilon \ln(\varepsilon W_1) +\varepsilon W_1 \right)\right)^{n-1}\\
=&\left(1+(t-1)(-W_1\varepsilon \ln(\varepsilon))(1+o(1))\right)^{n-1}\xrightarrow[-\varepsilon\ln(\varepsilon)=\frac{\lambda}{n} \atop n\to \infty]{} e^{(t-1)\lambda W_1}
\end{align*}
for a fixed value of $W_1$. 
Since $0 \le t^{D_n(1)} \le 1$ holds almost surely for any $t \in (0, 1)$, we can apply the Dominated Convergence Theorem to conclude that
\begin{align*}
\mathbb{E}[t^{D_n(1)}]\to 
\begin{cases}
\Eb[e^{(t-1)\frac{\lambda}{1-\alpha}W^\alpha}] 	&\alpha \in (0,1);\\
\Eb[e^{(t-1)\lambda W}] 					&\alpha=1;\\
\Eb[e^{(t-1)\lambda\Eb[W]W}] 			&\alpha >1,
\end{cases}
\end{align*}
where the right hand side is the probability generating function of a mixed Poisson variable with parameter 
\begin{equation*}
\Lambda=
\begin{cases}
\frac{\lambda}{1-\alpha}W^\alpha 	&\alpha \in(0,1);\\
\lambda W 					&\alpha =1;\\
\lambda \Eb[W]W 		&\alpha >1;
\end{cases}
\end{equation*}
with $W\sim Pareto(\alpha)$. 
\end{proof}

\subsection{Dust Regime Sufficient Condition}\label{isolated_vertices_with_small_weight} We formulate this result for a slightly more general situation than $\mathbb{G}_n(\alpha,\varepsilon)$, i.e. we choose independent Pareto weights but a general symmetric function $f$ as a connection probability. In this framework, let \(G = G_n(f, \{W_i\})\) be an inhomogeneous random graph on \(n\) vertices associated with i.i.d.\ weights \(W_1,\dots,W_n\), drawn from a Pareto($\alpha$) distribution.
Conditioning on $(W_i)_{i\in[n]}$ and given an increasing symmetric continuous connection function \(f:=f_n:\mathbb{R}_+^2 \to [0,1]\) with $f_n\to0$ as $n\to \infty$, edges are drawn independently with probabilities $p_{ij} = f(W_i, W_j)$, $1 \le i < j \le n$. We write as usual $\mathbb{E}_w[p_{ij}]$ to indicate $\mathbb{P}(i\leftrightarrow j | W_i=w)=\mathbb{E}[f(w,W_j)]$ for $w\geq 1$. In a similar way we indicate $\mathbb{E}_w[p_{ij}^2]=\mathbb{E}[f(w,W_j)^2]$ and so on. 

\begin{definition}\label{S}
Let
\begin{align*}\label{S_h}
\mathcal{S}_h:=\{i\in[n] : W_i \leq h(n)\}
\end{align*}
with $h(n)=1+g(n)$ where $g(n)\searrow0$ as $n\to \infty$ be the set of \textbf{small} vertices, with weight sufficiently close to the minimum $1$. Moreover, let
\begin{equation*}\label{I_h}
\mathcal{I}_h:=\sum_{i \in [n]}\ind(W_i\leq h(n))\ind(\text{i is isolated})=\sum_{i \in \mathcal{S}_h}\ind(\text{i is isolated})
\end{equation*}
be the number of vertices in $\mathcal{S}_h$ that are \textbf{isolated} in $G$.
\end{definition}

We have the following results.

\begin{prop}\label{smallverticesmoments3}
Let $h(n)=1+g(n)$ with $g(n)n\to\infty$. Let $G$, $\mathcal{S}_h$ and $\mathcal{I}_h$ as in Definition \ref{S}. For each $\alpha>0$ it holds:
\begin{equation*}
\frac{\mathbb{E}[\mathcal{I}_h^2]}{\mathbb{E}[\mathcal{I}_h]^2}\leq \frac{1}{\mathbb{E}[\mathcal{I}_h]}+(1+o(1))e^{\phi_n}
\end{equation*}
as $n\to \infty$, where
\begin{equation*}
\phi_n=:n\left(-2\mathbb{E}_1[p_{12}]+\mathbb{E}_1[p_{12}^2]+2\frac{\mathbb{E}_{h(n)}[p_{12}]}{1-\mathbb{E}_{h{(n)}}[p_{12}]}\right).
\end{equation*}
\end{prop}

\begin{proof}[Proof of Proposition \ref{smallverticesmoments3}]
The hypothesis $g(n)n\to \infty$ ensures then $S_h$ is non-empty w.h.p., in fact $|S_h| \sim \text{Bin}(n,p_n)$ with $p_n=1-\left(1+g(n)\right)^{-\alpha}\simeq \alpha g(n)$, thus
\begin{align*}
\notag&\mathbb{P}(|S_h|=0)=\left(1+g(n)\right)^{-\alpha n}
\to 0.
\end{align*}
Since $f$ is increasing,  when $1 \in S_h$ it holds $\mathbb{E}[p_{1j}]\leq \mathbb{E}_{h(n)}[p_{1j}]$. Thus, we have
\begin{align}\label{I}
\notag\mathbb{E}[\mathcal{I}_h]=&\mathbb{E}\left[\sum_{i \in [n]}\ind(W_i \leq h(n))\ind(\text{i is isolated})\right]
=n\mathbb{E}\left[\ind(W_1 \leq h(n))\prod_{j\geq 2} (1-p_{1j}) \right]\\
\geq & np_n\mathbb{E}[1-f(h(n),W_2)]^{n-1}=np_n\left(1-\mathbb{E}_{h(n)}[p_{12}] \right)^{n-1}.
\end{align}
Moreover,
\begin{align*}
\mathcal{I}_h^2=&\left(\sum_{i\in{[n]}}\ind(W_i\leq h(n))\ind(\text{i is isolated})\right)^2\\
=&\sum_{i\in[n]}\ind(W_i\leq h(n))^2\ind(\text{i is isolated})^2+\sum_{i,j\in[n]\atop i \neq j}\ind(W_i,W_j\leq h(n))\ind(\text{i and j are isolated})\\
=&\mathcal{I}_h+\sum_{i,j\in[n]\atop i \neq j}\ind(W_i,W_j\leq h(n))\ind(\text{i and j are isolated}).
\end{align*}
Using this time $\mathbb{E}[p_{ij}]\geq\mathbb{E}_1[p_{ij}]$ with $i=1$ and $2$ we obtain
\begin{align}\label{Nsecondmoment}
\notag\mathbb{E}[\mathcal{I}_h^2]\leq &\mathbb{E}[\mathcal{I}_h]+n(n-1)\mathbb{E}\left[\ind(W_1,W_2\leq h(n)) \prod_{j\geq 3} (1-p_{1j})(1-p_{2j}) \right]\\
\notag \leq &\mathbb{E}[\mathcal{I}_h]+n^2p^2_n\mathbb{E}\left[(1-f(1,W_3))^2\right]^{n-2}\\
=&\mathbb{E}[\mathcal{I}_h]+n^2p_n^2(1-2\mathbb{E}_1[p_{12}]+\mathbb{E}_1[p_{12}^2])^{n-2}.
\end{align}
  Using \eqref{I} and \eqref{Nsecondmoment} we have
\begin{align*}
\notag\frac{\mathbb{E}[\mathcal{I}_h^2]}{\mathbb{E}[\mathcal{I}_h]^2}\leq &\frac{1}{\mathbb{E}[\mathcal{I}_h]}+\frac{n^2p_n^2}{n^2p_n^2}\frac{(1-2\mathbb{E}_1[p_{12}]+\mathbb{E}_1[p_{12}^2])^{n-2}}{(1-\mathbb{E}_{h(n)}[p_{12}])^{2(n-1)}}\\
\notag=&\frac{1}{\mathbb{E}[\mathcal{I}_h]}+\frac{(1-\mathbb{E}_{h(n)}[p_{12}])^2}{(1-2\mathbb{E}_1[p_{12}]+\mathbb{E}_1[p_{12}^2])^2}\frac{(1-2\mathbb{E}_1[p_{12}]+\mathbb{E}_1[p_{12}^2])^{n}}{(1-\mathbb{E}_{h(n)}[p_{12}])^{2n}}\\
\notag=&\frac{1}{\mathbb{E}[\mathcal{I}_h]}+(1+o(1))\frac{(1-2\mathbb{E}_1[p_{12}]+\mathbb{E}_1[p_{12}^2])^{n}}{(1-\mathbb{E}_{h(n)}[p_{12}])^{2n}}\\
\leq&\frac{1}{\mathbb{E}[\mathcal{I}_h]}+(1+o(1))\exp\left\{n\left(-2\mathbb{E}_1[p_{12}]+\mathbb{E}_1[p_{12}^2]+2\frac{\mathbb{E}_{h(n)}[p_{12}]}{1-\mathbb{E}_{h(n)}[p_{12}]} \right)\right\}
\end{align*}
where we have used $(1-x)\leq \exp(-x)$ with $x=2\mathbb{E}_1[p_{12}]-\mathbb{E}_1[p_{12}^2]$ and $(1-x)\geq \exp(-x/(1-x))$ with $x=\mathbb{E}_{h(n)}[p_{12}]$ for $x$ small.
\end{proof}

\begin{lemma}[Dust Regime Sufficient Condition]\label{smallverticesmoment2}
    Let $h(n)=1+g(n)$ with $g(n)n\to\infty$. Let $G$, $\mathcal{S}_h$ and $\mathcal{I}_h$ as in Definition \ref{S}. Let $\alpha>0$. If it holds:    
    \begin{enumerate}
        \item $\mathbb{E}[\mathcal{I}_h]\to \infty\quad\text{ as }n\to \infty$;
        \item $\phi_n=:n\left(-2\mathbb{E}_1[p_{12}]+\mathbb{E}_1[p_{12}^2]+2\frac{\mathbb{E}_{h(n)}[p_{12}]}{1-\mathbb{E}_{h{(n)}}[p_{12}]}\right)\to0\quad \text{ as }n\to\infty$
    \end{enumerate}
    then $\mathbb{P}(\mathcal{I}_h\geq 1)\to 1$. 
\end{lemma}

\begin{proof}[Proof of Lemma \ref{smallverticesmoment2}]
  We want to count in $S_h$ the number $\mathcal{I}_h$ of vertices that are isolated in $\mathbb{G}_n(\alpha,\varepsilon)$ (see Definition \ref{S}) in order to show that $\mathbb{P}(\mathcal{I}_h\geq 1)\to1$. This implies that the graph $\mathbb{G}_n(\alpha,\varepsilon)$ has an isolated vertex of weight \textit{near} $1$ w.h.p., that obviously lead to have an isolated vertex in the graph $G$ w.h.p.
  To have $\mathbb{P}(\mathcal{I}_h\geq1)\to1$ we will use the second moment method, i.e.~we show that $\mathbb{E}[\mathcal{I}_h]^2/\mathbb{E}[\mathcal{I}_h^2]\to 1$ and use the Paley-Zygmund Inequality $\mathbb{P}(\mathcal{I}_h>0)\geq \mathbb{E}[\mathcal{I}_h]^2/\mathbb{E}[\mathcal{I}_h^2]$ to conclude. 
  In particular, we already know that $1\leq\mathbb{E}[\mathcal{I}_h^2]/\mathbb{E}[\mathcal{I}_h]^2$ holds, so it will be sufficient to show
\begin{equation}\label{dustregimeobiettivo}
\frac{\mathbb{E}[\mathcal{I}_h^2]}{\mathbb{E}[\mathcal{I}_h]^2}\leq f(n)
\end{equation}
where $f(n)$ is a suitable function satisfying $f(n)\to1$ as $n\to \infty$. In order to establish \eqref{dustregimeobiettivo} and conclude the proof, Proposition \ref{smallverticesmoments3} suggests that it is sufficient to ensure that $\mathbb{E}[\mathcal{I}_h] \to \infty$ and $\phi_n \to 0$ as $n \to \infty$.
\end{proof}

\begin{table}[p]
\centering
\fontsize{9.2pt}{13pt}\selectfont
\renewcommand{\arraystretch}{3.3}

\setlength{\tabcolsep}{6.4pt}
\begin{tabular}{|c|c|c|c|c|c|}
\hline
\cellcolor{color2}\textbf{Term}
&
\cellcolor{color2}\(\alpha\in(0,1)\)
&
\cellcolor{color2}\(\alpha=1\)
&
\cellcolor{color2}\(\alpha\in(1,2)\)
&
\cellcolor{color2}\(\alpha=2\)
&
\cellcolor{color2}\(\alpha>2\)
\\ \hline

\(A\)
&
\cellcolor{color1}
\(\displaystyle \frac{2\varepsilon^\alpha}{2-\alpha}
 +O(\varepsilon^2)\)
&
\cellcolor{color1}
\(2\varepsilon+O(\varepsilon^2)\)
&
\(O(\varepsilon^\alpha)\)
&
\(O(-\varepsilon^2\ln\varepsilon^2)\)
&
\(O(\varepsilon^2)\)
\\ \hline

\(B\)
&
\cellcolor{color1}
\(\displaystyle \frac{2\varepsilon^\alpha}{1-\alpha}
 +O(\varepsilon)\)
&
\cellcolor{color1}
\(-2\varepsilon\ln\varepsilon+2\varepsilon\)
&
\multicolumn{3}{c|}{\cellcolor{color1}
\(\displaystyle 2\mathbb E[W]\varepsilon+O(\varepsilon^\alpha)\)}
\\ \hline

\(C\)
&
\(\displaystyle \frac{2h(n)^\alpha\varepsilon^\alpha}{1-\alpha}
 +O(\varepsilon)\)
&
\(\displaystyle -2h(n)\varepsilon\ln\varepsilon
 +2h(n)\varepsilon\)
&
\multicolumn{3}{c|}{
\(\displaystyle 2h(n)\mathbb E[W]\varepsilon
 +O(\varepsilon^\alpha)\)}
\\ \hline

\(D\)
&
\cellcolor{color1}
\(\displaystyle \frac{2h(n)^\alpha\varepsilon^\alpha}{1-\alpha}
 +O(\varepsilon^{\min(1,2\alpha)})\)
&
\cellcolor{color1}
\(\begin{aligned}
 &-2h(n)\varepsilon\ln\varepsilon
   +2h(n)\varepsilon\\[-1mm]
 &\quad+O(\varepsilon^2\ln\varepsilon)
\end{aligned}\)
&
\multicolumn{3}{c|}{\cellcolor{color1}
\(\displaystyle 2h(n)\mathbb E[W]\varepsilon
 +O(\varepsilon^{\min(\alpha,2)})\)}
\\ \hline

\(\operatorname{Main}(\phi_n)\)
&
\(\displaystyle
 \frac{2\varepsilon^\alpha n}{1-\alpha}
 \left(h(n)^\alpha-1+\frac{1-\alpha}{2-\alpha}\right)\)
&
\(\displaystyle
 2h(n)\varepsilon n
 -\frac{2\varepsilon\ln\varepsilon\,n}{\ln n}\)
&
\multicolumn{3}{c|}{
\(\displaystyle \frac{2\mathbb E[W]\varepsilon n}{\ln n}\)}
\\ \hline

\(\operatorname{Res}(\phi_n)\)
&
\(O(n\varepsilon^{\min(1,2\alpha)})\)
&
\(O(n\varepsilon^2\ln\varepsilon)\)
&
\(O(n\varepsilon^\alpha)\)
&
\(O(-n\varepsilon^2\ln\varepsilon)\)
&
\(O(n\varepsilon^2)\)
\\ \hline
\end{tabular}

\caption{
\textbf{Asymptotic behavior of the terms composing $\phi_n$}. 
}
\vspace{0.1cm}
\parbox{\linewidth}{We compute, for each value of $\alpha$ the main term of $A=\mathbb{E}_1[p_{12}^2]$, $B=2\mathbb{E}_1[p_{12}]$ and $D=2\mathbb{E}_{h(n)}[p_{12}]/(1-\mathbb{E}_{h(n)}[p_{12}])$ according \eqref{A}, with the aim of finding the main term of $\phi_n=n(D-B+A)$. The term $D$ is derived from $C=2\mathbb{E}_{h(n)}[p_{12}]$ by noticing that $(1-\mathbb{E}_{h(n)}[p_{12}])^{-1}=1+\mathbb{E}_{h(n)}[p_{12}]+O(\mathbb{E}_{h(n)}[p_{12}])$ only affects the residual part. The terms contributing to the dominant behavior of $\phi_n$ are highlighted for each regime. Since $h(n)=1+\ln(n)^{-1}$, it holds
\begin{equation*}
\phi_n\simeq 
\begin{cases}
\frac{2}{2-\alpha}\varepsilon^\alpha n	&\alpha \in (0,1),\\
2\varepsilon n 	&\alpha=1,\\
2\mathbb{E}[W] \frac{\varepsilon n}{\ln(n)} 	&\alpha>1.
\end{cases}
\end{equation*}
}

\label{table:asymptotics_phi_detailed}
\end{table}

\section*{Acknowledgements}
The authors would like to express their gratitude to the International Centre for Theoretical Sciences (ICTS), Bangalore, India, and the organizers of the program ``Challenges in Networks'' for their hospitality and financial support, during which part of this work was carried out.

\section*{AI statement}
LLMs were recently used to review a first draft of the paper. They in particular suggested a shorter proof using Hoeffding's inequality of Theorem \ref{thm:maxdegree} which has been incorporated into the current version.

\printbibliography
\end{document}